\documentclass[12pt, twoside]{article}
\usepackage{enumitem}

\usepackage{tocloft}
\renewcommand{\cftsecleader}{\cftdotfill{\cftdotsep}}
\usepackage{titletoc}
\usepackage{pdfpages}
\usepackage{amsmath,amssymb}
\usepackage{amsthm}
\usepackage{leftidx}
\usepackage{epsfig}
\usepackage{graphics}
\usepackage{graphicx}
\usepackage{colordvi}
\usepackage{mathrsfs} 
\usepackage{stmaryrd}
\usepackage{geometry}
\usepackage{dsfont}
\usepackage{tikz}
\usetikzlibrary{automata}
\usetikzlibrary{snakes}
\usetikzlibrary{decorations.pathmorphing}
\usepackage{url}
\usepackage[hidelinks]
{hyperref}
\hypersetup{
    colorlinks=false,
   linkcolor=blue!75!black,
    filecolor=blue!75!black,      
   urlcolor=blue!75!black, 
   citecolor=blue!75!black,
}
\usepackage{enumitem}
\usepackage[textsize=small]{todonotes}
\usepackage[margin=1cm, size=small]{caption}
\usepackage{multirow}
\allowdisplaybreaks[4]
\oddsidemargin
-.25in \evensidemargin -.25in\marginparwidth 0.4in
\usepackage{color,soul}
\usepackage{xcolor}

\usepackage{etoolbox}
\usepackage{fancyvrb}
\usepackage{fvextra}
\usepackage{fancyhdr}
\usepackage{todonotes}

\numberwithin{figure}{section}

\usepackage{empheq}
\usetikzlibrary{automata}
\definecolor{blue2}{cmyk}{.94,.11,0,0}
\usepackage{enumitem}
\setitemize{itemsep=0pt}
\definecolor{myblue}{rgb}{.8, .8, 1}
\newlength\mytemplen
\newsavebox\mytempbox
\usepackage{verbatim}

\DeclareFontFamily{U}{mathx}{}
\DeclareFontShape{U}{mathx}{m}{n}{<-> mathx10}{}
\DeclareSymbolFont{mathx}{U}{mathx}{m}{n}
\DeclareMathAccent{\widecheck}{0}{mathx}{"71}
\DeclareMathAlphabet\mathbfcal{OMS}{cmsy}{b}{n}

\makeatletter
\renewenvironment{thebibliography}[1]
     {\section*{\refname}
      \@mkboth{\MakeUppercase\refname}{\MakeUppercase\refname}
      \begin{enumerate}[label={[\arabic{enumi}]},itemindent=*,leftmargin=2.5em]
      \@openbib@code
      \sloppy
      \clubpenalty4000
      \@clubpenalty \clubpenalty
      \widowpenalty4000
      \sfcode`\.\@m}
     {\def\@noitemerr
       {\@latex@warning{Empty `thebibliography' environment}}
      \end{enumerate}}
\makeatother

\makeatletter
\newcommand{\newparallel}{\mathrel{\mathpalette\new@parallel\relax}}
\newcommand{\new@parallel}[2]{
  \begingroup
  \sbox\z@{$#1T$}
  \resizebox{!}{\ht\z@}{\raisebox{\depth}{$\m@th#1/\mkern-4.5mu/$}}
  \endgroup
}
\makeatother

\makeatletter
\newcommand{\nnewparallel}{\mathrel{\mathpalette\nnew@parallel\relax}}
\newcommand{\nnew@parallel}[2]{
  \begingroup
  \sbox\z@{$#1T$}
  \resizebox{!}{\ht\z@}{\raisebox{\depth}{$\m@th#1/\mkern-4.5mu/\!\!\!\!\backslash $}}
  \endgroup
}
\makeatother

\makeatletter
\newcommand\cb{
    \@ifnextchar[
       {\@cb}
       {\@cb[5pt]}}

\def\@cb[#1]{
    \@ifnextchar[
       {\@@cb[#1]}
       {\@@cb[#1][5pt]}}

\def\@@cb[#1][#2]#3{
    \sbox\mytempbox{#3}
    \mytemplen\ht\mytempbox
    \advance\mytemplen #1\relax
    \ht\mytempbox\mytemplen
    \mytemplen\dp\mytempbox
    \advance\mytemplen #2\relax
    \dp\mytempbox\mytemplen
    \colorbox{myblue}{\hspace{1em}\usebox{\mytempbox}\hspace{1em}}}
\makeatother
\DeclareMathOperator*{\esssup}{ess\,sup}
\patchcmd{\thebibliography}{\section*}{\section}{}{}

\renewcommand{\Re}{{\rm Re}}
\newcommand{\BES}{{\rm BES}}

\renewcommand{\Im}{{\rm Im}}

\newcommand{\less}{\lesssim}
\newcommand{\more}{\gtrsim}
\newcommand{\e}{{\mathrm e}}
\newcommand{\1}{\mathds 1}

\newcommand{\C}{{\mathscr C}}
\newcommand{\F}{\mathscr F}
\newcommand{\B}{\mathscr B}

\newcommand{\vep}{{\varepsilon}}

\newcommand{\da}{{\downarrow}}
\newcommand{\ua}{{\uparrow}}

\newcommand{\la}{{\langle}}
\newcommand{\ra}{{\rangle}}

\newcommand{\ind}{{\perp\!\!\!\perp}}

\newcommand{\bs}{\boldsymbol}
\newcommand{\ms}{\mathscr}

\renewcommand{\P}{{\mathbb P}}
\newcommand{\E}{{\mathbb E}}

\newcommand{\mc}{\mathcal}
\renewcommand{\d}{{\mathrm d}}

\newcommand{\al}{(-\alpha)}

\newcommand{\albe}{(-\alpha),\beta \da }

\newcommand{\zbe}{(0),\beta \da }

\newcommand{\R}{{\Bbb R}}

\renewcommand{\i}{{\mathtt i}}
\newcommand{\defeq}{{\stackrel{\rm def}{=}}}

\newcommand{\EN}{{\mathcal E_N}}
\renewenvironment{proof}[1][\proofname]{\noindent {\bfseries #1.}\;}{\hfill\ensuremath{\blacksquare}\\}

\newcommand{\cc}{{^\circ}}

\newcommand{\eqspace}{{\quad\;}}

\newcommand{\bi}{\mathbf i}
\newcommand{\bj}{\mathbf j}
\newcommand{\bk}{\mathbf k}
\newcommand{\bl}{\mathbf l}

\newcommand{\J}{{\mathcal J}}
\newcommand{\hK}{{\widehat{K}}}
\newcommand{\two}{{\sqrt{2}}}

\newcommand{\bbeta}{{\bs \beta}}

\newcommand{\K}{{\widehat{ K}}}

\newcommand{\loc}{{\rm loc}}
\newcommand{\I}{{\rm I}}
\newcommand{\II}{{\rm II}}
\newcommand{\III}{{\rm III}}

\renewcommand{\sp}{{\shortparallel}}
\newcommand{\nsp}{{\nshortparallel}}
\newcommand{\CN}{{\mathbb C}^{\it N}}

\newcommand{\Twi}
{T^{\bs w\setminus\{w_\bi\}}}

\newcommand{\TwJm}
{T^{\bs w\setminus\{w_{\bs J_m}\}}}
\newcommand{\CNw}{{\CN_{\bs w\,\sp}}}
\newcommand{\CNwi}{{\CN_{\bi\,\nsp,\,\bs w\setminus\{w_\bi\}\,\sp}}}

\newcommand{\CNwJm}{{\CN_{\bs J_m\,\nsp,\,\bs w\setminus\{w_{\bs J_m}\}\,\sp}}}

\newcommand{\CNwni}{{\CN_{\bs w\setminus\{w_\bi\}\,\sp}}}

\usepackage{currfile}
\usepackage{fancyhdr}
\newcommand{\wt}{\widetilde}

\newtheoremstyle{slantthm}{10pt}{10pt}{\slshape}{}{\bfseries}{}{.5em}{\thmname{#1}\thmnumber{ #2}\thmnote{ (#3)}.}
\newtheoremstyle{slantrmk}{10pt}{10pt}{\rmfamily}{}{\bfseries}{}{.5em}{\thmname{#1}\thmnumber{ #2}\thmnote{ (#3)}.}

\begin{document}
\theoremstyle{slantthm}
\newtheorem{thm}{Theorem}[section]
\newtheorem{prop}[thm]{Proposition}
\newtheorem{lem}[thm]{Lemma}
\newtheorem{cor}[thm]{Corollary}
\newtheorem{defi}[thm]{Definition}
\newtheorem{claim}[thm]{Claim}
\newtheorem{disc}[thm]{Discussion}
\newtheorem{conj}[thm]{Conjecture}

\theoremstyle{slantrmk}
\newtheorem{ass}[thm]{Assumption}
\newtheorem{rmk}[thm]{Remark}
\newtheorem{eg}[thm]{Example}
\newtheorem{question}[thm]{Question}
\numberwithin{equation}{section}
\newtheorem{quest}[thm]{Quest}
\newtheorem{problem}[thm]{Problem}
\newtheorem{discussion}[thm]{Discussion}
\newtheorem{notation}[thm]{Notation}
\newtheorem{observation}[thm]{Observation}

\definecolor{db}{RGB}{13,60,150}
\definecolor{dg}{RGB}{150,40,40}

\newcommand{\thetitle}{
\vspace{-.5cm}
Stochastic motions of the two-dimensional many-body delta-Bose gas, II:
Many-$\bs \delta$ motions}

\title{\bf \thetitle\footnote{Support from an NSERC Discovery grant is gratefully acknowledged.}}

\author{Yu-Ting Chen\,\footnote{Department of Mathematics and Statistics, University of Victoria, British Columbia, Canada.}\,\,\footnote{Email: \url{chenyuting@uvic.ca}}}

\date{\today \vspace{-.5cm}}

\maketitle
\abstract{This paper is the second in a series devoted to constructing stochastic motions for the two-dimensional $N$-body delta-Bose gas for all integers $N\geq 3$ and establishing the associated Feynman--Kac-type formulas.  The main results here construct and study the more general \emph{stochastic many-$\delta$ motions} for $N$ particles. They have the interpretation of independent two-dimensional Brownian motions conditioned to attain the contact interactions that realize multiple two-body $\delta$-function potentials. For the construction, we transform the stochastic one-$\delta$ motions studied in \cite{C:SDBG1-2} by Girsanov's theorem \emph{locally before} a pair of particles with different initial conditions begins to contact each other. The strong Markov processes with lifetime thus obtained are concatenated by using the ``no-triple-contacts'' (NTC). This NTC phenomenon appears in the functional integral solutions of the two-dimensional many-body delta-Bose gas obtained earlier and is now proven at the pathwise level to a generalized degree.\medskip

\smallskip

\vspace{0cm}

\setlength{\cftbeforesecskip}{0pt}
\setlength\cftaftertoctitleskip{0pt}
\renewcommand{\cftsecleader}{\cftdotfill{\cftdotsep}}
\setcounter{tocdepth}{2}
\tableofcontents

\section{Introduction}\label{sec:intro}
This paper is the second in a series devoted to constructing stochastic motions for the two-dimensional $N$-body delta-Bose gas for all integers $N\geq 3$ and establishing the associated Feynman--Kac-type formulas. Our goal here is to construct and study the stochastic motions as diffusion processes, using the foundation from \cite{C:SDBG1-2} for the stochastic one-$\delta$ motions subject to probability measures denoted by $\P^{\beta_\bi\da,\bi}$. In Section~\ref{sec:selection}, we will discuss a very different possible alternative and explain our motivation for the formulation of the main results.

Given an integer $N\geq 3$, a set of weights $\bs w$, and a set of coupling constants $\bbeta$, the main results of this paper construct a $\Bbb C^N$-valued strong Markov process $\{\ms Z_t\}=\{Z^j_t\}_{1\leq j\leq N}$, called a {\bf stochastic many-$\bs \delta$ motion}, up to a terminal time $T_\partial>0$: $\ms Z_t=\partial$ for $t\in [T_\partial,\infty)$ on $\{T_\partial<\infty\}$, where $\partial$ is an extra point added to $\Bbb C^N$; see Remark~\ref{rmk:main1} (1$\cc$) for more details of this terminal time. Here, with 
\[
\mathcal E_N\,\defeq\,\{\bj=(j\prime,j)\in \Bbb N^2;1\leq j<j\prime\leq N\},
\]
$\bbeta=\{\beta_\bj\}_{\bj\in \mc E_N }\in (0,\infty)^{\mathcal E_N}$ and we require that $\bs w=\{w_\bj\}_{\bj\in \mc E_N}\in \R_+^{\mathcal E_N}$ satisfy $\#\{\bj;w_\bj>0\}\geq 2$.

The above stochastic many-$\delta$ motion satisfies several properties at the pathwise and expectation levels. In particular, it solves the following SDE: for all $1\leq j\leq N$, 
\begin{align}\label{SDE:Z10}
\begin{split}
Z^j_t&=z^j_0-\sum_{1\leq k\leq N\atop k\neq j}\int_0^t \frac{w_{j\vee\!\wedge k}\sqrt{\beta_{j\vee\!\wedge k}} K_1(\sqrt{\beta_{j\vee\!\wedge k}}|Z^{j}_s-Z^k_s|)(Z_s^{j}-Z^k_s)}{\sum_{\ell_2>\ell_1\atop N\geq \ell_2}w_{\ell_1\!\vee\!\wedge \ell_2}K_0(\sqrt{\beta_{\ell_2\!\vee\!\wedge \ell_1}}|Z^{\ell_2}_s-Z^{\ell_1}_s|)
|Z_s^{j}-Z^k_s|}\d s+\widetilde{W}^j_t,\; t<T_\partial,
\end{split}
\end{align}
which is formulated with the following ingredients besides $\bs w,\bbeta$ specified above:
\begin{itemize}
\item the set of eligible initial conditions $z_0=(z_0^1,\cdots,z_0^N)\in \Bbb C^N$ is such that 
\begin{align}\label{def:ic}
\#\{\bj=(j\prime,j)\in \mc E_N;z_0^{j\prime}-z_0^j=0\;\&\;w_\bj>0\}\leq 1,
\end{align}
\item ${\ell_1\!\vee\!\!\wedge \ell_2}\,\defeq\,(\max\{\ell_1,\ell_2\},\min\{\ell_1,\ell_2\})\in \mc E_N$ for integers $\ell_1,\ell_2\geq 1$ with $\ell_1\neq \ell_2$,
\item $K_\nu(\cdot)$ is the Macdonald function of index $\nu$, and
\item $\wt{\ms W}_t=\{\widetilde{W}^j_t\}_{1\leq j\leq N}$ is a family of independent two-dimensional standard Brownian motions with zero initial condition. 
\end{itemize}
In contrast, the SDE of the stochastic one-$\delta$ motion under $\P^{\beta_\bi\da,\bi}$ satisfies \eqref{SDE:Z10} with $w_\bj\equiv \1_{\bi}(\bj)$ and $T_\partial\equiv \infty$; see \eqref{def:ZSDE2-2} for a restatement. 
In any case, the SDE in \eqref{SDE:Z10} can be further identified as an SDE with singular drift \cite{KrRo-2, Ver-2} such that the drift  is supercritical under a Ladyzhenskaya--Prodi--Serrin-type condition:  $b(x,t)\in L^q_{\loc}(\R_+;L^p_{\loc}(\R^d))$ for $d/p+2/q>1$. Specifically, this supercriticality holds with $(d,p,q)=(2N,2,\infty)$ since \eqref{SDE:Z10} takes the form of 
\begin{align}\label{def:Zintro}
\ms Z_t=\ms Z_0+\int_0^t \bs b_\beta(\ms Z_s)\d s+\wt{\ms W}_t,\quad t<T_\partial,
\end{align}
with
$\bs b_\beta\in L^p_{\loc }(\Bbb C^N)$ if and only if $p\leq 2$; see \cite[Corollary~2.14]{C:BES-2} for $N=2$, whose proof extends to the present case of $N\geq 3$. More properties of the stochastic many-$\delta$ motions appear in
Theorem~\ref{thm:main1} and Propositions~\ref{prop:SDE1} and~\ref{prop:SDE2}. Note that the Feynman--Kac-type formulas in \cite{C:SDBG3-2} use the stochastic many-$\delta$ motions with $w_\bj>0$ for all $\bj\in \mc E_N$, while the extensions to general $\bs w$ are straightforward to our methods and have independent meanings explained below. 

The stochastic many-$\delta$ motions introduced above give a stochastic description of the two-dimensional many-body delta-Bose to a generalized degree. The corresponding characteristics are the nontrivial contacts realized in the forms of $Z^{j\prime}_t=Z^j_t$ for $\bj=(j\prime,j)\in \mc E_N$ with $w_\bj>0$ and the {\bf no-triple-contacts (NTC)} that the probability of $Z^{j_1}_t=Z^{j_2}_t=Z^{j_3}_t$ for some $0\leq t<T_\partial$ and distinct integers $j_1,j_2,j_3$ is zero. On the other hand, such contact interactions and the NTC are described in the original quantum Hamiltonian of the two-dimensional many-body delta-Bose; see \eqref{def:HNm} with $\mathcal J=\mathcal E_N$ for the Hamiltonian. They also find realizations in the earlier functional integral solutions, but at the operator and expectation levels (e.g. \cite{C:DBG-2}).

\begin{figure}[t]
\begin{center}
 \begin{tikzpicture}[scale=.8]
    \draw[gray, thin] (0,0) -- (7.9,0);
    \foreach \i in {0.0, 0.5,  2.0,  3.2,  5,  6.7, 7.9} {\draw [gray] (\i,-.05) -- (\i,.05);}
    \draw (0.0,-0.14) node[below]{$0$};
    \draw (0.55,-0.07) node[below]{$T_0^1$};
    \draw (2.05,-0.07) node[below]{$T_0^2$};
    \draw (3.25,-0.07) node[below]{$T_0^3$};
    \draw (5.05,-0.125) node[below]{$T_0^n$};
    \draw (7,-0.04) node[below]{$T_0^{n+1}$};
    \draw (7.9,-0.125) node[below]{$T_\partial$};
    \node at (0.0,0.40) {$\bullet$};
    \node at (0.0,1.20) {$\bullet$};
    \node at (0.0,1.80) {$\bullet$};
    \node at (0.0,2.30) {$\bullet$};
    \draw (0.0,0.40) node [left] {$Z^1_0$};
    \draw (0.0,1.20) node [left] {$Z^2_0$};
    \draw (0.0,1.80) node [left] {$Z^3_0$};
    \draw (0.0,2.30) node [left] {$Z^4_0$};    
    \draw [thick, color=black] (0.0,0.40) -- (0.5,0.6);
    \draw [thick, color=black] (0.0,1.20) -- (0.5,0.6);
    \draw [thick, color=black] (0.0,1.80) -- (0.5,1.15);    
    \draw [thick, color=black] (0.0,2.30) -- (0.5,2.20); 
    \node at (0.5,0.6) {$\bullet$};
    \node at (0.5,0.6) {$\bullet$};
    \node at (0.5,1.15) {$\bullet$};
    \node at (0.5,2.20) {$\bullet$};
    \draw [thick, color=black, snake=coil, segment length=4pt] (0.5,0.60) -- (1.2,0.80);
    \draw [thick, color=black] (0.5,1.15) -- (1.2,1.40);    
    \draw [thick, color=black] (0.5,2.20) -- (1.2,2.00);    
    \node at (1.2,0.8) {$\bullet$};
    \node at (1.2,0.8) {$\bullet$};
    \node at (1.2,1.40) {$\bullet$};
    \node at (1.2,2.00) {$\bullet$};
    \draw [thick, color=black] (1.2,0.8) -- (2.0,0.15);
    \draw [thick, color=black] (1.2,0.8) -- (2.0,1.45);
    \draw [thick, color=black] (1.2,1.40) -- (2.0,2.20);    
    \draw [thick, color=black] (1.2,2.00) -- (2.0,2.20);  
    \node at (2.0,0.15) {$\bullet$};
    \node at (2.0,1.45) {$\bullet$};
    \node at (2.0,1.45) {$\bullet$};
    \node at (2.0,2.20) {$\bullet$};
    \draw [thick, color=black] (2.0,0.15) -- (2.5,0.45);
    \draw [thick, color=black] (2.0,1.45) -- (2.5,1.2);
    \draw [thick, color=black, snake=coil, segment length=4pt] (2.0,2.20) -- (2.52,1.85);    
    \node at (2.5,0.45) {$\bullet$};
    \node at (2.5,1.2) {$\bullet$};
    \node at (2.5,1.2) {$\bullet$};
    \node at (2.5,1.80) {$\bullet$};
    \draw [thick, color=black] (2.5,0.45)--(3.2,0.85);  
    \draw [thick, color=black] (2.5,1.80)--(3.2,1.35);    
    \draw [thick, color=black] (2.5,1.20)--(3.2,1.35);  
    \draw [thick, color=black] (2.5,1.80)--(3.2,2.10);        
    \node at (3.2,0.85) {$\bullet$};
    \node at (3.2,0.85) {$\bullet$};
    \node at (3.2,1.35) {$\bullet$};
    \node at (3.2,2.10) {$\bullet$};
    \draw [thick, color=black, snake=coil, segment length=4pt] (3.2,1.30)--(3.68,1.05) ;
    \draw [thick, color=black] (3.2,0.85)--(3.7,0.5);    
    \draw [thick, color=black] (3.2,2.10)--(3.7,2.02);  
    \node at (3.7,0.5) {$\bullet$};
    \node at (3.7,1.05) {$\bullet$};
    \node at (3.7,2.02) {$\bullet$};
    \draw [thick, color=black] (3.7,2.02)--(4,1.5);    
    \draw [thick, color=black] (3.7,1.05)--(4,1);    
    \draw [thick, color=black] (3.7,1.05)--(4,0.5);    
    \draw [thick, color=black] (3.7,0.5)--(4,0.5);    
    \node at (4,0.5) {$\bullet$};
    \node at (4,1) {$\bullet$};
    \node at (4,1.5) {$\bullet$};
    \draw [ultra thick,loosely dotted, black] (4.3,2.1) -- (4.7,2.1);
    \draw [ultra thick,loosely dotted, black] (4.3,1.3) -- (4.7,1.3);
    \draw [ultra thick,loosely dotted, black] (4.3,.5) -- (4.7,.5);
    \node at (5,1.25) {$\bullet$};
    \node at (5,1.75) {$\bullet$};
    \node at (5,2.30) {$\bullet$};
    \node at (5,2.30) {$\bullet$};
    \draw [thick, color=black] (5,1.25)--(5.8,0.45);
    \draw [thick, color=black] (5,1.75)--(5.8,1.20);   
    \draw [thick, color=black, snake=coil, segment length=4pt] (5,2.30)--(5.8,1.80); 
    \node at (5.8,0.45) {$\bullet$};
    \node at (5.8,1.20) {$\bullet$};
    \node at (5.8,1.8) {$\bullet$};
    \node at (5.8,1.8) {$\bullet$};
    \draw [thick, color=black] (5.8,0.45)--(6.7,0.35);
    \draw [thick, color=black] (5.8,1.20)--(6.7,0.35); 
    \draw [thick, color=black] (5.8,1.8)--(6.7,1.15);    
    \draw [thick, color=black] (5.8,1.8)--(6.7,2.40); 
    \node at (6.7,0.35) {$\bullet$};
    \node at (6.7,0.35) {$\bullet$};
    \node at (6.7,1.15) {$\bullet$};
    \node at (6.7,2.40) {$\bullet$};
    \draw [ultra thick,loosely dotted, black] (7.1,.5) -- (7.8,.5);
    \draw [ultra thick,loosely dotted, black] (7.1,1.3) -- (7.8,1.3);
    \draw [ultra thick,loosely dotted, black] (7.1,2.1) -- (7.8,2.1);
\end{tikzpicture}
\vspace{-.5cm}
\end{center} 
\caption{The figure is a schematic diagram for the stochastic many-$\delta$ motions $\{\ms Z_t\}=\{Z^j_t\}_{1\leq j\leq 4}$ with $Z^{j\prime}_0\neq Z^j_0$ and $w_\bj>0$ for all $\bj=(j\prime,j)\in \mc E_4$. A straight line segment represents part of the path of a particle when it has no contact with other particles except possibly at the endpoints.  A coiled line segment simultaneously represents the maximal parts of the paths of two particles undergoing contact interactions. The times $T_0^n$ are the contact-creation times. 
}
\label{fig:2}
\end{figure}
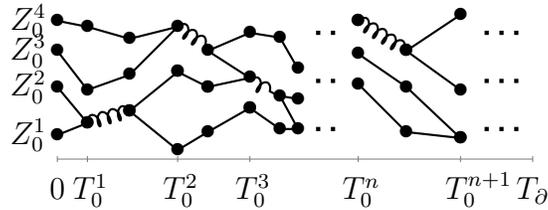

We will prove the existence of the stochastic many-$\delta$ motions by constructing and concatenating strong Markov processes with lifetime of the following two classes:
\begin{itemize}
\item The first class only deals with the initial conditions without contacts (Section~\ref{sec:SDE1}). The construction
allows some exact formulas at the expectation level since the law of a strong Markov process in this class is defined explicitly as a ``weighted average'' of the laws of the stochastic one-$\delta$ motions; see Definition~\ref{def:SDE1ult}. The weights use $\bs w$, and the process stops at the lifetime defined as the first time that contact occurs. In particular, the formulation is a multi-dimensional analogue of the Esscher transform in Donati-Martin and Yor's motivation for their method of constructing $\BES(0,\beta\da)$ \cite[(3.1)]{DY:Krein-2}. Specifically, the context in \cite{DY:Krein-2} is on transforming $2(1-\alpha)$-dimensional Bessel processes to $\BES(-\alpha,\beta\da)$, and our context is on transforming stochastic one-$\delta$ motions to stochastic many-$\delta$ motions.  See Section~\ref{sec:selection} for the role of $\BES(-\alpha,\beta\da)$ in the present problem.   
\item The second class mainly handles the initial conditions with contacts and subject to NTC (Section~\ref{sec:SDE2}). This class also extends the processes from the first class [Remark~\ref{rmk:SDE2} (1$\cc$)], but without pursuing extensions of those exact formulas. Now, the process is expressible as a stochastic one-$\delta$ motion via a Girsanov transformation locally before a lifetime for using an exponential local martingale which is not a martingale. In this case, the lifetime refers to the first time a pair of particles with different initial conditions begins to contact. Moreover, we close the local transformation by proving the continuous extensions of paths to the lifetime. The issue of this closure arises since the situation concerns constructing exit measures of supermartingales \cite{Follmer:exit-2,IW:Mult-2}. 
\end{itemize}
 In any case, under each of the strong Markov processes thus defined, the NTC is proven to hold up to and including the lifetime. This way, by starting with processes of the first class whenever initial conditions of no contacts are imposed or with processes of the second class otherwise, the construction can continue with processes of the second class and restart inductively the same way for the concatenation. The lifetimes mentioned above are superposed after the concatenation as {\bf contact-creation times} [\eqref{def:Tm1} and \eqref{def:Tm2}], and the terminal time $T_\partial$  is set as the increasing limit of the contact-creation times. See Figure~\ref{fig:2}. On the other hand, this construction motivates us to call our solutions to \eqref{SDE:Z10} stochastic many-$\delta$ motions: For $\mc J\subset\mc E_N$ with $\#\mathcal J\geq 2$ and $\bs w$ defined by $w_\bj\equiv \1_\mathcal J(\bj)$, the change of measures uses the stochastic one-$\delta$ motions under $\P^{\beta_\bj\da,\bj}$ for all $\bj\in \mathcal J$, and so, suggests that the associated stochastic many-$\delta$ motion is a probabilistic counterpart of the following formal operator:
\begin{align}\label{def:HNm}
 -\frac{1}{2}\sum_{j=1}^N\Delta_{z^j}-\sum_{\bj=(j\prime,j)\in \mc J}\Lambda_\bj \delta(z^{j\prime}-z^j),\quad (z^1,\cdots,z^N)\in \Bbb C^N.
\end{align}
Note that such an interpretation already exists in the stochastic one-$\delta$ motion under $\P^{\beta_\bi\da,\bi}$ for giving a probabilistic counterpart of \eqref{def:HNm} with $\mathcal J=\{\bi\}$ \cite[Section~1]{C:SDBG1-2}. 
Moreover, for given $N$ and $\bbeta$, all of such stochastic many-$\delta$ motions with $\{0,1\}$-valued weights and the stochastic one-$\delta$ motions relate to one another by the construction using change of measures. These processes obey a probabilistic counterpart of the BBGKY hierarchy in the kinetic theory of gases \cite[Section~3.3]{Kardar:SPP-3}. The extension to general $\bs w\in \R_+^{\mathcal E_N}$ has the interpretation of interpolations.

The proofs of this paper primarily deal with stochastic analytic problems complicated ubiquitously by the singularities of the Macdonald functions $K_0(\cdot)$ and $K_1(\cdot)$ as in the drift coefficients in \eqref{SDE:Z10}. The multi-dimensional setting is certainly responsible for the complexities, too. Roughly speaking, those problems emerge in the following two stages:
\begin{itemize}
\item The first stage is an intrinsic consequence of the fact that some singularities must be incorporated in any Radon--Nikod\'ym derivative process to create the contact interaction of a new pair of particles from a stochastic one-$\delta$ motion by changing measures (Proposition~\ref{prop:BMntc}). Specifically, the singularities now arise as a key issue in proving the semimartingale decompositions of the logarithms of the following non-martingale Radon--Nikod\'ym derivative processes we start with (Proposition~\ref{prop:logsum}) under $\P^{\beta_\bi\da,\bi}$ for $w_\bi>0$: 
\begin{align}\label{RNchoice}
\left.\left(\frac{K^{\bbeta,\bs w}_0 (t)}{ w_\bi K^{\bbeta,\bi}_0 (t)}\right)\right/ \left(\frac{K^{\bbeta,\bs w}_0 (0)}{ w_\bi K^{\bbeta,\bi}_0 (0)}\right),
\end{align}
where
\begin{align}
K_0^{\bbeta,\bj}(s)\,\defeq\, K_0(\sqrt{\beta_\bj}|Z^{j\prime}_s-Z^j_s|), \quad 
K^{\bbeta,\bs w}_0(s)\,\defeq\,\sum_{\bj\in \mc E_N}w_\bj K_0^{\bbeta,\bj}(s);\label{def:Kbbetaw0}
\end{align}
see the discussion of \eqref{def:Psi} for the choice of \eqref{RNchoice}. It\^{o}'s formula is applied to obtain the semimartingale decompositions, but only at an approximate level (Section~\ref{sec:logsum-1}). Moreover, the corresponding first-order finite-variation and It\^{o} correction terms are decomposed such that the finer terms are recombined appropriately to induce meaningful, but seemingly nonobvious, limits after the approximation is removed (Sections~\ref{sec:logsumA}--\ref{sec:logsumC}). For example, new singular integrals against local times are among the limiting terms. Above all, the stochastic many-$\delta$ motions have the interpretation that due to the SDE \eqref{SDE:Z10}, they mix the Laplacian with the $\delta$-function potentials for \eqref{def:HNm}. Therefore, singularities in relevant functions may have strong enough relations with the noise to make straightforward applications of It\^{o}'s formula infeasible. On the other hand, the semimartingale decompositions 
are used to identify the necessary normalization to induce probability measures, as we start with non-martingale Radon--Nikod\'ym derivative processes (Corollary~\ref{cor:superMG}). They also determine the SDEs of the strong Markov processes for concatenation and will serve as the main tool for proving the Feynman--Kac-type formulas in \cite{C:SDBG3-2} of this series. 

Note that the above characteristics find analogues in dealing with the one-dimensional diffusions $\BES(-\alpha,\beta\da)$ by which our methods for $N=2$ are built; see Section~\ref{sec:selection} for $\BES(-\alpha,\beta\da)$. Specifically, the difficulty of directly applying It\^{o}'s formula appears in proving the semimartingale decomposition of the logarithm of the Radon--Nikod\'ym derivative process for $\BES(-\alpha,\beta\da)$ with respect to $\BES(-\alpha)$ when $\alpha>0$ \cite[Section~4.1]{C:BES-2}.

\item The second stage concerns two problems from constructing the two classes of strong Markov processes with lifetime discussed above: (1) the continuous extensions of paths up to the lifetimes under the second class; (2) the validity of the NTC. In this stage, The SDEs of the {\bf stochastic relative motions}  $\{Z^\bj_t\}$ then take on their vital roles, where 
\begin{align}\label{def:Zbj2}
Z^\bj_t\;\defeq\,(Z^{j\prime}_t-Z^j_t)/\two, \quad \bj=(j\prime,j)\in \mc E_N.
\end{align}
In particular, we will identify in Theorem~\ref{thm:NTC} more general, sharp conditions for validating the NTC. The proof of this theorem extends a version of the comparison theorem of SDEs due to Ichiba and Karatzas~\cite{IK:BM-2} mainly to the extent of using many particles.  
\end{itemize}

\noindent {\bf Remainder of this paper.} Section~\ref{sec:selection} discusses a possible alternative to the main results of this paper and the motivations for both. Section~\ref{sec:many-bSDE} presents the main theorem (Theorem~\ref{thm:main1}) of this paper in detail and constructs the above-mentioned strong Markov processes with lifetime for concatenation. Section~\ref{sec:NTC} gives the proof of a general theorem (Theorem~\ref{thm:NTC}) by which the NTC is proved. Finally, Section~\ref{sec:logsum} gives the proof of the semimartingale decompositions of the logarithms of the non-martingale Radon--Nikod\'ym derivative processes (Proposition~\ref{prop:logsum}). \medskip

\noindent {\bf Frequently used notation.} $C(T)\in(0,\infty)$ is a constant depending only on $T$ and may change from inequality to inequality unless indexed by labels of equations. Other constants are defined analogously. We write $A\less B$ or $B\more A$ if $A\leq CB$ for a universal constant $C\in (0,\infty)$. $A\asymp B$ means both $A\less B$ and $B\more A$. For a process $Y$, the expectations $\E^Y_y$ and $\E^Y_\nu$ and the probabilities $\P^Y_y$ and $\P^Y_\nu$ mean that the initial conditions of $Y$ are the point $x$ and the probability distribution $\nu$, respectively. \emph{Unless otherwise mentioned, the processes are subject to constant initial conditions}. Lastly, $\log $ is defined with base $\e$, and $\log^ba\,\defeq\, (\log a)^b$. \medskip

\noindent {\bf Frequently used asymptotic representations.} The following can be found in \cite[p.136]{Lebedev-2}:
\begin{align}
&K_0(x)\sim \log x^{-1},&&\hspace{-3cm} x\searrow 0,\label{K00}\\
&K_0(x)\sim \sqrt{\pi/(2x)}\e^{-x},&&\hspace{-3cm} x\nearrow\infty,\label{K0infty}\\
&K_1(x)\sim  x^{-1},&&\hspace{-3cm} x\searrow 0,\label{K10}\\
&K_1(x)\sim \sqrt{\pi/(2x)}\e^{-x},&&\hspace{-3cm} x\nearrow\infty.\label{K1infty}
\end{align}

\section{A possible alternative and motivations}\label{sec:selection}
As pointed out in \cite{C:SDBG1-2}, we regard a legitimate combination of diffusion process and multiplicative functional as central to the problem of proving Feynman--Kac-type formulas of the two-dimensional many-body delta-Bose gas.
To determine the combination, this paper first constructs the more general stochastic many-$\delta$ motions. The Feynman--Kac-type formulas to be proven in \cite{C:SDBG3-2} then use the case of $w_\bj>0$ for all $\bj\in \mc E_N$ and
are obtained by \emph{deriving} the complete form of the multiplicative functional from the analytic solutions for the two-dimensional many-body delta-Bose gas in the earlier literature. On the other hand, we initially considered a very different combination, which is proposed in \cite[Remark~2.3]{C:BES-2} and intends to address the difficulty of mollifying delta potentials. Although the justification of this choice from \cite{C:BES-2} remains challenging, it has an appealing characteristic of using  $\delta$-functions directly and explicitly via local times at the approximate level. Therefore, this choice from \cite{C:BES-2} is arguably closest to the physical nature of the two-dimensional many-body delta-Bose gas. It is also closest to the Feynman--Kac formula of the one-dimensional many-body delta-Bose gas by Bertini and Cancrini~\cite{BC:1D-2} since this formula from \cite{BC:1D-2} admits direct applications of Brownian local times. 

Let us discuss the choice from \cite{C:BES-2} in more detail and explain why it is very different. The setting in \cite{C:BES-2}  extends the case of two particles and an original construction method of $\BES(0,\beta\da)$, $\beta>0$, due to Donati--Martin and Yor~\cite{DY:Krein-2}. Specifically, in the case of two particles, \cite{C:BES-2} obtains  the Feynman--Kac-type formulas for the semigroup solutions $\{P^\beta_t\}$ from \cite{AGHH:2D-2} for the relative motion operator $\ms L=-\Delta-\Lambda \delta(z)$
in several ways as follows, where $\hK_\nu(x)\,\defeq\, x^\nu K_\nu(x)$: 
 \begin{align}
P_t^\beta f(z_0)&=\E^{\beta\da}_{z^0/\two}\left[\frac{\e^{\beta t}K_0(\sqrt{\beta}\lvert z^0 \rvert)}{K_0(\sqrt{\beta}\lvert \two Z_t\rvert)}f(\two Z_t)\right]\label{PbetaFK0}\\
&=\lim_{\alpha\searrow 0}\E^{(-\alpha),\beta\da}_{z_0/\two}\biggl[\frac{\e^{\beta t}\widehat{K}_\alpha(\sqrt{\beta}\lvert z^0 \rvert)}{\widehat{K}_\alpha(\sqrt{\beta}\lvert \two Z_t\rvert)}f(\two Z_t)\biggr]\label{PbetaFK0alpha}\\
&=\lim_{\alpha\searrow 0}\E^{(-\alpha)}_{z_0/\two}[\e^{\Lambda_\alpha(\beta)L_t}f(\sqrt{2} Z_t)]
,\quad \forall\; f\in \C_b(\Bbb C),\;z^0\in \Bbb C\setminus\{0\}.\label{PbetaFK}
\end{align}
These approximations are ``lower-dimensional approximations,'' using the following definitions:
\begin{itemize}
\item For $\alpha\in (0,1/2)$, $\P^{(-\alpha)}$ denotes the probability measure under which $\{|Z_t|\}$ is a $2(1-\alpha)$-dimensional Bessel process. Let $\{L_t\}$ denote its Markovian local time at level zero.
\item $\{|Z_t|\}\sim \BES(-\alpha,\beta\da)$ under $\P^{(-\alpha),\beta\da}$, where, with a suitable constant $\Lambda_\alpha(\beta)$,
\begin{align}\label{DY:approx}
\d \P^{(-\alpha),\beta\da}_{z^0}\;\defeq\; \frac{\hK_\alpha(\sqrt{2\beta}|Z_t|)}{\hK_\alpha(\sqrt{2\beta}|z^0|)}\e^{\Lambda_\alpha (\beta)L_t-\beta t}\d \Bbb P^{(-\alpha)}_{z^0}\quad\mbox{ on } \sigma(|Z_s|,s\leq t).
\end{align}
\item The entire $\Bbb C$-valued processes $\{Z_t\}$ under $\P^{(-\alpha)}$ and $\P^{\albe}$ are both extended from Erickson's theorem \cite[Theorem~1, p.75--76]{Erickson-2} for continuous extensions of skew-product diffusion subject to the rule $Z_t=|Z_t|\exp \{\i \gamma_{\int_0^t \d s/|Z_s|^2}\}$ for $t<T_0(|Z|)$.
Here, $\{\gamma_t\}$ is an independent one-dimensional Brownian motion, and $T_\eta(\mathcal Z)\,\defeq\,\inf\{t\geq 0;\mathcal Z_t=\eta\}$.
\end{itemize}
In particular, \eqref{PbetaFK} is apparently similar to the following formal Feynman--Kac formula:
\begin{align}\label{Pbetadelta}
\e^{-t\ms L}f(z^0)=\E^{W}_{z_0/\two}\left[\exp\left\{\int_0^t \delta(\sqrt{2}W_s)\d s\right\}f(\sqrt{2}W_t)\right],
\end{align}
whereas the diffusion $\{Z_t\}$ relates to the stochastic one-$\delta$ motion $\{\ms Z_t\}=\{Z^j_t\}_{j=1,2}$ probabilistically representing the two-dimensional two-body delta-Bose gas by $\two Z_t=Z^2_t-Z^1_t$. Also, the expectation in  \eqref{PbetaFK0alpha} is a ground state transformation of $\BES(-\alpha)$ \cite[Section~3.1]{DY:Krein-2}.

The above-mentioned similarity between \eqref{PbetaFK} and \eqref{Pbetadelta} motivates a generalization in  \cite[Remark~2.2]{C:BES-2} to the case of many particles when $w_\bj\equiv 1$ and $\beta_\bj\equiv \beta$. The intended effect is that at the approximate level, the multiplicative functional is exponential with the exponent given by the sum of local times $\sum_{\bj=(j\prime,j)\in \mc E_N}\int_0^t \delta(Z^{j\prime}_s-Z^j_s)\d s$ for $\ms Z_t=(Z^1_t,\cdots,Z^N_t)$ approximating $2N$-dimensional Brownian motion in order to generalize \eqref{PbetaFK}. Additionally, the limit is expected to be made exact by a well-defined expected value as in \eqref{PbetaFK0}. Accordingly, the choice in \cite{C:BES-2} first \emph{postulates} the radial parts $\{|Z^\bj_t|\}_{\bj\in \mc E_N}$ as $\{\rho^{\bj}_t \}_{\bj\in \mc E_N}$, say under $\P^{N,\al}$, such that each $\{\rho^{\bj}_t\}$ is a version of $\BES(-\alpha)$, and  the $\two$-multiples of these Bessel processes converge to the distances among $N$ independent planar Brownian motions as $\alpha\searrow 0$. The multi-dimensional generalization of \eqref{PbetaFK} is then proposed in \cite[Remark~2.3]{C:BES-2} as
the $\alpha\searrow 0$ limit of 
\begin{align}
\E^{N,\al}\big[\e^{\sum_{\bj\in \mc E_N}\Lambda_\alpha(\beta) L^\bj_t}F\big]
=\E^{N,\albe}\biggl[\biggl(\prod_{\bj\in \mc E_N}\frac{\widehat{K}_\alpha(\sqrt{2\beta}\rho^{\bj}_0)}{\widehat{K}_\alpha(\sqrt{2\beta}\rho^{\bj}_t)}\biggr)\e^{(\sum_{\bj\in \mc E_N}\beta t)+A_t}\biggr].\label{DBG:fail}
\end{align}
Here, $F=F(\rho^\bj_t;\bj\in \mc E_N)$, and $\P^{N,\albe}$ extends $\P^{\albe}$ 
 on $\bigvee_{\bj}\sigma(\rho^{\bj}_s;s\leq t)$:
\begin{align}\label{def:BESabN}
\d \P^{N,\albe}\,\defeq\,\biggl(\,\prod_{\bj\in \mc E_N}\frac{\widehat{K}_\alpha(\sqrt{2\beta}\rho^{\bj}_t)}{\widehat{K}_\alpha(\sqrt{2\beta}\rho^{\bj}_0)}\biggr)\e^{\sum_{\bj\in \mc E_N}[\Lambda_\alpha(\beta) L^\bj_t-\beta t]+A_t}\d \P_{z_0}^{N,\al}
\end{align}
such that $\{L^\bj_t\}$ is the local time of $\{\rho^\bj_t\}$ at level $0$, and $\{A_t\}$ is a nonzero correction term defined by a sum of quadratic covariations. Due to the use of $\prod_{\bj\in \mc E_N}\widehat{K}_\alpha(\sqrt{\beta}|z^{j\prime}-z^j|)$, the possibility of \eqref{DBG:fail} may be further supported for sharing the same spirit of the one-dimensional many-body delta-Bose gas where an eigenfunction can be chosen to be a product form~\cite[Section~2.2]{Kardar:87-2}. Moreover, this is not a unique example showing eigenfunctions of product form in quantum many-body problems. On the other hand, to use \eqref{DBG:fail} for an authentic Feynman--Kac-type formula, there should be three fundamental problems to settle: (1) prove the existence of the limiting probability measure of $\P^{N,\albe}$ as $\alpha\searrow 0$, say $\P^{N,\zbe}$; (2) justify a Portmanteau-theorem-type continuity to represent the right-hand side of \eqref{DBG:fail} as an expectation under $\P^{N,\zbe}$; (3) incorporate angular parts under $\P^{N,\zbe}$ appropriately to recover the known analytic solutions for the two-dimensional many-body delta-Bose gas. 

The stochastic many-$\delta$ motions discussed in Section~\ref{sec:intro}
turn to a very different motivation, however. They consider Streit's example \cite[Example~2.2]{Streit-2} showing the following eigenfunction of sum form for $N$-body Hamiltonians  with contact interactions \emph{in three dimensions}: 
\[
\Phi(\bs z)\,\defeq\,\sum_{\bj\in \mc E_N}\Phi_0(\bs z^{j\prime}-\bs z^j),  \quad \bs z=(\bs z^1,\cdots,\bs z^N)\in \R^{3N},
\]
where $\bj=(j\prime,j)$, and $\Phi_0$ is an eigenfunction for the Hamiltonian with a $\delta$-function potential for one particle. Given the known relation of \eqref{PbetaFK0} to a ground state transformation \cite[Section~1.2]{C:SDBG1-2}, a natural counterpart of Streit's example for the two-dimensional $N$-body delts-Bose gas  is
\begin{align}\label{def:Psi}
\Psi(z )\,\defeq\, \sum_{\bj\in \mc E_N}K_0(\sqrt{\beta}|z^{j\prime}-z^{j}|),\quad z=(z^1,\cdots,z^N)\in \Bbb C^N.
\end{align}
Moreover, for initial conditions $z_0$ with $z_0^{j\prime}\neq z_0^j$ for all $\bj$, the ground state transformation corresponding to $\Psi(z)$ of $2N$-dimensional Brownian motion yields an SDE with the drift
\begin{align}\label{choice:b}
\wt{{\bs b}}_\beta(z)\,\defeq\, \nabla_z\log \Psi( z),
\end{align}
so that \eqref{def:Zintro} with ${\bs b}_\beta(z)=\wt{\bs b}_\beta(z)$ with $w_\bj\equiv 1$ and $\beta_\bj\equiv \beta$
for $z\neq 0$
 can be partially recovered.
Specifically, this use of a ground state transformation recovers \eqref{def:Zintro} \emph{only before} the first contact-creation time. As in the two-body case \cite{C:BES-2}, an extension beyond that time should expect the emergence of new objects, if not just for the forms, and deal with the construction. This issue is resolved here for constructing the stochastic processes and will be fully settled in \cite{C:SDBG3-2} when we prove Feynman--Kac-type formulas for the two-dimensional many-body delta-Bose gas. In particular, those Feynman--Kac-type formulas will indeed use formulations with new, nontrivial ingredients and mechanisms, not just $\Psi(z)$ and its immediate generalizations to $\sum_{\bj\in \mc E_N}w_\bj K_0(\sqrt{\beta_\bj}|z^{j\prime}-z^{j}|)$ for other $\bs w,\bbeta$. The Radon--Nikod\'ym derivative processes in \eqref{RNchoice} thus have the interpretation of reverting the ground state transformation in \eqref{PbetaFK0}
 and then defining a stochastic many-$\delta$ motion by another ground state transformation. 

Although the choice in \eqref{choice:b} appears to deviate significantly from the local time formulations discussed above, the use of a ground state transformation still allows a physical interpretation by using the distorted Brownian motion~\cite{AHS:DBM-2, EKS:DBM-2}. This relation has appeared in \cite{C:BES-2} for the two-dimensional two-body delta-Bose gas. Another interpretation implied by \eqref{choice:b} is that the stochastic many-$\delta$ motions may be viewed as $N$ independent two-dimensional Brownian motions conditioned for contact interactions; see \cite[Section~1]{C:SDBG1-2} for a related discussion. 

\section{Construction of the stochastic many-$\bs \delta$ motions}\label{sec:many-bSDE}
Our goal in this section is to prove the main theorem of this paper. It is stated below as Theorem~\ref{thm:main1}. The formulation uses probability measures on the space of $\CN\cup \{\partial\}$-valued c\`adl\`ag functions $D_{\CN\cup \{\partial\}}[0,\infty)$. Here, $\partial$ is an extra point added to $\CN$ and, as usual, $D_{\CN\cup \{\partial\}}[0,\infty)$ is equipped with Skorokhod's $J_1$-topology.  Accordingly, the statement of Theorem~\ref{thm:main1} uses the coordinate process $\{\ms Z_t\}=\{Z^j_t\}_{1\leq j\leq N}$ of $D_{\CN\cup \{\partial\}}[0,\infty)$. The proofs in Section~\ref{sec:SDE1} and~\ref{sec:SDE2} for Theorem~\ref{thm:main1} will restrict to the space $C_{\CN}[0,\infty)$ of $\CN$-valued continuous functions, though. The statement of Theorem~\ref{thm:main1} also uses the notations in \eqref{def:Kbbetaw0} and the following ones:
\begin{align}
\Twi_\eta&\,\defeq
 \inf\{t\geq 0;\exists\;\bj\neq \bi\;\mbox{s.t.}\;w_\bj>0 \;\&\;|Z^\bj_t|=\eta\}\label{def:Twi},\\
T_\eta^{\bs w}&\,\defeq\,\inf\{t\geq 0;\exists\;\bj\;\mbox{s.t.}\;w_\bj>0\;\&\; |Z^\bj_t|=\eta\},\label{def:Tw}\\
\hK_1^{\bbeta,\bj}(s)&\,\defeq\,\hK_1(\sqrt{2\beta_\bj}|Z^\bj_s|),\label{def:K1bj}
\end{align}
where $\bi,\bj\in\mc E_N$, $K_\nu$ is the Macdonald function of indices $\nu$, $\widehat{K}_\nu(x)\,\defeq\,x^\nu K_\nu(x)$, and $\{Z^\bj_t\}$ is a stochastic relative motion as defined in \eqref{def:Zbj2}.
Finally, the following sets serve as state spaces or sets of initial conditions to refine the condition in \eqref{def:ic}:
\begin{align}
\begin{split}\label{def:CNw}
\CNw&\,\defeq\, \{z\in \CN;z^\bj\neq 0\;\forall\;\bj\in \mc E_N\;\mbox{s.t.}\; w_\bj>0\},\\
\CNwi&\,\defeq\,\{z\in \CN;z^\bi=0\;\&\; z^\bj\neq 0\;\forall\;\bj\in \mc E_N\setminus\{\bi\}\;\mbox{s.t.}\; w_\bj>0\},\\
\CNwni&\,\defeq\, \CNw\cup \CNwi.
\end{split}
\end{align}
Here, $z^\bj\;\defeq\,(z^{j\prime}-z^j)/\two$ when given states $z=(z^1,\cdots,z^N)\in \CN$ of $\{\ms Z_t\}$.

\begin{thm}\label{thm:main1}
Fix an integer $N\geq 3$, $\bbeta\in (0,\infty)^{\mc E_N}$, and $\bs w\in \R_+^{\mc E_N}$ with $\#\{\bj;w_\bj>0\}\geq 2$. There exists a family 
$\ms P=\{\P^{\bbeta,\bs w}_{z_0};z_0\in \CNw\}\cup 
\{ \P^{\bbeta,\bs w,\bi}_{z_0};z_0\in\CNwi\}_{\bi\in \mc E_N,w_\bi>0}$ of probability measures  on $(D_{\CN\cup \{\partial\}}[0,\infty),\B(D_{\CN\cup \{\partial\}}[0,\infty)))$ such that the following properties hold: 
\begin{itemize}
\item [\rm (1$\cc$)] Let $\vartheta_s$ denote the shift operator on  $D_{\CN\cup \{\partial\}}[0,\infty)$: $\vartheta_s(\mathrm w)=\mathrm w(\cdot+s)$ for $\mathrm w\in D_{\CN\cup \{\partial\}}[0,\infty)$. For any $\P\in \ms P$, it holds $\P$-a.s. that $\ms Z_t\in \CN$ for all $t<T_\partial$ and $\ms Z_t\equiv \partial$ for all $t\geq T_\partial$ on $\{T_\partial<\infty\}$. Here, $T_\partial$ is defined as follows: 
\begin{align}\label{def:Tinfty}
\begin{aligned}
T_\partial=T_0^\infty&\,\defeq\,\lim_{m\to\infty}\ua\, T^m_0& &\P^{\bbeta,\bs w}_{z_0}\mbox{-a.s. for $z_0\in \CNw$},\\
T_\partial=T_0^{\bi,\infty}&\,\defeq\,\lim_{m\to\infty}\ua\, T^{\bi,m}_0&&\P^{\bbeta,\bs w,\bi}_{z_0}\mbox{-a.s. for $z_0\in \CNwi$},
\end{aligned}
\end{align}
and $T^m_0$ and $T^{\bi,m}_0$ are stopping times defined inductively as follows:  
\begin{align}
\label{def:Tm1}
 T^1_0&\,\defeq\, T^{\bs w}_0,
&T^{m+1}_0&\,\defeq\, \begin{cases}
\TwJm_0\circ \vartheta_{T^m_0}, &\mbox{if}\quad T^m_0<\infty,\\
\infty, &\mbox{otherwise},
\end{cases}\\
T^{\bi,1}_0&\,\defeq\, 
\Twi_0, &
T^{\bi,m+1}_0&\,\defeq\, \begin{cases}
\TwJm_0\circ \vartheta_{T^{\bi,m}_0}, &\mbox{if}\quad T^{\bi,m}_0<\infty,\\
\infty, &\mbox{otherwise},
\end{cases}\label{def:Tm2}
\end{align}
for $m\in\Bbb N$, where $\bs J_m$ is the unique random index in $\mc E_N$ such that $w_{\bs J_m}>0$ and $\ms Z_{T^{m}_0}\in \CNwJm$ for \eqref{def:Tm1}, or $w_{\bs J_m}>0$ and $\ms Z_{T^{\bi,m}_0}\in \CNwJm$ for \eqref{def:Tm2}. Also, $\P^{\bbeta,\bs w}_{z_0}(T_0^1<\infty)=1$ for all $z_0\in \CNw$, and
if $\#\{w_\bj;w_\bj>0\}\geq 2$ and $\beta_\bj=\beta$ for all $\bj\in \mc E_N$ with $w_\bj>0$, then $T_0^m$ and $T_0^{\bi,m}$ are a.s. finite for all $m\in \Bbb N$. 

\item[\rm (2$\cc$)] The family $\ms P$ defines
$\{\ms Z_t\}$ as a time-homogeneous strong Markov process with absorbing state $\partial$ and time to absorption $T_\partial$ from {\rm (1$\cc$)}, and 
\begin{align}\label{pathproperty}
\P\Bigg(\ms Z_t\in\bigcup_{\bi\in \mc E_N:w_\bi>0}\CNwni\cup\{\partial\},\;\forall\;t\geq 0\Bigg)=1,\quad \forall\;\P\in \ms P.
\end{align}

\item [\rm (3$\cc$)] For all $1\leq j\leq N$ and $\bj=(j\prime,j)\in \mc E_N$, the following equations hold over $0\leq t<T_\partial$ under $\P$ for any $\P\in \ms P$:
\begin{align}
Z^j_t&=z^j_0-\sum_{1\leq k\leq N\atop k\neq j}\int_0^t \frac{w_{j\vee\!\wedge k}\sqrt{\beta_{j\vee\!\wedge k}} K_1(\sqrt{\beta_{j\vee\!\wedge k}}|Z^{j}_s-Z^k_s|)(Z_s^{j}-Z^k_s)}{\sum_{\ell_2>\ell_1\atop N\geq \ell_2}w_{\ell_1\!\vee\!\wedge \ell_2}K_0(\sqrt{\beta_{\ell_2\!\vee\!\wedge \ell_1}}|Z^{\ell_2}_s-Z^{\ell_1}_s|)
|Z_s^{j}-Z^k_s|}\d s+\widetilde{W}^j_t,\label{SDE:Z1final}\\
Z^\bj_t&=Z^\bj_0-\sum_{\bk\in \EN}\frac{\sigma(\bj)\cdot \sigma(\bk)}{2}\int_0^t \frac{w_\bk\hK_1^{\bbeta,\bk}(s)}{K_0^{\bbeta,\bs w}(s)}\left(\frac{1}{\overline{Z}^\bk_s}\right)\d s+\widetilde{W}^\bj_t.\label{SDE:Z1}
\end{align}
Here, ${\ell_1\!\vee\!\!\wedge \ell_2}\,\defeq\,(\max\{\ell_1,\ell_2\},\min\{\ell_1,\ell_2\})\in \mc E_N$ for integers $\ell_1,\ell_2\geq 1$ with $\ell_1\neq \ell_2$, the Riemann-integral terms are well-defined as absolutely convergent integrals, $\{\wt{W}^j_t\}_{1\leq j\leq N}$ is a family of independent two-dimensional standard Brownian motions with zero initial condition, and $\wt{W}^\bj_t\,\defeq\,(\wt{W}^{j\prime}_t-\wt{W}^j_t)/\two$. Also, given $\bk=(k\prime,k)\in \mc E_N$, $\sigma(\bk)\in \{-1,0,1\}^N$ denotes the column vector such that the $k\prime$-th component is $1$, the $k$-th component is $-1$, and the remaining components are zero.
\end{itemize}
\end{thm}

\begin{rmk}\label{rmk:dropi}
In \cite{C:SDBG3-2}, $\bi$'s in $\P_{z_0}^{\bbeta,\bs w,\bi}$ and $T^{\bi,m}_0$ are dropped and reflected through the pair $\bi=(i\prime,i)$ defining contact in $z_0$. But we do not use these simplifications elsewhere in this paper. \qed 
\end{rmk}

\begin{rmk}\label{rmk:main1}
(1$\cc$) Theorem~\ref{thm:main1} does \emph{not} conclude that $\P(T_\partial<\infty)>0$ for some $\P\in\ms P$. $T_\partial$ is just a time when our construction ``terminates,'' and the setting of $\ms Z_t\equiv \partial$ for $t\geq T_\partial$ when $T_\partial<\infty$ is just to complete the definition of the whole process $\{\ms Z_t\}$ for all $t\in \R_+$. Presently, we do not know whether or not $\P(T_\partial=\infty)=1$ for all $\P\in \ms P$. \medskip 

\noindent {(2$\cc$)}  Remark~\ref{rmk:SDE2} (2$\cc$) gives more details about the a.s. non-explosion of $T^m_0$ and $T^{\bi,m}_0$. \medskip 

\noindent {\rm (3$\cc$)} The strong Markov processes described in the two cases in Section~\ref{sec:intro} will be constructed under the following probability measures: (1) $\P^{\bbeta,\bs w}_{z_0}$ for $z_0\in \CNw$ up to and including $T^{\bs w}_0$ (Section~\ref{sec:SDE1}), and (2) $\P^{\bbeta,\bs w,\bi}_{z_0}$ for $z_0\in \CNwi$ with $w_\bi>0$ up to and including $\Twi_0$ (Section~\ref{sec:SDE2}). Accordingly, more details of $\P^{\bbeta,\bs w}_{z_0}$ and $\P^{\bbeta,\bs w,\bi}_{z_0}$ can be found in Propositions~\ref{prop:SDE1} and~\ref{prop:SDE2}. Note that Proposition~\ref{prop:SDE2} also constructs $\P^{\bbeta,\bs w,\bi}_{z_0}$ for $z_0\in \CNw$ with $w_\bi>0$ up to and including $\Twi_0$. For fixed $z_0\in \CNw$, these probability measures relate to $\P^{\bbeta,\bs w}_{z_0}$ by 
\begin{align}\label{PP:extension}
\E_{z_0}^{\bbeta,\bs w}[F(\ms Z_{t\wedge T^{\bs w}_0};t\geq 0)]=\E_{z_0}^{\bbeta,\bs w,\bi}[F(\ms Z_{t\wedge T^{\bs w}_0};t\geq 0)],\quad \forall\;0\leq F\in \B(C_{\CN}[0,\infty)).
\end{align}
See Remark~\ref{rmk:SDE2} (1$\cc$) for more details of \eqref{PP:extension}. \qed 
\end{rmk}

The proof of Theorem~\ref{thm:main1} (Sections~\ref{sec:SDE1}--\ref{sec:SDEcont}) will build on the stochastic one-$\delta$ motions $\{\ms Z_t\}=\{Z^j_t\}_{1\leq j\leq N}$ defined under $\P^{\beta_\bi\da,\bi}$ together with three tools discussed in Section~\ref{sec:threetool}
(Theorem~\ref{thm:NTC}, Proposition~\ref{prop:BMntc} and Proposition~\ref{prop:logsum}). Accordingly, we will use \cite{C:SDBG1-2} for several properties of the stochastic one-$\delta$ motions, especially those in \cite[Section~4 before Proposition~4.2]{C:SDBG1-2}. To lighten the notation for the applications, we write from now on
\begin{align}
\P^{\bi}\,\defeq\, \P^{\beta_\bi\da,\bi}.
\end{align}

\subsection{Three preliminary tools}\label{sec:threetool}
Theorem~\ref{thm:NTC} below gives the first tool. It can be viewed as a theorem identifying conditions on multiple nonnegative semimartingales with Bessel-process-like finite-variation parts such that the sum of the semimartingales does not hit zero simultaneously. The precise condition is only in a sum form and more general, though. As an application of this theorem, we will prove the stronger NTC described in Section~\ref{sec:intro} for the stochastic many-$\delta$ motions.

\begin{thm}[No simultaneous contacts]\label{thm:NTC}
Let $\mc J$ be a finite set with $n_0\defeq \#\J\geq 2$. Suppose that  for some stopping times $S,T$, 
we have a family of nonnegative processes $\{\varrho^\bj_t;S\leq t<T\}_{\bj\in \J}$ defined on $\{S<T\}$ such that 
\begin{gather}
\sum_{\bj\in \mc J}\varrho^\bj_t=\sum_{\bj\in \mc J}\varrho^\bj_S+\int_S^t
\sum_{\bj\in \J}\frac{\mu_\bj(r)}{\varrho_r^\bj}\d r+\sum_{\bj\in \J}(\mc B^{\bj}_t-\mc B^\bj_S),\quad \forall\;S\leq t<T,\label{eq:ntc}\\
 \int_S^t \sum_{\bj\in \J}\frac{\d r}{\varrho^\bj_r}<\infty, \quad \forall\;S\leq t<T,\label{ass:ntc}
\end{gather}
and the following conditions are satisfied by $\{\mu_\bj(s);S\leq t<T\}_{\bj\in \mc J}$  and $\{\mc B_t^\bj\}_{\bj\in \J}$:
\begin{itemize}
\item $\{\mu_\bj(s)\}$ are bounded progressively measurable, and for some constant $\alpha\geq 0$ satisfying
\begin{align}\label{def:alpha}
\frac{n_0^2(1-2\alpha)}{n_0+\frac{n_0(n_0-1)}{2}}\geq 1,
\end{align}
a.s. on $ \{S<T\}$, we have $\mu_\bj(s)=\mu_\bj(s,\omega)\geq (1-2\alpha)/2>0$ for a.e. $s\in [S, T)$ and all $\bj\in \mc E_N$.
\item $\{\mc B_t^\bj\}_{\bj\in \J}$ consists of one-dimensional standard Brownian motions satisfying 
\begin{align}\label{W:qcv}
\la \mc B^{\bj},\mc B^{\bk}\ra_t-\la \mc B^{\bj},\mc B^{\bk}\ra_S&=\frac{1}{2}\int_S^t \sigma_{\bj,\bk}(s)\d s,\quad \forall\;S\leq t< T,\;\bj\neq \bk,
\end{align}
for progressively measurable $\sigma_{\bj,\bk}(s)=\sigma_{\bj,\bk}(s,\omega)$ with $\esssup \{|\sigma_{\bj,\bk}(s,\omega)|;s\in [S,T)\}\leq 1$ a.s. on $ \{S<T\}$. 
\end{itemize}
Define a process $\{\varrho_t;S\leq t\leq T\}$ on $\{S<T\}$ by  
\begin{align}
\varrho_t\,\defeq
\begin{cases}\label{def:rhoT}
\displaystyle \;\sum_{\bj\in\J}\varrho^\bj_t,& S\leq t<T,\\
\displaystyle \;\lim_{s\nearrow T}\varrho_s\in [0,\infty],& t=T\mbox{ if }T<\infty.
\end{cases}
\end{align} 
Then $\P\left( \varrho_t>0\mbox{ for all $t\in (S,T]$},\;S<T\right)=\P(S<T)$.
\end{thm}

Here and in what follows, {\bf all the time variables are understood to take values only in $\R_+$ unless otherwise mentioned}. So, $\{\varrho_t;S\leq t\leq T\}=\{\varrho_t;S\leq t< T\}$ on $\{T=\infty\}$. 

\begin{rmk}\label{rmk:ntc}
(1$\cc$) For any integer $n_0\geq 2$, \eqref{def:alpha} can be attained at all small $\alpha>0$ and at $\alpha=0$ since for integers $n>1$,
\begin{align*}
\frac{n^2}{n+\frac{n(n-1)}{2}}> 1 \Longleftrightarrow \frac{2n^2}{n^2+n}> 1\Longleftrightarrow n^2> n\Longleftrightarrow n>1.
\end{align*}

\noindent (2$\cc$) The existence of $\lim_{t\nearrow}\varrho_t$, possibly $\infty$-valued, in \eqref{def:rhoT} follows from \eqref{eq:ntc} by using the increasing finite-variation process in  \eqref{eq:ntc}  and the H\"older regularity of Brownian motions on compacts. In particular, a.s. on $\{S<T<\infty\}$, the case of $\varrho_T=\infty$ from \eqref{def:rhoT} is attained if and only if the finite-variation part in \eqref{eq:ntc} explodes at $t=T$. 
\qed 
\end{rmk}

The proof of Theorem~\ref{thm:NTC} will be presented in Section~\ref{sec:NTC}. It is a comparison-theorem-type argument extending \cite[Lemma~2.1]{IK:BM-2} by proving the \emph{converse} property of no contacts and incorporating some key features of the present problem, including using many processes. 

Our second tool is obtained from applying Theorem~\ref{thm:NTC}. Recall that under $\P^{\bi}$, 
\begin{align}
\label{def:ZSDE2-2}
Z^j_t&=
Z_0^{j}-\frac{(\1_{j=i\prime}-\1_{j=i})}{\two}\int_0^t \frac{\hK_1(\sqrt{2\beta_\bi}|Z^\bi_s|)}{ K_0(\sqrt{2\beta_\bi}|Z^\bi_s|)}\biggl(\frac{1}{ \overline{Z}^\bi_s}\biggr)\d s+W^{j}_t,\\
Z^\bj_t\,&=Z_0^\bj-\frac{\sigma(\bj)\cdot \sigma(\bi)}{2}\int_0^t \frac{\hK_1(\sqrt{2\beta_\bi}|Z^\bi_s|)}{ K_0(\sqrt{2\beta_\bi}|Z^\bi_s|)}
\biggl( \frac{1}{\overline{Z}^\bi_s}\biggr)\d s+W^\bj_t,\label{SDE:Zbj-2}\\
|Z^\bj_t|&=|Z^\bj_0|+\int_0^t \biggl[\frac{1}{2|Z^\bj_s|}-\frac{\sigma(\bj)\cdot \sigma(\bi)}{2|Z^\bj_s|} \Re\biggl(\frac{Z^\bj_s}{Z^\bi_s}\biggr)\frac{\hK_1(\sqrt{2\beta_\bi}|Z^\bi_s|)}{ K_0(\sqrt{2\beta_\bi}|Z^\bi_s|)}\biggr]\d s+B^\bj_t;\label{SDE:|Zbj|-2}
\end{align}
see \cite[(4.1), (4.3) and (4.5)]{C:SDBG1-2}. Here, $\widehat{K}_\nu(\cdot)$ is defined before \eqref{PbetaFK0},
$\sigma(\bk)$ is defined in Theorem~\ref{thm:main1},
$\{W^j_t\}_{1\leq j\leq N}$ is a family of independent two-dimensional standard Brownian motions with $W^j_0=0$, $\{Z^\bj_t\}$ is defined in \eqref{def:Zbj2},
$W^\bj_t\,\defeq\,(W^{j\prime}_t-W^j_t)/\two$,
and with $Z^\bj_t=X^\bj_t+\i Y^\bj_t$ and $W^\bj_t=U^\bj_t+\i V^\bj_t$ for real $X^\bj_t, Y^\bj_t,U^\bj_t, V^\bj_t$,
\begin{align}\label{def:Bj-2}
B^\bj_t\;\defeq \int_0^t \frac{X^\bj_s\d U^\bj_s+Y^\bj_s\d V^\bj_s}{|Z^\bj_s|}\Longrightarrow \la B^\bj,B^\bk\ra_t=\frac{\sigma(\bj)\cdot \sigma(\bk)}{2}\int_0^t \frac{X^\bj X^\bk_s+Y^\bj_sY^\bk_s}{|Z^\bj_s||Z^\bk_s|}\d s,
\end{align}
where the implication follows by using
\begin{align}\label{covar:UV-2}
\la U^\bj,U^\bk\ra_t=\la V^\bj,V^\bk\ra_t=\frac{\sigma(\bj)\cdot \sigma(\bk)}{2}t.
\end{align}

\begin{prop}\label{prop:BMntc}
For all  $\bi\in \mc E_N$ and $z_0\in \CN$, it holds that 
$\P_{z_0}^{\bi}(T_0^\bi< \min_{\bj:\bj\neq \bi}\widetilde{T}_0^\bj=\infty)=1$, 
where 
$T_\eta^\bj\,\defeq\,\inf\{t\geq 0;|Z^\bj_t|=\eta\}$ and $\widetilde{T}_\eta^\bj\,\defeq\,\inf\{t> 0;|Z^\bj_t|=\eta\}$
 for all $\eta\geq 0$ and $\bj\in \mc E_N$. 
\end{prop}
\begin{proof}
We first show the following property:
\begin{align}\label{BMntc:000}
\P^{\bi}_{z_0}(\widetilde{T}_0^\bj\geq T_0^\bi)=1,\quad \forall\;z_0\in \Bbb \CN,\;\bj\in \mc E_N\setminus\{\bi\}.
\end{align}
To see \eqref{BMntc:000}, it is enough to consider $z_0^\bi\neq 0$. In this case, $T^\bi_\vep\nearrow T_0^\bi$ as $\vep\searrow 0$. Also, for any $0<\vep<|z_0^\bi|$, $\P^{\bi}_{z_0}(\widetilde{T}_0^\bj\leq T^\bi_\vep)=0$ since we can locally invert $\P^{\bi}$ to $\P^{(0)}$ before $T_0^\bi$ via a change of measure relationship between $\P^\bi$ and $\P^{(0)}$ (cf. \cite[(2.7)]{C:SDBG1-2}):
\begin{align}
\P^{\bi}_{z_0}(\widetilde{T}^\bj_0\leq T_\vep^\bi)&=\P^{\bi}_{z_0}(\widetilde{T}^\bj_0\leq T_\vep^\bi<T_0^\bi)=\lim_{n\to\infty}\ua \sum_{k=1}^\infty\P^{\bi}_{z_0}\left(\widetilde{T}^\bj_0\leq T_\vep^\bi,\frac{k-1}{2^n}\leq T_\vep^\bi<\frac{k}{2^n}<T_0^\bi\right)\notag\\
&=\lim_{n\to\infty}\ua \sum_{k=1}^\infty\E^{(0)}_{z_0}\Biggl[\frac{\e^{-\frac{\beta_\bi k}{2^n}}K_0(\sqrt{2\beta_\bi}|Z_{\frac{k}{2^n}}^\bi|)}{K_0(\sqrt{2\beta_\bi}|z_0^\bi|)};\widetilde{T}^\bj_0\leq T_\vep^\bi,\frac{k-1}{2^n}\leq T_\vep^\bi<\frac{k}{2^n}\Biggr]=0,\label{DY:comapp}
\end{align}
where the first equality holds since $\P^\bi_{z_0}(T^\bi_0<\infty)=1$ \cite[(2.9)]{DY:Krein-2}, and the last equality follows since $\P^{(0)}_{z_0}(\widetilde{T}_0^\bj=\infty)=1$ by the polarity of $0$ under planar Brownian motion \cite[(2.7) Proposition, p.191]{RY-2}, and so, all of the expectations are just equal to zero.

Next, we show that for all $\bi\in \mc E_N$, $\bj\neq \bi$, and $z_0\in \CN$,
 \begin{align}\label{BMntc}
 \P^{\bi}_{z_0}(\exists\; t\in (0,\infty)\mbox{ s.t. } |Z^\bj_t|=|Z^\bi_t|=0)=0.
 \end{align}
To see \eqref{BMntc}, we use the SDEs of $|Z^\bj_t|$ and $|Z^\bi_t|$ under $\P^{\bi}_{z_0}$ and the notation in \eqref{def:Kbbetaw0}:
\begin{align*}
|Z^\bj_t|+|Z^\bi_t|&=|Z_0^\bj|+|Z_0^\bi|+\int_0^t
\frac{1}{|Z^\bi_s|}\Biggl(\frac{1}{2}-\frac{\hK_1^{\bbeta,\bi}(s)}{K_0^{\bbeta,\bi}(s)}\Biggr)\d s
\\
&\quad \;+\int_0^t \frac{\d s}{2|Z^\bj_s|}\d s-\int_0^t \frac{1}{|Z^\bj_s|}\Re\biggl(\frac{Z^\bj_s}{Z^\bi_s}\biggr)\cdot \frac{\sigma(\bi)\cdot \sigma(\bj)}{2}\frac{\hK_1^{\bbeta,\bi}(s)}{K_0^{\bbeta,\bi}(s)}\d s+B^\bj_t+B^\bi_t\\
&=|Z_0^\bj|+|Z_0^\bi|+\int_0^t\frac{\d s}{2|Z^\bj_s|}\\
&\quad\;+\int_0^t
\frac{1}{2|Z^\bi_s|}\Biggl(1-2\frac{\hK_1^{\bbeta,\bi}(s)}{K_0^{\bbeta,\bi}(s)}- 2\frac{|Z^\bi_s|}{|Z^\bj_s|}\Re\biggl(\frac{Z^\bj_s}{Z^\bi_s}\biggr)\cdot \frac{\sigma(\bi)\cdot \sigma(\bj)}{2}\frac{\hK_1^{\bbeta,\bi}(s)}{K_0^{\bbeta,\bi}(s)}\Biggr)\d s  +B^\bj_t+B^\bi_t,
\end{align*}
which allows the application of Theorem~\ref{thm:NTC}. Specifically, for any constant $\theta>0$, let
\[
S_1\,\defeq\, \inf\{t\geq 0;|Z^\bj_{t}|+|Z^\bi_{t}|\leq \theta\},\quad T_1\,\defeq\, \inf\{t\geq S_1;|Z^\bj_{t}|+|Z^\bi_{t}|=2 \theta\}.
\]
Then on the event $\{S_1<T_1\}$,  we write
\begin{align}
\begin{split}\label{SDE:NTC1}
|Z^\bj_t|+|Z^\bi_t|&=|Z^\bj_{S_1}|+|Z^\bi_{S_1}|+\int_{S_1}^t \frac{\mu_\bj(s)}{|Z^\bj_s|}\d s+\int_{S_1}^t \frac{\mu_\bi(s)}{|Z^\bi_s|}\d s\\
&\quad+(B^\bj_t-B^\bj_{S_1})+(B^\bi_t-B^\bi_{S_1}),\quad t\geq S_1.
\end{split}
\end{align}
Here, $\mu_\bj(s)\equiv 1/2$, and $\mu_\bi(s,\omega)\equiv [1-\nu_\bi^\theta(s,\omega)]/2$ for 
\[
\nu_\bi^\theta(s)\,\defeq\, 2\frac{\hK_1^{\bbeta,\bi}(s)}{K_0^{\bbeta,\bi}(s)}+2\frac{|Z^\bi_s|}{|Z^\bj_s|}\Re\biggl(\frac{Z^\bj_s}{Z^\bi_s}\biggr)\cdot \frac{\sigma(\bi)\cdot \sigma(\bj)}{2}\frac{\hK_1^{\bbeta,\bi}(s)}{K_0^{\bbeta,\bi}(s)}.
\]
The crucial property we need is that for all $\eta>0$, there exists $\theta_0>0$ such that on $\{S_1<\infty\}$, 
\[
\esssup\{|\nu_\bi^\theta(s,\omega)|;s\in [S_1,T_1]\}<\eta,\quad \forall\;0<\theta<\theta_0.
\]
This property holds since $K_0(0)^{-1}=0$, $\hK_1$ is bounded on compacts in $\R_+$,
\begin{align}\label{ZbiZbj:time}
\biggl|\frac{|Z^\bj_s|}{|Z^\bj_s|}\Re\biggl(\frac{Z^\bj_s}{Z^\bi_s}\biggr)\biggr|\leq 1,\quad\forall\;
s\geq 0 \;\mbox{s.t.}\; Z^\bj_s\neq 0\;\&\;Z^\bi_s\neq 0,
\end{align}
and the Lebesgue measure of the set of times $s\geq 0$ such that $Z^\bj_s=0$ or $Z^\bi_s=0$ is zero. [The latter holds because $Z^\bi$ has transition densities \cite[Theorem~2.1 (2$\cc$)]{C:SDBG1-2} and $Z^\bj$ is a linear combination of $Z^{\bi}$ and an independent two-dimensional Brownian motion with the contribution of the Brownian motion being nontrivial by \eqref{def:ZSDE2-2}.] Also, by \eqref{def:Bj-2}, $B^\bj$ and $B^\bi$ satisfy \eqref{W:qcv} for $\mc J=\{\bj,\bi\}$ by \eqref{def:Bj-2}. Hence, by choosing $\theta$ sufficiently small, 
Theorem~\ref{thm:NTC}  applies and gives 
\begin{align}\label{NTC:t00}
\P^\bi_{z_0}(|Z^\bj_t|+|Z^\bi_t|=0\mbox{ for some }t\in (S_1,T_1]\cap \R_+,S_1<T_1)=0;
\end{align}
see also Remark~\ref{rmk:ntc} (1$\cc$). Similarly, we have 
\begin{align}\label{NTC:t01}
\P^\bi_{z_0}(|Z^\bj_t|+|Z^\bi_t|=0\mbox{ for some }t\in (S_n,T_n]\cap\R_+,S_n<T_n)=0,\quad \forall\;n\geq 2,
\end{align}
where $S_n,T_n$, for $n\geq2$, are defined inductively by 
\[
S_n\,\defeq\, \inf\{t\geq T_{n-1};|Z^\bj_{t}|+|Z^\bi_{t}|= \theta\},\quad T_n\,\defeq\, \inf\{t\geq S_n;|Z^\bj_{t}|+|Z^\bi_{t}|=2 \theta\}.
\]

Now, we show that \eqref{NTC:t00} and \eqref{NTC:t01} are enough to get the following property: 
\begin{align}\label{NTCt0}
\P_{z_0}^\bi(\exists\;0<t\leq t_0\mbox{ s.t. }|Z_t^\bj|+|Z^\bi_t|=0)=0\quad\mbox{for any fixed $t_0>0$}.
\end{align}
To this end, we first bound the event under consideration in the following manner:  
\begin{align}\label{NTCt0bdd}
\begin{split}
&\quad\;\{\exists\;0<t\leq t_0\mbox{ s.t. }|Z_t^\bj|+|Z^\bi_t|=0\}\\
&\subset \left(\bigcup_{n=1}^\infty  \{|Z^\bj_t|+|Z^\bi_t|=0\mbox{ for some }t\in (S_n,T_n]\cap\R_+,S_n<T_n\} 
\right)
\cup \{T_n\leq t_0,\forall\;n\geq 1\}.
\end{split}
\end{align}
In more detail, this bound considers the fact that $\{t<\sup_nT_n;|Z^\bj_t|+|Z^\bi_t|=0\}\subset \bigcup_n [S_n,T_n]$, so $\{T_n\leq t_0,\forall\;n\geq 1\}$ handles the case that by time $t_0$, all $[S_n,T_n]$ have failed to capture any zero of $|Z^\bj|+|Z^\bi|$ in $(0,t_0]$ as $\{T_n\}$ has accumulated by time $t_0$. To bound the probability of the right-hand side of \eqref{NTCt0bdd}, first, use \eqref{NTC:t00} and \eqref{NTC:t01}. To bound the probability of $\{T_n\leq t_0,\forall\;n\geq 1\}$, note that on this event, $\lim_n\ua T_n=\lim_n\ua S_n\leq t_0$ and $|T_n-S_n|\to 0$ as $n\to\infty$ by monotonicity and $S_n\leq T_n\leq S_{n+1}$, so the continuity of paths forces that
\[
\theta=\lim_{n\to\infty}|Z^\bj_{S_n}|+|Z^\bi_{S_n}|=\lim_{n\to\infty}|Z^\bj_{T_n}|+|Z^\bi_{T_n}|=2\theta,
\]
which contradicts the assumption that $\theta>0$. We have proved \eqref{NTCt0}, and hence, \eqref{BMntc}. 

Now we prove $\P^{\bi}_{z_0}(\widetilde{T}^\bj_0=\infty)=1$ for all $\bj\neq \bi$ and $z_0\in \CN$. For $m\in \Bbb N$, define $(g_n(m),d_n(m))$ as the $n$-th excursion interval of $|Z^\bi|$ away from zero such that the length of the interval is $\geq 1/m$.
Note that $g_n(m)+1/m$ are stopping times. See the discussion after Lemma~\ref{lem:RRprop} for more details. 
With $\vartheta_\bullet$ denoting the shift operator, 
\[
\P_{z_0}^{\bi}\Big(g_n(m)+1/m< \widetilde{T}_0^{\bj}<T_0^\bi\circ\vartheta_{g_n(m)+1/m}\Big)=0
\]
by the strong Markov property of $\{(Z^\bi_t,Z^\bj_t)\}$ at time $g_n(m)+1/m$ and
the fact obtained in the first paragraph that $\P^{\bi}_{z_1}(\widetilde{T}_0^\bj\geq T_0^\bi)=1$ for all $z_1\in \CN$. The foregoing equality implies that
\[
\P^{\bi}_{z_0}(\widetilde{T}_0^\bj\in J)=0,\quad \mbox{where }J\;\defeq \bigcup_{m=1}^\infty\bigcup_{n=1}^\infty (g_n(m)+1/m,d_n(m)).
\]
Since $\R_+\setminus\{t\geq 0;|Z^\bi_t|=0\}=J$, we must have 
\begin{align}\label{BMntc1}
\P^{\bi}_{z_0}(\widetilde{T}_0^\bj\in \{\infty\}\cup \{t\geq 0;|Z^\bi_t|=0\})=1. 
\end{align}
Moreover, by \eqref{BMntc}, \eqref{BMntc1} can be reduced to $\P^{\bi}_{z_0}(\widetilde{T}_0^\bj=\infty)=1$. The proof is complete.
\end{proof}

Our last tool consists of Proposition~\ref{prop:logsum} and Corollary~\ref{cor:superMG} on the semimartingale decompositions of the logarithms of the non-martingale Radon--Nikod\'ym derivative processes we use to construct the stochastic many-$\delta$ motions; recall \eqref{RNchoice}.
 We will prove Proposition~\ref{prop:logsum}  in Section~\ref{sec:logsum} and
Corollary~\ref{cor:superMG} at the end of the present subsection (Section~\ref{sec:threetool}).

\begin{prop}\label{prop:logsum}
Fix $\bi\in \mc E_N$, $\bbeta\in (0,\infty)^{\mc E_N}$ and $\bs w\in \R_+^{\mc E_N}$ with $w_\bi>0$ and $\#\{\bj;w_\bj>0\}\geq 2$. Then for any $z_0\in \CNwni$,  it holds $\P_{z_0}^{\bi}$-a.s. that
\begin{align}
\begin{split}
 \log \frac{K^{\bbeta,\bs w}_0 (t)}{ w_\bi K^{\bbeta,\bi}_0 (t)}
&= \log \frac{K^{\bbeta,\bs w}_0 (0)}{ w_\bi K^{\bbeta,\bi}_0 (0)}+A^{\bbeta,\bs w,\bi}_0(t) +N_0^{\bbeta,\bs w,\bi}(t)-\frac{1}{2}\la N^{\bbeta,\bs w,\bi}_0,N^{\bbeta,\bs w,\bi}_0\ra_t,\quad \forall\; t\geq 0.\label{RN}
\end{split}
\end{align}
Here, the terms on both sides satisfy the following properties: 
\begin{itemize}
\item [\rm (1$\cc$)] $ K^{\bbeta,\bs w}_0 (t)/ [w_\bi K^{\bbeta,\bi}_0 (t)]$ is understood to be equal to $1$ whenever $\ms Z_t\notin \CNw$, $t\geq 0$.
\item [\rm (2$\cc$)] $\{A^{\bbeta,\bs w,\bi}_0(t)\}$ is a finite-variation process defined by
\begin{align}
A^{\bbeta,\bs w,\bi}_0(t)&\,\defeq\, \mathring{A}^{\bbeta,\bs w,\bi}_0(t)+\widetilde{A}^{\bbeta,\bs w,\bi}_0(t),\label{def:A}\\
\mathring{A}^{\bbeta,\bs w,\bi}_0(t)&\,\defeq\, \sum_{\bj\in\mathcal E_N\setminus\{ \bi\}}2\left(\frac{w_\bj}{w_\bi}\right)\int_0^t K_0(\sqrt{2\beta_\bj} |Z^\bj_s|)\d L^{\bi}_s,\label{def:Aring}\\
\widetilde{A}^{\bbeta,\bs w,\bi}_0(t)&\,\defeq\,\sum_{\bj\in \mathcal E_N} \int_0^t \beta_\bj\frac{w_\bj  K^{\bbeta,\bj}_0(s)}{K^{\bbeta,\bs w}_0(s)}\d s- \beta_\bi t,\label{def:Atilde}
\end{align}
where $\{L^\bi_t\}$ is the Markovian local time of $\{Z^\bi_t\}$ at level $0$ \cite{DY:Krein-2} subject to the normalization 
\begin{align}\label{def:DYLT-2}
\E^{\bi}_0\left[\int_0^\infty\e^{-q\tau} \d  L^\bi_\tau\right]=\frac{1}{\log (1+q/\beta_\bi)},\quad \forall\;q\in (0,\infty).
\end{align}

\item [\rm (3$\cc$)] $\{N^{\bbeta,\bs w,\bi}_0(t)\}$ is a real-valued continuous local martingale defined by 
\begin{align}
N^{\bbeta,\bs w,\bi}_0(t)&\,\defeq \,\mathring{N}^{\bbeta,\bs w,\bi}_0(t)+\widetilde{N}^{\bbeta,\bs w,\bi}_0(t),\label{def:Nwi}\\
\mathring{N}^{\bbeta,\bs w,\bi}_0(t)&\,\defeq\,
\int_0^{t}\frac{\hK^{\bbeta,\bi}_1(s)[K^{\bbeta,\bs w}_0(s)-w_\bi K_0^{\bbeta,\bi}(s)]}{|Z^\bi_s|K^{\bbeta,\bi}_0(s)K^{\bbeta,\bs w}_0(s)}\d B^\bi_s,\label{def:Nwiring}\\
\widetilde{N}^{\bbeta,\bs w,\bi}_0(t)&\,\defeq\,
-\sum_{\bj\in \mc E_N\setminus\{ \bi\}}\int_0^{t}\frac{w_\bj\hK^{\bbeta,\bj}_1(s)}{|Z^\bj_s|K^{\bbeta,\bs w}_0(s)}\d B^\bj_s,\label{def:Nwitilde}
\end{align}
where $\{B^\bj_t\}$ is the one-dimensional Brownian motion defined in \eqref{def:Bj-2}. 

\item [\rm (4$\cc$)] Recall that $\Twi_\eta$ is defined in \eqref{def:Twi}, and $T^{\bs w\setminus\{w_\bi\}}_{\eta}\nearrow\infty$ as $\eta\searrow 0$ $\P^{\bi}_{z_0}$-a.s. by Proposition~\ref{prop:BMntc}. 
For any $\eta\in (0,\min_{\bj:w_\bj>0,\bj\neq \bi}|z^\bj_0|)$, it holds that
\[
\{N^{\bbeta,\bs w,\bi}_0(t\wedge T^{\bs w\setminus\{w_\bi\}}_{\eta})\}\quad \mbox{and}\quad
\{\mathring{N}^{\bbeta,\bs w,\bi}_0(t\wedge T^{\bs w\setminus\{w_\bi\}}_{\eta})\}
\]
are continuous $L^2$-martingales, and 
\[
\E^{\bi}_{z_0}\left[\exp\left\{\lambda \la \mathring{N}^{\bbeta,\bs w,\bi}_0,\mathring{N}^{\bbeta,\bs w,\bi}_0\ra_{t\wedge T^{\bs w\setminus\{w_\bi\}}_{\eta}}\right\}\right]<\infty,\quad \forall\;\lambda\in \R,\;t\geq 0.
\]
\end{itemize}
\end{prop}

\begin{cor}\label{cor:superMG}
Assume the conditions of Proposition~\ref{prop:logsum}.
Then for any $z_0\in \CNwni$, 
\begin{align}\label{superMG}
\begin{split}
\mathcal E_{z_0}^{\bbeta,\bs w,\bi}(t)&\,\defeq\,\frac{w_\bi K_0^{\bbeta,\bi}(0)}{K_0^{\bbeta,\bs w}(0)}
\frac{\e^{-A^{\bbeta,\bs w,\bi}_0(t)}K_0^{\bbeta,\bs w}(t)}{w_\bi K_0^{\bbeta,\bi}(t)}
=\exp\left\{N_0^{\bbeta,\bs w,\bi}(t)-\frac{1}{2}\la N^{\bbeta,\bs w,\bi}_0,N^{\bbeta,\bs w,\bi}_0\ra_t\right\}
\end{split}
\end{align}
is a continuous local martingale under $\P_{z_0}^{\bi}$ such that for any $\eta\in (0,\min_{\bj:w_\bj>0,\bj\neq \bi}|z^\bj_0|)$, 
\begin{align}\label{L2MG}
\mathcal E_{z_0}^{\bbeta,\bs w,\bi}(t\wedge \Twi_\eta),\quad t\geq 0,
\end{align}
is a martingale under $\P_{z_0}^{\bi}$. In particular, the process in \eqref{superMG} is a supermartingale. 
\end{cor}

Proposition~\ref{prop:logsum} may be viewed as a special version of the multi-dimensional It\^{o}'s formula as \eqref{RN} expands $\log [K^{\bbeta,\bs w}_0(t)/(w_\bi K^{\bbeta,\bi}_0(t))]$ into a semimartingale according to the SDEs of $\{Z^\bj_t\}_{\bj\in \mc E_N}$ under $\P_{z_0}^\bi$. The proof addresses various issues described in Section~\ref{sec:intro} now that It\^{o}'s formula is not directly applicable. Properties that make It\^{o}'s formula not directly applicable include the fact that $K_0'(x)=-K_1(x)$ diverges to $-\infty$ at $x=0$ so that $K_0(x)$ is not a difference of convex functions in an open set containing $\R_+$, but the range of $t\mapsto |Z^\bi_t|$ under $\P^{\bi}$ includes $0$. The multi-dimensionality is also fundamental to the non-applicability. Note that the choice of initial conditions $z_0\in \CNwni$ plays a nontrivial role in the proof of Proposition~\ref{prop:logsum}.

\begin{rmk}\label{rmk:NTC}
\noindent (1$\cc$) For any nonzero $\bs w\in \R_+^{\mc E_N}$ with $\#\{\bj;w_\bj>0\}\geq 2$, it is plain that
\[
 \left(\min_{\bj:w_\bj>0}\beta_\bj-\beta_\bi\right)t\leq
\widetilde{A}^{\bbeta,\bs w,\bi}_0(t)\leq \left(\max_{\bj:w_\bj>0}\beta_\bj-\beta_\bi\right)t.
\]
In particular, $\widetilde{A}^{\bbeta,\bs w,\bi}_0(t)\equiv 0$ if $\beta_\bj=\beta$ for all $\bj\in \mc E_N$ with $w_\bj>0$. \medskip 

\noindent (2$\cc$) The following records some finer properties of Proposition~\ref{prop:logsum}:
\begin{enumerate}
\item [\rm (a)] The assumption $z_0\in \CNwni$ and the definition of $ K^{\bbeta,\bs w}_0 (t)/ [w_\bi K^{\bbeta,\bi}_0 (t)]$ for $\ms Z_t\notin \CNw$ in Proposition~\ref{prop:logsum} (1$\cc$) ensures the continuity of $t\mapsto   K^{\bbeta,\bs w}_0 (t)/ [w_\bi K^{\bbeta,\bi}_0 (t)]$ by the following:
\begin{itemize}
\item By Proposition~\ref{prop:BMntc}, $\P^\bi_{z_0}(\ms Z_t\in \CNwni,\;\forall\;t> 0)=1$.
\item Let $\{z_n\}\subset \CN$ be any sequence such that $z_n^\bj\neq 0$ for all $\bj\in \mc E_N$, $z_n^\bj\to z_\infty^\bj\in \Bbb C$ as $n\to\infty$, $z_\infty^\bi=0$, and $\lim_n (\log |z_n^\bj|)/(\log |z_n^\bi|)$ exists in $\R$ whenever $w_\bj>0$ and $z^\bj_\infty=0$. Then by the asymptotic representation \eqref{K00} of $K_0$ as $x\to 0$,
\[
\lim_{n\to\infty}\frac{\sum_{\bj\in \mc E_N}w_\bj K_0(\sqrt{2\beta_\bj}|z_n^\bj|)}{w_\bi K_0(\sqrt{2\beta_\bi}|z_n^\bi|)}=\sum\bigg\{\frac{w_\bj}{w_\bi}\lim_{n\to\infty}\frac{\log |z^\bj_n|}{\log |z^\bi_n|};w_\bj>0,  z_\infty^\bj=0\bigg\}.
\]
\end{itemize}
The context for $ K^{\bbeta,\bs w}_0 (t)/ [w_\bi K^{\bbeta,\bi}_0 (t)]$ is that $z^\bj_\infty\neq 0$ for all $\bj\neq \bi$. 
Hence, a ``removal'' of singularities in $K^{\bbeta,\bs w}_0 (t)$ and $w_\bi K^{\bbeta,\bi}_0 (t)$ occurs in the $\infty/\infty$-ratio $ K^{\bbeta,\bs w}_0 (t)/ [w_\bi K^{\bbeta,\bi}_0 (t)]$ whenever $\ms Z_t\notin \CNw$. 

\item [\rm (b)] On the other hand, for \emph{fixed} $t>0$, $w_\bi K^{\bbeta,\bi}_0 (t)$ and $K^{\bbeta,\bs w}_0 (t)$ are strictly positive by themselves $\P^\bi_{z_0}$-a.s. due to the existence of the probability density of $Z^\bi_t$ \cite[Theorem~2.1]{C:SDBG1-2}. 

\item [\rm (c)] Whereas $K_0(0)=\infty$ and $\d L_s^\bi$ is supported in $\{s\geq 0;Z^\bi_s=0\}$, the integrals in \eqref{def:Aring} are $\P^\bi_{z_0}$-a.s. finite since for all $\bj\in \mc E_N\setminus\{ \bi\}$ with $w_\bj>0$,
 $s\mapsto K_0(\sqrt{2\beta_\bj}|Z^\bj_s|)$, $s\geq 0$, are bounded on compacts by the choice $z_0\in \CNwni$ and Proposition~\ref{prop:BMntc}. \qed 
\end{enumerate}
\end{rmk}

\begin{proof}[Proof of Corollary~\ref{cor:superMG}]
The second equality in \eqref{superMG} can be seen by taking the exponentials of both sides of \eqref{RN} and rearranging. To see the martingale property of $\mathcal E_{z_0}^{\bbeta,\bs w,\bi}$, by the second equality of \eqref{superMG} and Novikov's criterion for exponential local martingales \cite[(1.15) Proposition, p.332]{RY-2}, it suffices to check that
\begin{align}\label{Nqvexp}
\E^{\bi}_{z_0}\left[\exp\left\{\lambda \la N^{\bbeta,\bs w,\bi}_0,N^{\bbeta,\bs w,\bi}_0\ra_{t\wedge T^{\bs w\setminus\{w_\bi\}}_{\eta} }\right\}\right]<\infty ,\quad \forall\;\lambda\in \R,\;t\geq 0.
\end{align}
To obtain \eqref{Nqvexp}, note that 
\begin{align}
\la N^{\bbeta,\bs w,\bi}_0,N^{\bbeta,\bs w,\bi}_0\ra_t&=\la \mathring{N}^{\bbeta,\bs w,\bi}_0,\mathring{N}^{\bbeta,\bs w,\bi}_0\ra_t+2\la \mathring{N}^{\bbeta,\bs w,\bi}_0,\widetilde{N}^{\bbeta,\bs w,\bi}_0\ra_t+\la \widetilde{N}^{\bbeta,\bs w,\bi}_0,\widetilde{N}^{\bbeta,\bs w,\bi}_0\ra_t\notag\\
&\leq \la \mathring{N}^{\bbeta,\bs w,\bi}_0,\mathring{N}^{\bbeta,\bs w,\bi}_0\ra_t+2\la \mathring{N}^{\bbeta,\bs w,\bi}_0,\mathring{N}^{\bbeta,\bs w,\bi}_0\ra_t^{1/2}
\la \widetilde{N}^{\bbeta,\bs w,\bi}_0,\widetilde{N}^{\bbeta,\bs w,\bi}_0\ra_t^{1/2}\notag\\
&\quad +\la \widetilde{N}^{\bbeta,\bs w,\bi}_0,\widetilde{N}^{\bbeta,\bs w,\bi}_0\ra_t\notag\\
&\leq 2\la \mathring{N}^{\bbeta,\bs w,\bi}_0,\mathring{N}^{\bbeta,\bs w,\bi}_0\ra_t+2\la \widetilde{N}^{\bbeta,\bs w,\bi}_0,\widetilde{N}^{\bbeta,\bs w,\bi}_0\ra_t,\label{Nqvexp:1}
\end{align}
where the first inequality uses the Kunita--Watanabe inequality \cite[(1.15) Proposition, p.126]{RY-2}, and the second equality uses the bound $2ab\leq a^2+b^2$ for all $a,b\in \R$. 
A similar argument shows that by the definition \eqref{def:Nwitilde} of $\widetilde{N}^{\bbeta,\bs w,\bi}_0$, 
\begin{align*}
\la \widetilde{N}^{\bbeta,\bs w,\bi}_0,\widetilde{N}^{\bbeta,\bs w,\bi}_0\ra_t&\leq C(N) \sum_{\bj\in \mc E_N\setminus\{\bi\}}\int_0^t \frac{w_{\bj}^2 \widehat{K}_1^{\bbeta,\bj}(s)^2}{|Z^\bj_s|^2 K^{\bbeta,\bs w}_0(s)^2}\d s= C(N)\sum_{\bj\in \mc E_N\setminus\{\bi\}}\int_0^t \frac{w_{\bj}^2 2\beta_\bj K_1^{\bbeta,\bj}(s)^2}{ K^{\bbeta,\bs w}_0(s)^2}\d s\\
&\leq C(N)\sum_{\scriptstyle \bj\in \mc E_N\setminus\{\bi\}\atop \scriptstyle w_\bj>0}\int_0^t \frac{w_{\bj}^2 2\beta_\bj K_1^{\bbeta,\bj}(s)^2}{w_\bj^2 K^{\bbeta,\bj}_0(s)^2}\d s,
\end{align*}
where the equaltiy uses the definition \eqref{def:K1bj} of $\widehat{K}^{\bbeta,\bj}_1$, and 
the last inequality uses the definition \eqref{def:Kbbetaw0} of $K^{\bbeta,\bs w}_0$. By the definition \eqref{def:Twi} of $\Twi_\eta$ and the asymptotic representations \eqref{K0infty} and \eqref{K1infty} of $K_0(\cdot)$ and $K_1(\cdot)$ as $x\to\infty$, the last inequality implies
\begin{align}
\la \widetilde{N}^{\bbeta,\bs w,\bi}_0,\widetilde{N}^{\bbeta,\bs w,\bi}_0\ra_{t\wedge \Twi_\eta}\leq C(N,\bbeta,\bs w,\eta)t,\quad \forall\;t\geq 0.\label{Nqvexp:2}
\end{align}
By \eqref{Nqvexp:1}, \eqref{Nqvexp} follows upon using the exponential moment condition in Proposition~\ref{prop:logsum} (4$\cc$) and \eqref{Nqvexp:2}. The proof is complete.
\end{proof}

\subsection{The first class of strong Markov processes}\label{sec:SDE1}
In this subsection, we begin the proof of Theorem~\ref{thm:main1}. From now on until before Section~\ref{sec:SDEcont}, we use the space $C_{\Bbb C^N}[0,\infty)$ of continuous functions from $[0,\infty)$ to $\Bbb C^N$ as the sample space, and  $\{\ms Z_t\}=\{Z^j_t\}_{1\leq j\leq N}$ refers to the coordinate process of $C_{\Bbb C^N}[0,\infty)$. By convention, $C_{\Bbb C^N}[0,\infty)$ is equipped with the topology of uniform convergence on compacts. Accordingly, we regard the probability measures $\P^\bi$ of the stochastic one-$\delta$ motions as probability measures on $C_{\Bbb C^N}[0,\infty)$ and set 
\[
\F_t^0\,\defeq\,\sigma(\ms Z_s;s\leq t),\quad 0\leq t\leq \infty.
\]
The following definition specifies the first class of strong Markov processes described in Section~\ref{sec:intro}.

\begin{defi}\label{def:SDE1ult}
Fix  $\bbeta\in (0,\infty)^{\mc E_N}$ and nonzero $\bs w\in \R_+^{\mc E_N}$.
For $z_0\in \CNw$, $\P^{\bbeta,\bs w}_{z_0}$ is a probability measure on $(C_{\Bbb C^N}[0,\infty),\B(C_{\Bbb C^N}[0,\infty)))$ defined as follows: for $0\leq F\in \B(C_{\Bbb C^N}[0,\infty))$, 
\begin{align}
\E_{z_0}^{\bbeta,\bs w}[F(\ms Z_{t};t\geq 0)]
&\,\defeq\,\sum_{\bj\in \mc E_N}\frac{w_\bj K^{\bbeta,\bj}_0(0)}{K^{\bbeta,\bs w}_0(0)}\E_{z_0}^{\bj}[\e^{-A^{\bbeta,\bs w,\bj}_0(T^{\bj}_0)}F(\ms Z_{t\wedge T^\bj_0};t\geq 0)].\label{def:Q1ult}
\end{align}
Here, $\CNw$ is defined in \eqref{def:CNw}, $A^{\bbeta,\bs w,\bj}_0(\cdot)$ is defined in \eqref{def:A}, and $T^\bj_\eta$ are stopping times defined in Proposition~\ref{prop:BMntc}. Moreover, since $\mathring{A}^{\bbeta,\bs w,\bj}_0(t)=0$ for all $t\leq T_0^\bj$ under $\P^{\bj}$, $\e^{-A^{\bbeta,\bs w,\bj}_0(T^{\bj}_0)}$ in \eqref{def:Q1ult} is replaceable by $\e^{-\wt{A}^{\bbeta,\bs w,\bj}_0(T^{\bj}_0)}$, where $\mathring{A}^{\bbeta,\bs w,\bj}_0(\cdot)$ and $\wt{A}^{\bbeta,\bs w,\bj}_0(\cdot)$ are defined in  \eqref{def:Aring} and \eqref{def:Atilde}. 
\end{defi}

Proposition~\ref{prop:SDE1} below is the main result of Section~\ref{sec:SDE1}. 

\begin{prop}\label{prop:SDE1}
For any $\bbeta\in (0,\infty)^{\mc E_N}$ and nonzero $\bs w\in \R_+^{\mc E_N}$, the following holds:
\begin{itemize}
\item [\rm (1$\cc$)]  For any $z_0\in \CNw$, $\P_{z_0}^{\bbeta,\bs w}$ in Definition~\ref{def:SDE1ult} indeed defines a probability measure. Moreover, for any $\bi\in \mc E_N$ with $w_\bi>0$ and $0\leq F\in \B(C_{\CN}[0,\infty))$,
\begin{align}\label{Pinhomo:id}
 \E^{\bbeta,\bs w}_{z_0}[F(\ms Z_{t\wedge T_0^{\bs w}};t\geq 0)\e^{A^{\bbeta,\bs w,\bi}_0(T^{\bs w}_0)};T^{\bs w}_0=T^\bi_0]=\frac{w_\bi K_0^{\bbeta,\bi}(0)}{K_0^{\bbeta,\bs w}(0)}\E^{\bi}_{z_0}[F(\ms Z_{t\wedge T_0^{\bi}};t\geq 0)].
 \end{align}
 In particular,
 \begin{align}\label{Pinhomo:idsum}
\sum_{\bi\in \mc E_N\atop w_\bi>0}\frac{w_\bi K_0^{\bbeta,\bi}(0)}{K_0^{\bbeta,\bs w}(0)}\E^{\bi}_{z_0}[\e^{-A^{\bbeta,\bs w,\bi}_0(T^{\bs w}_0)}]=1,\quad \forall\;z_0\in \CNw.
\end{align} 

\item [\rm (2$\cc$)]  The family of probability measures $\{\P^{\bbeta,\bs w}_{z_0};z_0\in  \Bbb C^N_{\bs w\,\sp}\}$ defines $\{\ms Z_t\}$ as a strong Markov process stopped at $T_0^{\bs w}$ with $\P^{\bbeta,\bs w}_{z_0}(\ms Z_0=z_0)=1$ and $\P^{\bbeta,\bs w}_{z_0}(T_0^{\bs w}<\infty)=1$ for all $z_0\in  \Bbb C^N_{\bs w\,\sp}$.

\item [\rm (3$\cc$)] For any $z_0\in \CN_{\bs w\,\sp}$,  
\begin{gather}
\P_{z_0}^{\bbeta,\bs w}\left(
\begin{array}{cc}
\exists\;\bj_1,\bj_2\in \mc E_N\mbox{ with }\bj_1\neq \bj_2\\
\&\; \exists\;0<t\leq T^{\bs w}_0\;\mbox{s.t.}\; Z^{\bj_1}_t=Z^{\bj_2}_t=0
\end{array}
\right)=0.\label{Z1:NTC}
\end{gather}

\item [\rm (4$\cc$)] Under $\P^{\bbeta,\bs w}_{z_0}$ for any $z_0\in \CNw$, $\{Z^j_t\}$ and $\{Z^\bj_t\}_{\bj\in \mc E_N}$ for $0\leq t \leq T_0^{\bs w}$  satisfy the SDEs in \eqref{SDE:Z1final} and~\eqref{SDE:Z1} for a family of independent two-dimensional standard Brownian motions $\{\wt{W}^j_t\}_{1\leq j\leq N}$. The Riemann-integral terms in these SDEs are absolutely integrable with
\begin{align}\label{Lp:Z1drift}
 \E_{z_0}^{\bbeta,\bs w}\biggl[\int_0^{t\wedge T_0^{\bs w}}\biggl(\frac{w_\bj\hK^{\bbeta,\bj}_1(s)}{K_0^{\bbeta,\bs w}(s)|Z^\bj_s|} \biggr)^p\d s\biggr]<\infty,\quad \forall\;t\geq 0,\;1\leq p<2,\;\bj\in \mc E_N.
\end{align}
\end{itemize}
\end{prop}

We will prove Proposition~\ref{prop:SDE1} (1$\cc$)--(3$\cc$) and Proposition~\ref{prop:SDE1} (4$\cc$) separately.
For (1$\cc$)--(3$\cc$), the major property to be proved is that Definition~\ref{def:SDE1ult} indeed defines $\P^{\bbeta,\bs w}_{z_0}$ as a probability measure since the presence of $\e^{-A^{\bbeta,\bs w,\bj}_0(T^{\bj}_0)}=\e^{-\wt{A}^{\bbeta,\bs w,\bj}_0(T^{\bj}_0)}$ should make it unclear why the right-hand side of \eqref{def:Q1} has a total mass equal to $1$. The trivial case, however, is when $\wt{A}_0^{\bbeta,\bs w,\bj}(T_0^\bj)=0$ under $\P^{\bj}_{z_0}$ for all $\bj\in \mc E_N$, so Remark~\ref{rmk:NTC} (1$\cc$) suggests the following definition:

\begin{defi}\label{def:homo}
Let $\bs \beta\in (0,\infty)^{\mc E_N}$ and let $\bs w\in \R_+^{\mc E_N}$ be nonzero.
The set $\bbeta$ is {\bf $\bs w$-homogeneous} if for some $\beta>0$, $\beta_\bj=\beta$ for all $\bj\in \mc E_N$ with $w_\bj>0$ and is {\bf $\bs w$-inhomogeneous} otherwise.  
\end{defi}

Lemma~\ref{lem:SDE1} below specifies the first set of properties of $\P_{z_0}^{\bbeta,\bs w}$ from Definition~\ref{def:SDE1ult} when $\bbeta$ is $\bs w$-homogeneous. 
 
\begin{lem}\label{lem:SDE1}
Let $\bs w\in \R_+^{\mc E_N}$ be nonzero and $z_0\in \CNw$. \medskip

\noindent {\rm (1$\cc$)} For any given $\bs w$-homogeneous
 $\bbeta\in (0,\infty)^{\mc E_N}$,  $\P^{\bbeta,\bs w}_{z_0}$ defines a probability measure on $(C_{\Bbb C^N}[0,\infty),\B(C_{\Bbb C^N}[0,\infty)))$ with 
\begin{align}
\E_{z_0}^{\bbeta,\bs w}[F(\ms Z_{t};t\geq 0)]
=\sum_{\bj\in \mc E_N}\frac{w_\bj K^{\bbeta,\bj}_0(0)}{K^{\bbeta,\bs w}_0(0)}\E_{z_0}^{\bj}[F(\ms Z_{t\wedge T^\bj_0};t\geq 0)]\label{def:Q1}
\end{align}
for any $0\leq F\in \B(C_{\Bbb C^N}[0,\infty))$. Moreover, with $T^{\bs w}_\eta$ defined in \eqref{def:Tw},
 \begin{align}
 &\P^{\bbeta,\bs w}_{z_0}(T_0^{\bs w}<\infty)=1,\label{T0:finite}\\
 &\P^{\bbeta,\bs w}_{z_0}(\lim_{\eta\searrow 0}\ua T^{\bs w}_\eta= T^{\bs w}_0)=1,\label{T0:approx}\\
& \P^{\bbeta,\bs w}_{z_0}(\ms Z_t=\ms Z_{t\wedge T^{\bs w}_{0}},\;\forall\;t\geq 0)=1,\label{T0:stopped}\\
& \P_{z_0}^{\bbeta,\bs w}\left(
\begin{array}{cc}
\exists\;\bj_1,\bj_2\in \mc E_N\mbox{ with }\bj_1\neq \bj_2\\
\&\; \exists\;0<t\leq T^{\bs w}_0\;\mbox{s.t.}\; Z^{\bj_1}_t=Z^{\bj_2}_t=0
\end{array}
\right)=0,\label{Z1:NTC0}\\
&\E_{z_0}^{\bbeta,\bs w}[F(\ms Z_{t\wedge T^{\bs w}_0};t\geq 0);T^{\bs w}_0=T_0^\bi]=\frac{w_\bi K^{\bbeta,\bi}_0(0)}{K^{\bbeta,\bs w}_0(0)}\E_{z_0}^{\bi}[F(\ms Z_{t\wedge T^\bi_0};t\geq 0)],\label{cov:3}
\end{align}
where $0\leq F\in \B(C_{\CN}[0,\infty))$.\medskip 

\noindent {\rm (2$\cc$)} Let 
 $\bbeta\in (0,\infty)^{\mc E_N}$ be  $\bs w$-homogeneous.
For any $\bi\in \mc E_N$ with $w_\bi>0$, $\eta\in (0,\min_{\bk:w_\bk>0}|z^\bk_0|)$, $(\F_t^0)$-stopping time $\tau$ satisfying $\tau\leq  T^{\bs w}_{\eta}$, $H\in \F^0_\tau$, and $0\leq F\in \B(C_{\Bbb C^N}[0,\infty))$,  we have
\begin{align}
\E_{z_0}^{\bbeta,\bs w}[F(\ms Z_{t\wedge \tau};t\geq 0);H]
&=\E^{(0)}_{z_0}\left[\frac{\e^{-\beta \tau} K^{\bbeta,\bs w}_0(\tau)}{K^{\bbeta,\bs w}_0(0) }F(\ms Z_{t\wedge \tau};t\geq 0);H\right]\label{eq:Tvep0} \\
&=\frac{w_\bi K^{\bbeta,\bi}_0(0)}{ K^{\bbeta,\bs w}_0(0)}\E_{z_0}^{\bi}\biggl[F(\ms Z_{t\wedge \tau};t\geq 0)\frac{ 
K^{\bbeta,\bs w}_0(\tau)
}{w_\bi K^{\bbeta,\bi}_0(\tau)};H\biggr].\label{eq:Tvep}
\end{align}

\noindent {\rm (3$\cc$)} Let $\bbeta\in (0,\infty)^{\mc E_N}$  and $\widetilde{\bbeta}\in (0,\infty)^{\mc E_N}$ be such that  $\widetilde{\beta}_\bj=\widetilde{\beta}$ for all $\bj$ with $w_\bj>0$,
 \begin{align}\label{cond:homo}
 \wt{\beta}>\max_{\bj:w_\bj>0}\beta_\bj-\min_{\bj:w_\bj>0}\beta_\bj,
 \end{align}
and $\bbeta$ is not required to be $\bs w$-homogeneous. Then for any bounded continuous $F$,
\begin{align}\label{inhomo:1}
\begin{split}
&\lim_{\eta\searrow 0}\E^{\widetilde{\bbeta},\bs w}_{z_0}[F(\ms Z_{t\wedge T^{\bs w}_\eta};t\geq 0)\mathcal E^{\bbeta,\bs w,\bi}_{z_0}(T^{\bs w}_\eta)\mathcal E^{\widetilde{\bbeta},\bs w,\bi}_{z_0}(T^{\bs w}_\eta)^{-1}]\\
&\quad\;=\E^{\widetilde{\bbeta},\bs w}_{z_0}[F(\ms Z_{t\wedge T^{\bs w}_0};t\geq 0)\mathcal E^{\bbeta,\bs w,\bi}_{z_0}(T^{\bs w}_0)\mathcal E^{\widetilde{\bbeta},\bs w,\bi}_{z_0}(T^{\bs w}_0)^{-1}].
\end{split}
\end{align}
Here, the right-hand side uses the following strictly positive random variable under $\P^{\widetilde{\bbeta},\bs w}_{z_0}$: 
\begin{align}\label{ER:lim}
\begin{split}
&\mathcal E^{\bbeta,\bs w,\bi}_{z_0}(T^{\bs w}_0)\mathcal E^{\widetilde{\bbeta},\bs w,\bi}_{z_0}(T^{\bs w}_0)^{-1}\\&\quad\;\,\defeq\,\lim_{t\nearrow T_0^{\bs w}}\left.\left(
\frac{w_\bi K_0^{\bbeta,\bi}(0)}{K_0^{\bbeta,\bs w}(0)}
\frac{\e^{-\widetilde{A}^{\bbeta,\bs w,\bi}_0(t)} K_0^{\bbeta,\bs w}(t)}{w_\bi K_0^{\bbeta,\bi}(t)}\right)\right/\left(\frac{w_\bi K_0^{\widetilde{\bbeta},\bi}(0)}{K_0^{\widetilde{\bbeta},\bs w}(0)}
\frac{K_0^{\widetilde{\bbeta},\bs w}(t)}{w_\bi K_0^{\widetilde{\bbeta},\bi}(t)}\right).
\end{split}
\end{align}
Moreover, for any $\eta\in (0,\min_{\bj;w_\bj>0}|z^\bj_0|)$, $(\F_t^0)$-stopping time $\tau$ with $\tau\leq T^{\bs w}_{\eta}$ and $0\leq F\in \B(C_{\Bbb C^N}[0,\infty))$,  
\begin{align}
&\quad\;\E^{\widetilde{\bbeta},\bs w}_{z_0}[F(\ms Z_{t\wedge \tau};t\geq 0)\mathcal E^{\bbeta,\bs w,\bi}_{z_0}(T^{\bs w}_0)\mathcal E^{\widetilde{\bbeta},\bs w,\bi}_{z_0}(T^{\bs w}_0)^{-1}]\notag\\
&=\E^{\bi}_{z_0}[F(\ms Z_{t\wedge \tau};t\geq 0)\mathcal E^{\bbeta,\bs w,\bi}_{z_0}(\tau)]\label{inhomo:2}\\
&=\E^{(0)}_{z_0}\Biggl[F(\ms Z_{t\wedge \tau};t\geq 0)\exp\Biggl\{-\int_0^{\tau} \sum_{\bj\in \mathcal E_N}\beta_\bj\frac{w_\bj  K^{\bbeta,\bj}_0(s)}{K^{\bbeta,\bs w}_0(s)}\d s\Biggr\}\frac{K^{\bbeta,\bs w}_0(\tau)}{K_0^{\bbeta,\bs w}(0)}\Biggr].\label{inhomo:3}
\end{align}
In particular, by passing $\eta\searrow 0$ for both sides of \eqref{inhomo:3} with $\tau=T_\eta^{\bs w}$,
 the limit in \eqref{inhomo:1} is independent of $\widetilde{\bbeta}$ and $\bi\in \mc E_N$ with $w_\bi>0$. 
\end{lem}

Note that \eqref{cov:3} is a particular case of \eqref{Pinhomo:id}, and the existence and strict positivity of the limit in \eqref{ER:lim} can be seen by using \eqref{Z1:NTC0} and a modification of Remark~\ref{rmk:NTC} (2$\cc$)-(a). Also, the proof of Proposition~\ref{prop:SDE1} (1$\cc$) below will use Lemma~\ref{lem:SDE1} (3$\cc$) to show that the measures defined by \eqref{def:Q1ult} in the inhomogeneous case can be obtained from the probability measures in the homogeneous case by changing measures, and so, indeed define probability measures.   \medskip

\begin{proof}[Proof of Lemma~\ref{lem:SDE1}]
{\bf (1$\cc$)} First,  \eqref{T0:finite} holds since for any $\bj\in \mc E_N$ with $w_\bj>0$, $\P^{\bj}_{z_0}(T^{\bs w}_{0}<\infty)\geq \P^{\bj }_{z_0}(T^\bj_0<\infty)=1$ \cite[(2.9), p.884]{DY:Krein-2}. Also, \eqref{T0:approx} can be deduced from the fact that $\{\bj\in \mc E_N;w_\bj>0\}$ is a finite set. The remaining properties can be obtained from
Proposition~\ref{prop:BMntc}: $\P^\bi_{z_0}(T^{\bs w}_0=T^{\bi}_0<T^\bj_0=\infty)=1$ for any $\bi\neq \bj$ with $w_\bi>0$ so \eqref{T0:stopped} holds by \eqref{def:Q1}, and \eqref{Z1:NTC0} and \eqref{cov:3} are immediate. \medskip  
  
\noindent {\bf (2$\cc$)}  Let $\eta\in (0,\min_{\bk:w_\bk>0}|z^\bk_0|)$, $F$ be bounded continuous, $\tau$  be an $(\F^0_t)$-stopping time with $\tau \leq T^{\bs w}_\eta$, and $H\in \F^0_\tau$. Note that 
for $\bj\in \mc E_N$ with $w_\bj>0$, since $\P^{\bj}_{z_0}(T^{\bs w}_\eta\leq T^\bj_\eta<T^\bj_0)=1$, applying 
 \cite[(2.7)]{C:SDBG1-2} as in \eqref{DY:comapp} gives
 \begin{align}
&\quad\;\E^{\bj}_{z_0}[ F(\ms Z_{t\wedge \tau};t\geq 0);H]
 =\E^{\bj}_{z_0}[F(\ms Z_{t\wedge \tau};t\geq 0) ;H\cap\{\tau<T_0^\bj\}]\notag\\
 &=\lim_{n\to\infty}\ua \sum_{k=1}^\infty  \E^{\bj}_{z_0}\left[F(\ms Z_{t\wedge \tau};t\geq 0);H\cap\left\{\frac{k-1}{2^n}\leq \tau<\frac{k}{2^n}<T^\bj_0\right\}\right]\notag\\
 &=
\lim_{n\to\infty}\ua \sum_{k=1}^\infty\E^{(0)}_{z_0}\left[\frac{\e^{-\beta_\bj \frac{k}{2^n}}K^{\bbeta,\bj}_0(\tfrac{k}{2^n})}{K_0^{\bbeta,\bj}(0)}F(\ms Z_{t\wedge \tau};t\geq 0);H\cap\left\{\frac{k-1}{2^n}\leq \tau<\frac{k}{2^n}\right\}\right]\notag\\
&=\E^{(0)}_{z_0}\left[\frac{\e^{-\beta_\bj \tau}K_0^{\bbeta,\bj}(\tau)}{K^{\bbeta,\bj}_0(0)}F(\ms Z_{t\wedge \tau};t\geq 0);H\right],\label{SDE1:BM-1000}
\end{align}
where the last equality uses dominated convergence since $\tau\leq T^{\bs w}_\eta$ by assumption and $K_0$ is bounded continuous on $[\sqrt{2\beta_\bj}\eta,\infty)=[\sqrt{2\beta}\eta,\infty)$. Note that \eqref{SDE1:BM-1000} does not require the homogeneity of $\bbeta$, though. 
  
Now, to see  \eqref{eq:Tvep0},  
we use \eqref{def:Q1} and \eqref{SDE1:BM-1000} in the same order:
\begin{align}
\E_{z_0}^{\bbeta,\bs w}[F(\ms Z_{t\wedge \tau};t\geq 0);H]
&=\sum_{\bj\in \mc E_N}\frac{w_\bj K^{\bbeta,\bj}_0(0)}{K^{\bbeta,\bs w}_0(0)}\E_{z_0}^{\bj}[ F(\ms Z_{t\wedge \tau};t\geq 0);H]\notag\\
&=\sum_{\bj\in \mc E_N}\frac{w_\bj}{K^{\bbeta,\bs w}_0(0)}\E^{(0)}_{z_0}\left[\e^{-\beta\tau} K^{\bbeta,\bj}_0(\tau)F(\ms Z_{t\wedge \tau};t\geq 0);H\right]\notag\\
&=\E^{(0)}_{z_0}\left[\frac{\e^{-\beta\tau} K^{\bbeta,\bs w}_0(\tau)}{K^{\bbeta,\bs w}_0(0) }F(\ms Z_{t\wedge \tau};t\geq 0);H\right],\label{SDE1:BM}
\end{align}
as required. Also, the second required equality \eqref{eq:Tvep} follows upon applying \eqref{SDE1:BM-1000} with $\bj=\bi$ for $\bi\in \mc E_N$ with $w_\bi>0$:
\begin{align*}
\E^{(0)}_{z_0}\left[\frac{\e^{-\beta\tau}K^{\bbeta,\bs w}_0(\tau)}{K^{\bbeta,\bs w}_0(0) }F(\ms Z_{t\wedge \tau};t\geq 0);H\right]
=\frac{w_\bi K^{\bbeta,\bi}_0(0)}{ K^{\bbeta,\bs w}_0(0)}\E_{z_0}^{\bi}\biggl[F(\ms Z_{t\wedge \tau};t\geq 0)\frac{ 
K^{\bbeta,\bs w}_0(\tau)
}{w_\bi K^{\bbeta,\bi}_0(\tau)};H\biggr].
\end{align*}
 
 \noindent {\bf (3$\cc$)} We first show \eqref{inhomo:1}.
Let us begin by making two observations for
\begin{align*}
R_\bi(t)&\,\defeq\, \mathcal E^{\bbeta,\bs w,\bi}_{z_0}(t)\mathcal E^{\widetilde{\bbeta},\bs w,\bi}_{z_0}(t)^{-1},
\quad t<T_0^{\bs w},
\end{align*}
under $\P^{\widetilde{\bbeta},\bs w}_{z_0}$. First, by Remark~\ref{rmk:NTC} (1$\cc$), 
\begin{align*}
0\leq R_\bi(t)&\leq C(\bbeta,\widetilde{\bbeta})\exp\left\{-\left(\min_{\bj:w_\bj>0}\beta_\bj-\beta_\bi\right)t\right\}\frac{K_0^{\bbeta,\bs w}(t)}{K_0^{\widetilde{\bbeta},\bs w}(t)}\frac{K_0^{\widetilde{\bbeta},\bi}(t)}{K_0^{\bbeta,\bi}(t)}\\
&\leq C(\bbeta,\widetilde{\bbeta})\exp\left\{-\left(\min_{\bj:w_\bj>0}\beta_\bj-\beta_\bi\right)t\right\},\quad t<T_0^{\bs w}, 
\end{align*}
where the second equality follows by modifying the proof of \eqref{bdd:Kratio}. Second, observe that $\E^{\widetilde{\bbeta},\bs w}_{z_0}[\e^{qT_0^{\bs w}}]<\infty$ for all $q<\widetilde{\beta}$ since $\E^{\tilde{\beta}\da,\bi}_{z_0}[\e^{qT^\bi_0}]<\infty$ for all $q<\widetilde{\beta}$ by the explicit formula of the PDF of $T^\bi_0$ \cite[(2.9)]{C:SDBG1-2}. These observations are enough to get \eqref{inhomo:1} by applying the dominated convergence theorem. 

Finally, to prove \eqref{inhomo:2} and \eqref{inhomo:3} for all $0\leq F\in \B(C_{[0,\infty)}(\Bbb C^N))$, it suffices to consider bounded continuous $F$ since $C_{[0,\infty)}(\CN)$ is a Polish space.
For the case of \eqref{inhomo:2}, we consider the following,  using \eqref{inhomo:1} and \eqref{eq:Tvep} in the same order: 
\begin{align*}
&\quad\;\E^{\widetilde{\bbeta},\bs w}_{z_0}[F(\ms Z_{t\wedge \tau};t\geq 0)\mathcal E^{\bbeta,\bs w,\bi}_{z_0}(T^{\bs w}_0)\mathcal E^{\widetilde{\bbeta},\bs w,\bi}_{z_0}(T^{\bs w}_0)^{-1}]\\
&=\lim_{\eta\searrow 0}\E^{\widetilde{\bbeta},\bs w}_{z_0}[F(\ms Z_{t\wedge \tau};t\geq 0)\mathcal E^{\bbeta,\bs w,\bi}_{z_0}(T^{\bs w}_\eta)\mathcal E^{\widetilde{\bbeta},\bs w,\bi}_{z_0}(T^{\bs w}_\eta)^{-1}]
=\lim_{\eta\searrow 0}\E^{\bi}_{z_0}[F(\ms Z_{t\wedge \tau};t\geq 0)\mathcal E^{\bbeta,\bs w,\bi}_{z_0}(T^{\bs w}_\eta)]\\
&=\E^{\bi}_{z_0}[F(\ms Z_{t\wedge \tau};t\geq 0)\mathcal E^{\bbeta,\bs w,\bi}_{z_0}(\tau)].
\end{align*}
Here, the last equality uses the optional stopping theorem \cite[(3.1) Theorem, p.68]{RY-2} since the boundedness of $K_0(\cdot)$ on $[\eta_0,\infty)$ for any $\eta_0>0$ implies that
$\{\mathcal E^{\bbeta,\bs w,\bi}_{z_0}(t\wedge T^{\bs w}_{\eta})\}$ is a bounded martingale for $\eta\in (0,\min_{\bj:w_\bj>0}|z^\bj_0|)$ (Corollary~\ref{cor:superMG}), and $\tau\leq T^{\bs w}_\eta$ by assumption. To get \eqref{inhomo:3}, we use the definition \eqref{superMG} of $\mathcal E^{\bbeta,\bs w,\bi}_{z_0}(\cdot)$ and \eqref{SDE1:BM-1000} with $\bj=\bi$, the latter not requiring $\bs w$-homogeneous $\bbeta$:
\begin{align*}
&\quad\;\E^{\bi}_{z_0}[F(\ms Z_{t\wedge \tau};t\geq 0)\mathcal E^{\bbeta,\bs w,\bi}_{z_0}(\tau)]\\
&=\E^{(0)}_{z_0}\left[F(\ms Z_{t\wedge \tau};t\geq 0)\frac{w_\bi K^{\bbeta,\bi}_0(0)}{K_0^{\bbeta,\bs w}(0)}
\frac{\e^{-\widetilde{A}_0^{\bbeta,\bs w,\bi}(\tau)}K_0^{\bbeta,\bs w}(\tau)}{w_\bi K^{\bbeta,\bi}_0(\tau)}
\frac{\e^{-\beta_\bi \tau}K_0^{\bbeta,\bi}(\tau)}{K_0^{\bbeta,\bi}(0)}
\right]\\
&=\E^{(0)}_{z_0}\Biggl[F(\ms Z_{t\wedge \tau};t\geq 0)\exp\biggl\{-\int_0^{\tau} \sum_{\bj\in \mathcal E_N}\beta_\bj\frac{w_\bj  K^{\bbeta,\bj}_0(s)}{K^{\bbeta,\bs w}_0(s)}\d s\biggr\}\frac{K^{\bbeta,\bs w}_0(\tau)}{K_0^{\bbeta,\bs w}(0)}\Biggr],
\end{align*}
where the last equality follows by some algebra, using the definition \eqref{def:Atilde} of $\widetilde{A}_0^{\bbeta,\bs w,\bi}(\cdot)$. 
\end{proof}

\begin{proof}[Proof of (1$\cc$)--(3$\cc$) for Proposition~\ref{prop:SDE1}]
\noindent {\bf (1$\cc$)} By Lemma~\ref{lem:SDE1}, it remains to consider the inhomogeneous case. 
For $\bs w$-inhomogeneous $\bbeta$, we first define an auxiliary probability measure $\widehat{\P}^{\bbeta,\bs w}_{z_0}$ on the space $(C_{\Bbb C^N}[0,\infty),\B(C_{\Bbb C^N}[0,\infty)))$ by
\begin{align}\label{inhomo:0}
\begin{split}
&\widehat{\E}^{\bbeta,\bs w}_{z_0}[F(\ms Z_{t\wedge T_0^{\bs w}};t\geq 0)]\\
&\quad \,\defeq\,\E^{\widetilde{\bbeta},\bs w}_{z_0}[F(\ms Z_{t\wedge T_0^{\bs w}};t\geq 0)\mathcal E^{\bbeta,\bs w,\bi}_{z_0}(T^{\bs w}_0)\mathcal E^{\widetilde{\bbeta},\bs w,\bi}_{z_0}(T^{\bs w}_0)^{-1}],\quad 0\leq F\in \B(C_{\Bbb C^N}[0,\infty)),
\end{split}
\end{align}
where $z_0\in \CNw$, and $\wt{\bbeta}$ is any choice satisfying the condition in Lemma~\ref{lem:SDE1} (3$\cc$). Note that by taking $F\equiv 1$ in  \eqref{inhomo:1}, $\widehat{\P}^{\bbeta,\bs w}_{z_0}$ is indeed a probability measure. Also,  for any $\eta\in (0,\min_{\bj;w_\bj>0}|z^\bj_0|)$, \eqref{inhomo:2} with $\tau=T_\eta^{\bs w}$ gives
\begin{align*}
\E^{\bi}_{z_0}[F(\ms Z_{t\wedge T^{\bs w}_\eta};t\geq 0)]&=\E^{\widetilde{\bbeta},\bs w}_{z_0}[F(\ms Z_{t\wedge T^{\bs w}_\eta};t\geq 0)\mathcal E^{\bbeta,\bs w,\bi}_{z_0}(T^{\bs w}_\eta)^{-1}
\mathcal E^{\bbeta,\bs w,\bi}_{z_0}(T^{\bs w}_0)\mathcal E^{\widetilde{\bbeta},\bs w,\bi}_{z_0}(T^{\bs w}_0)^{-1}]\\
&=\widehat{\E}^{\bbeta,\bs w}_{z_0}[F(\ms Z_{t\wedge T^{\bs w}_\eta};t\geq 0)\mathcal E^{\bbeta,\bs w,\bi}_{z_0}(T^{\bs w}_\eta)^{-1}]\notag
\end{align*}
where the last equality uses \eqref{inhomo:0}. Hence, by the definition \eqref{superMG} of $\mathcal E^{\bbeta,\bs w,\bi}_{z_0}(\cdot)$,
\begin{align*}
\frac{w_\bi K_0^{\bbeta,\bi}(0)}{K^{\bbeta,\bs w}_0(0)}\E^{\bi}_{z_0}[F(\ms Z_{t\wedge T^{\bs w}_\eta};t\geq 0)\e^{-A^{\bbeta,\bs w,\bi}_0(T^{\bs w}_\eta)-\beta_\bi T_\eta^{\bs w}}]
=\widehat{\E}^{\bbeta,\bs w}_{z_0}\left[F(\ms Z_{t\wedge T^{\bs w}_\eta};t\geq 0)\frac{\e^{-\beta_\bi T^{\bs w}_\eta}w_\bi K_0^{\bbeta,\bi}(T^{\bs w}_\eta)}{K^{\bbeta,\bs w}_0(T^{\bs w}_\eta)}
\right].
\end{align*}
Now, note that  $\widehat{\P}^{\bbeta,\bs w}_{z_0}(T_0^{\bs w}<\infty)=1$ by \eqref{T0:finite}, and $-A^{\bbeta,\bs w,\bi}_0(T^{\bs w}_\eta)-\beta_\bi T_\eta^{\bs w}\leq 0$ by \eqref{def:A}. Hence, by dominated convergence and Remark~\ref{rmk:NTC} (2$\cc$), passing $\eta\searrow 0$ for the leftmost side and the rightmost side of the foregoing display in the case of bounded continuous $F$ gives 
\begin{align*}
\frac{w_\bi K_0^{\bbeta,\bi}(0)}{K^{\bbeta,\bs w}_0(0)}\E^{\bi}_{z_0}[F(\ms Z_{t\wedge T^{\bs w}_0};t\geq 0)\e^{-A^{\bbeta,\bs w,\bi}_0(T^{\bs w}_0)-\beta_\bi T_0^{\bs w}}]
=\widehat{\E}^{\bbeta,\bs w}_{z_0}[F(\ms Z_{t\wedge T^{\bs w}_0};t\geq 0)\e^{-\beta_\bi T^{\bs w}_0};T_0^{\bs w }=T_0^\bi
].
\end{align*}
Since $\widehat{\P}^{\bbeta,\bs w}_{z_0}(T_0^{\bs w}<\infty)=1$, 
the last equality implies \eqref{Pinhomo:id} for all $0\leq F\in \B(C_{\CN}[0,\infty))$ with $\E^{\bbeta,\bs w}_{z_0}$ replaced by $\widehat{\E}_{z_0}^{\bbeta,\bs w}$. By \eqref{def:Q1ult}, we deduce $\P_{z_0}^{\bbeta,\bs w}=\widehat{\P}_{z_0}^{\bbeta,\bs w}$, so \eqref{Pinhomo:id} also holds. \medskip 

\noindent {\bf (2$\cc$)} By \eqref{T0:finite}, \eqref{T0:stopped}, and \eqref{inhomo:0}, it is immediate that for any $z_0\in  \Bbb C^N_{\bs w\,\sp}$,
$\{\ms Z_t\}$ under $\P_{z_0}^{\bbeta,\bs w}=\widehat{\P}_{z_0}^{\bbeta,\bs w}$ defines a process stopped at $T_0^{\bs w}$ such that $\P^{\bbeta,\bs w}_{z_0}(T_0^{\bs w}<\infty)=1$.
It remains to prove the required strong Markov property. For any bounded continuous $F_1,F_2$ and any $(\F_t^0)$-stopping time $\tau_0$, we apply \eqref{T0:approx}, \eqref{inhomo:3} with $\tau=T^{\bs w}_\eta$, the strong Markov property of Brownian motion, \eqref{inhomo:3}, and \eqref{T0:approx} again in the same order
 to get
\begin{align}
&\quad\;\E_{z_0}^{\bbeta,\bs w}[F_1(\ms Z_{t\wedge \tau_0};t\geq 0)F_2(\ms Z_{(\tau_0+t)\wedge T^{\bs w}_0};t\geq 0);\tau_0<T^{\bs w}_0]\notag\\
&=\lim_{\eta\searrow 0} \E_{z_0}^{\bbeta,\bs w}[F_1(\ms Z_{t\wedge \tau_0};t\geq 0)F_2(\ms Z_{(\tau_0+t)\wedge T^{\bs w}_\eta};t\geq 0);\tau_0< T^{\bs w}_\eta]\notag\\
&=\lim_{\eta\searrow 0}\E^{(0)}_{z_0}\Biggl[F_1(\ms Z_{t\wedge \tau_0};t\geq 0)F_2(\ms Z_{(\tau_0+t)\wedge T^{\bs w}_\eta};t\geq 0)\\
&\quad \times \exp\biggl\{-\int_0^{T_{\eta}^{\bs w}} \sum_{\bj\in \mathcal E_N}\beta_\bj\frac{w_\bj  K^{\bbeta,\bj}_0(s)}{K^{\bbeta,\bs w}_0(s)}\d s\biggr\}\frac{K^{\bbeta,\bs w}_0(T_{\eta}^{\bs w})}{K_0^{\bbeta,\bs w}(0)} ;\tau_0<T^{\bs w}_\eta\Biggr]\notag\\
&=\lim_{\eta\searrow 0}\E^{(0)}_{z_0}\left[F_1(\ms Z_{t\wedge \tau_0};t\geq 0)\exp\biggl\{-\int_0^{\tau_0} \sum_{\bj\in \mathcal E_N}\beta_\bj\frac{w_\bj  K^{\bbeta,\bj}_0(s)}{K^{\bbeta,\bs w}_0(s)}\d s\biggr\}\frac{K^{\bbeta,\bs w}_0(\tau_0)}{K_0^{\bbeta,\bs w}(0)}\right.\notag\\
&\quad \left.\times\E^{(0)}_{\ms Z_{\tau_0}}\left[F_2(\ms Z_{t\wedge T^{\bs w}_\eta};t\geq 0)\exp\biggl\{-\int_0^{T_{\eta}^{\bs w}} \sum_{\bj\in \mathcal E_N}\beta_\bj\frac{w_\bj  K^{\bbeta,\bj}_0(s)}{K^{\bbeta,\bs w}_0(s)}\d s\biggr\}\frac{K^{\bbeta,\bs w}_0(T_{\eta}^{\bs w})}{K_0^{\bbeta,\bs w}(0)}\right];\tau_0<T^{\bs w}_\eta\right]\notag\\
&=\lim_{\eta\searrow 0} \E_{z_0}^{\bbeta,\bs w}[F_1(\ms Z_{t\wedge \tau_0};t\geq 0)\E^{\bbeta,\bs w}_{\ms Z_{\tau_0}}[F_2(\ms Z_{t\wedge T^{\bs w}_\eta};t\geq 0)];\tau_0<T^{\bs w}_\eta]\notag\\
&=\E_{z_0}^{\bbeta,\bs w}[F_1(\ms Z_{t\wedge \tau_0};t\geq 0)\E^{\bbeta,\bs w}_{\ms Z_{\tau_0}}[F_2(\ms Z_{t\wedge T^{\bs w}_0};t\geq 0)];\tau_0<T^{\bs w}_0].\label{SMP:proof}
\end{align}
In more detail, the fourth equality has also used the fact that $\ms Z_{\tau_0}\in \CNw$ on $\{\tau_0<T^{\bs w}_\eta\}$ so that \eqref{inhomo:3} is indeed applicable. Since $\{\ms Z_t\}$ is stopped at time $T_0^{\bs w}$ under $\P^{\bbeta,\bs w}_{z_0}$, the last equality is enough to get the required strong Markov property.\medskip 

\noindent {\bf (3$\cc$)} We obtain \eqref{Z1:NTC} immediately from \eqref{Z1:NTC0} and \eqref{inhomo:0}.
\end{proof}

To prepare the proof of Proposition~\ref{prop:SDE1} (4$\cc$), we show the following lemma, which will also be used in Section~\ref{sec:SDE2}. The reader may recall the set $\CNwni$ defined in \eqref{def:CNw} and the convention in Proposition~\ref{prop:logsum} (1$\cc$). 

\begin{lem}\label{lem:Qtight}
Let $\bbeta\in (0,\infty)^{\mc E_N}$, and let $\bs w\in \R_+^{\mc E_N}$ be nonzero. For any 
 $\bi\in \mc E_N$ with $w_\bi>0$, $z_0\in\CNwni$,
$1\leq p<2$, and $0<t<\infty$, the following $L^1$-property holds for all $r\in \{s,t\}$:
\begin{align}
\E^{\bi}_{z_0}\Biggl[\Biggl(\,\sum_{\bj\in \mc E_N}\frac{w_\bj\hK^{\bbeta,\bj}_1(s)}{K_0^{\bbeta,\bs w}(s)|Z^\bj_s|} \Biggr)^p\mathcal E^{\bbeta,\bs w,\bi}_{z_0}(r)\Biggr]\in L^1([0,t],\d s).\label{eq:Qtight}
\end{align}
Also, the expectation in \eqref{eq:Qtight} with $r=s$ as a function of $s\in (0,\infty)$ is bounded on compacts in $ (0,\infty)$.
\end{lem}
\begin{proof}
By the supermartingale property of $\mathcal E^{\bbeta,\bs w,\bi}_{z_0}(\cdot)$ (Corollary~\ref{cor:superMG}), it suffices to prove the part of Lemma~\ref{lem:Qtight} using $r=s$. 
In this case, note that since $\widehat{K}_1(x)\,\defeq\,xK_1(x)$, for $\{x_\bj\}\in (0,\infty)^{\mc E_N}$, 
\begin{align}
\sum_{\bj\in \mc E_N}\frac{w_\bj \widehat{K}_1(\sqrt{2\beta_\bj}x_\bj)}{\sum_{\bk\in \mc E_N}w_\bk K_0(\sqrt{2\beta_\bk} x_\bk)x_\bj}&\leq 
\sum_{\bj:w_\bj>0}\frac{w_\bj\sqrt{2\beta_\bj} K_1(\sqrt{2\beta_\bj} x_\bj)}{w_\bj K_0(\sqrt{2\beta_\bj} x_\bj)}\notag\\
&\leq C(\bbeta,\bs w) \sum_{\bj\in \mc E_N}\frac{\1_{\{\sqrt{2\beta_\bj}x_\bj\leq 0.5\}}}{x_\bj K_0(\sqrt{2\beta_\bj} x_\bj)}+C(\bbeta,\bs w,N)\label{bdd:Kratio}
\end{align}
by the asymptotic representations \eqref{K10}--\eqref{K1infty} of $K_1$ as $x\to 0$ and $x\to\infty$ and the asymptotic representation \eqref{K0infty} of $K_0$ as $x\to\infty$. Also, recall the definitions of $\hK^{\bbeta,\bj}_1(\cdot)$,
$K^{\bbeta,\bj}_0(\cdot)$, $K_0^{\bbeta,\bs w}(\cdot)$, and $\mathcal E^{\bbeta,\bs w,\bi}_{z_0}(\cdot)$ 
in \eqref{def:K1bj}, \eqref{def:Kbbetaw0}, and \eqref{superMG}. Hence, for all $0<s<\infty$,
\begin{align}
&\quad\; \E^{\bi}_{z_0}\Biggl[\Biggl(\,\sum_{\bj\in \mc E_N}\frac{w_\bj\hK^{\bbeta,\bj}_1(s)}{K_0^{\bbeta,\bs w}(s)|Z^\bj_s|} \Biggr)^p\mathcal E^{\bbeta,\bs w,\bi}_{z_0}(s)\Biggr]\notag\\
&=\E^{\bi}_{z_0}\Biggl[\Biggl(\,\sum_{\bj\in \mc E_N}\frac{w_\bj\hK^{\bbeta,\bj}_1(s)}{K_0^{\bbeta,\bs w}(s)|Z^\bj_s|} \Biggr)^p\frac{w_\bi K_0^{\bbeta,\bi}(0)}{K_0^{\bbeta,\bs w}(0)}\frac{\e^{-\mathring{A}^{\bbeta,\bs w,\bi}_0(s)-\widetilde{A}^{\bbeta,\bs w,\bi}_0(s)}K_0^{\bbeta,\bs w}(s)}{w_\bi K_0^{\bbeta,\bi}(s)}\Biggr]\notag\\
&\leq C(\bbeta,\bs w,p,N)\notag\\
&\quad \times\E^{\bi}_{z_0}\Biggl[\Biggl(\,\sum_{\bj\in \mc E_N}\Biggl(\frac{\1_{\{\sqrt{2\beta_\bj}|Z^\bj_s|\leq 0.5\}}}{|Z^\bj_s|K_0(\sqrt{2\beta_\bj}|Z^\bj_s|)}\Biggr)^p+1\Biggr)\frac{w_\bi K_0^{\bbeta,\bi}(0)}{K_0^{\bbeta,\bs w}(0)}\frac{\e^{-\mathring{A}^{\bbeta,\bs w,\bi}_0(s)-\widetilde{A}^{\bbeta,\bs w,\bi}_0(s)} K_0^{\bbeta,\bs w}(s)}{w_\bi K_0^{\bbeta,\bi}(s)}\Biggr]\notag\\
\begin{split}\label{ineq:tight1000}
&\leq C(\bbeta,\bs w,p,N)\e^{C(\bbeta)t}\sum_{\bj\in \mc E_N}\E^{\bi}_{z_0}\Biggl[\Biggl(\frac{\1_{\{\sqrt{2\beta_\bj}|Z^\bj_s|\leq 0.5\}}}{|Z^\bj_s|K_0(\sqrt{2\beta_\bj}|Z^\bj_s|)}\Biggr)^p\frac{w_\bi K_0^{\bbeta,\bi}(0)}{K_0^{\bbeta,\bs w}(0)}\frac{K_0^{\bbeta,\bs w}(s)}{w_\bi K_0^{\bbeta,\bi}(s)}\Biggr]\\
&\quad +C(\bbeta,\bs w,p,N).
\end{split}
\end{align}
Here, the first inequality uses \eqref{bdd:Kratio} and then the inequality $(x+y)^p\leq C(p)(x^p+y^p)$ for all $x,y\geq 0$ multiple times. Also, in \eqref{ineq:tight1000}, $\e^{C(\bbeta)t}$ arises from Remark~\ref{rmk:NTC} (1$\cc$), and
 the last $C(\bbeta,\bs w,p,N)$ follows from the supermartingale property of $\mathcal E^{\bbeta,\bs w,\bi}_{z_0}(\cdot)$ (Corollary~\ref{cor:superMG}) since $\mathcal E^{\bbeta,\bs w,\bi}_{z_0}(0)=1$. To handle the sum on the right-hand side of \eqref{ineq:tight1000}, note that
\begin{align}
&\quad\;\E^{\bi}_{z_0}\left[\left(\frac{\1_{\{\sqrt{2\beta_\bj}|Z^\bj_s|\leq 0.5\}}}{|Z^\bj_s|K_0(\sqrt{2\beta_\bj}|Z^\bj_s|)}\right)^p\frac{w_\bi K_0^{\bbeta,\bi}(0)}{K_0^{\bbeta,\bs w}(0)}\frac{K_0^{\bbeta,\bs w}(s)}{w_\bi K_0^{\bbeta,\bi}(s)}\right]\notag\\
\begin{split}
&\leq C(\bs w)
\E^{\bi}_{z_0}\left[\left(\frac{\1_{\{\sqrt{2\beta_\bj}|Z^\bj_s|\leq 0.5\}}}{|Z^\bj_s|K_0(\sqrt{2\beta_\bj}|Z^\bj_s|)}\right)^p\right]\\
&\quad +C(\bs w)\sum_{ \bk:\bk\neq \bi}\E^{\bi}_{z_0}\left[\left(\frac{\1_{\{\sqrt{2\beta_\bj}|Z^\bj_s|\leq 0.5\}}}{|Z^\bj_s|K_0(\sqrt{2\beta_\bj}|Z^\bj_s|)}\right)^p\frac{K_0^{\bbeta,\bk}(s)}{K_0^{\bbeta,\bi}(s)}\right].\label{ineq:tight1}
\end{split}
\end{align}
Here, we use $w_\bi K_0^{\bbeta,\bi}(0)/K_0^{\bbeta,\bs w}(0)\leq 1$ for $z_0\in \CNw\cup \CNwi$ [Proposition~\ref{prop:logsum} (1$\cc$)] and write $K_0^{\bbeta,\bs w}(s)/[w_\bi K_0^{\bbeta,\bi}(s)]=\sum_\bk \{w_\bk K_0^{\bbeta,\bk}(s)/[w_\bi K_0^{\bbeta,\bi}(s)]\}$, where each summand is finite a.s. by Remark~\ref{rmk:NTC} (2$\cc$)-(b).

The first expectation on the right-hand side of \eqref{ineq:tight1} can be bounded as follows
by considering $\bj=\bi$ and $\bj\neq \bi$ separately. For $\bj=\bi$, it follows readily from \cite[(4.19)]{C:SDBG1-2} and the assumption $1\leq p<2$ that, for $\delta_{\ref{ineq:tight1-1}}=\delta_{\ref{ineq:tight1-1}}(\beta_\bi)$, 
\begin{align}\label{ineq:tight1-1}
\E^{\bi}_{z_0}\Biggl[\Biggl(\frac{\1_{\{\sqrt{2\beta_\bi}|Z^\bi_s|\leq 0.5\}}}{|Z^\bi_s|K_0(\sqrt{2\beta_\bi}|Z^\bi_s|)}\Biggr)^p\Biggr]\leq C(p,\beta_\bi,t)\left(\frac{\1_{\{s\leq 2\delta_{\ref{ineq:tight1-1}}\}}}{s\log^2 s}+1\right)\in L^1([0,t],\d s).
\end{align}
For $\bj\neq \bi$, $Z^\bj$ is a linear combination of $Z^\bi$ and an independent two-dimensional Brownian motion, with the contribution of the Brownian motion being nontrivial, which can be seen by using \eqref{SDE:Zbj-2}. Also, for a two-dimensional standard Brownian motion $\{W'_t\}$ independent of $ \{Z^\bi_t\}$, $z^1\in \Bbb C$, $\sigma\in [0,\infty)$ and $\varsigma\in (0,\infty)$, conditioning on $Z^\bi_s$ and applying the comparison theorem of SDEs \cite[2.18 Proposition, p.293]{KS:BM-2} to $\BES Q$ of dimension $2$ for varying initial conditions give
\begin{align} 
\E^{\bi}_{z_0}\left[\frac{1}{|\sigma  Z^\bi_s+z^1+\varsigma W'_s|^{p'}}\right]
\leq \E^{\bi}_0\left[\frac{1}{|\varsigma W'_s|^{p'}}\right]=C(\varsigma,p') s^{-p'/2},\;\; 0<s<\infty,\;0<p'<2.\label{ineq:tight1-1-1}
\end{align}
Note that the last expectation is finite only if $p'<2$, and the last equality holds by the Brownian scaling. By the linear structure of $Z^\bj$ just mentioned and \eqref{ineq:tight1-1-1}, we deduce that
\begin{align}\label{ineq:tight1-2}
\E^{\bi}_{z_0}\left[\left(\frac{\1_{\{\sqrt{2\beta_\bj}|Z^\bj_s|\leq 0.5\}}}{|Z^\bj_s|K_0(\sqrt{2\beta_\bj}|Z^\bj_s|)}\right)^p\right]\leq C(p,\beta_\bj)s^{-p/2},\quad 0<s<\infty,\; \bj\neq \bi.
\end{align}

To bound the remaining expectations on the right-hand side of \eqref{ineq:tight1}, we take $1<p_1,p_2,p_3<\infty$ such that $p_1^{-1}+p_2^{-1}+p_3^{-1}=1$ and $1<pp_1<2$. Then by H\"older's inequality, 
\begin{align}
&\quad\;\E^{\bi}_{z_0}\left[\left(\frac{\1_{\{\sqrt{2\beta_\bj}|Z^\bj_s|\leq 0.5\}}}{|Z^\bj_s|K_0(\sqrt{2\beta_\bj}|Z^\bj_s|)}\right)^p\frac{K_0^{\bbeta,\bk}(s)}{K_0^{\bbeta,\bi}(s)}\right]\notag\\
&\leq \E^{\bi}_{z_0}\left[\left(\frac{\1_{\{\sqrt{2\beta_\bj}|Z^\bj_s|\leq 0.5\}}}{|Z^\bj_s|K_0(\sqrt{2\beta_\bj}|Z^\bj_s|)}\right)^{pp_1}\right]^{1/p_1}\times \E^\bi_{z_0}[K_0^{\bbeta,\bk}(s)^{p_2}]^{1/p_2}\times \E^\bi_{z_0}[K_0^{\bbeta,\bi}(s)^{-p_3}]^{1/p_3}\notag\\
&\leq \begin{cases}
\displaystyle C(p,p_1,\beta_\bi,t)\left(\frac{\1_{\{s\leq 2\delta_{\ref{ineq:tight1-1}}\}}}{s\log^2 s}+1\right)^{1/p_1}\\
\vspace{-.4cm}\\
\quad \times \E^\bi_{z_0}[K_0^{\bbeta,\bk}(s)^{p_2}]^{1/p_2}\times \E^\bi_{z_0}[K_0^{\bbeta,\bi}(s)^{-p_3}]^{1/p_3}    ,& \bj=\bi,\\
\vspace{-.4cm}\\
\displaystyle C(p,p_1,\beta_\bj)s^{-p/2}\times  \E^\bi_{z_0}[K_0^{\bbeta,\bk}(s)^{p_2}]^{1/p_2}\times\E^\bi_{z_0}[K_0^{\bbeta,\bi}(s)^{-p_3}]^{1/p_3},&\bj\neq \bi,
\end{cases}\label{ineq:tight2-1}
\end{align}
for $0<s\leq t$
by \eqref{ineq:tight1-1} and \eqref{ineq:tight1-2}. To bound the right-hand side of \eqref{ineq:tight2-1}, note that for any $0<p_4<2$, the asymptotic representations \eqref{K00}--\eqref{K0infty} of $K_0$ as $x\to 0$ and $x\to\infty$ give 
\begin{align}
\forall\;\bk:\bk\neq \bi,\quad
\E^\bi_{z_0}[K_0^{\bbeta,\bk}(s)^{p_2}]^{1/p_2}& \leq C(p_2,\beta_\bk,p_4)\E^{\bi}_{z_0}[|Z^\bk_s|^{-p_4}]^{1/p_2}\notag\\
&\leq C(p_2,\beta_\bk,p_4)s^{-p_4/(2p_2)},\quad 0<s<\infty,\label{ineq:tight2-2}
\end{align}
where the last inequality can be seen by using the argument to get \eqref{ineq:tight1-2}. Also, by the decreasing monotonicity of $K_0$, \cite[Lemma~4.12]{C:SDBG1-2}, and the asymptotic representations \eqref{K00}--\eqref{K0infty}  of $K_0$ as $x\to 0$ and $x\to\infty$, 
\begin{align}
\E^\bi_{z_0}[K_0^{\bbeta,\bi}(s)^{-p_3}]^{1/p_3}&\leq \E^{(0)}_{z_0}[K_0(\sqrt{2\beta_\bi}|Z^\bi_s|)^{-p_3}]^{1/p_3}\notag\\
&\leq C(\beta_\bi,p_3)\e^{C(\beta_\bi,p_3)(|z_0^\bi|+s)},\quad 0<s<\infty. \label{ineq:tight2-3}
\end{align}
Putting \eqref{ineq:tight2-1}, \eqref{ineq:tight2-2} and \eqref{ineq:tight2-3} together shows that we can choose $\tilde{p}=\tilde{p}(p)<1$ such that 
\[
\E^{\bi}_{z_0}\left[\left(\frac{\1_{\{\sqrt{2\beta_\bj}|Z^\bj_s|\leq 0.5\}}}{|Z^\bj_s|K_0(\sqrt{2\beta_\bj}|Z^\bj_s|)}\right)^p\frac{K_0^{\bbeta,\bk}(s)}{K_0^{\bbeta,\bi}(s)}\right]\leq C(p,\beta_\bi,\beta_\bj,\beta_\bk,t)s^{-\tilde{p}},\quad\forall\;0<s\leq t,\;\bj\in \mc E_N,\;\bk\neq \bi.
\]

Finally, we apply \eqref{ineq:tight1-1}, \eqref{ineq:tight1-2}, and the last display to the right-hand side of \eqref{ineq:tight1}. By \eqref{ineq:tight1000}, this leads to \eqref{eq:Qtight} and the property that the expectation in \eqref{eq:Qtight} as a function of $s\in (0,\infty)$ is bounded on compacts in $ (0,\infty)$, both for $r=s$. 
The proof is complete.
\end{proof}

\begin{proof}[Proof of (4$\cc$) for Proposition~\ref{prop:SDE1}]
We can use a change of measures from $\P_{z_0}^\bi$ to $\P_{z_0}^{\bbeta,\bs w}$ to find the SDEs of $\{Z^j_{t}\}_{1\leq j\leq N}$ and $\{Z^\bj_{t}\}_{\bj\in \mc E_N}$ for $0\leq t\leq T^{\bs w}_\eta$ under $\P_{z_0}^{\bbeta,\bs w}$, for any $\eta\in (0,\max_{\bj:w_\bj>0}|z^\bj_0|)$ and $\bi\in \mc E_N$ with $w_\bi>0$. To carry this out, first, set
\begin{align}\label{def:Wcom}
\widetilde{W}^j_t\;\defeq\; W^j_t-\la W^j, N^{\bbeta,\bs w,\bi}_0\ra_t, \quad 1\leq j\leq N,\;0\leq t< T^{\bs w}_0,
\end{align}
where $W^j$ is the driving Brownian motion of $Z^j$ under $\P^{\bi}_{z_0}$ as specified in \eqref{def:ZSDE2-2}, and $N_0^{\bbeta,\bs w,\bi}$ is the continuous local martingale defined in \eqref{def:Nwi}. Then by linearity, 
\begin{align}\label{def:Wcomtilde}
\widetilde{W}^\bj_t= W^\bj_t-\la W^\bj, N^{\bbeta,\bs w,\bi}_0\ra_t, \quad \bj\in \mc E_N,\;0\leq t< T^{\bs w}_0.
\end{align}
Recall that by \eqref{inhomo:2} and \eqref{inhomo:0}, 
$\d \P^{\bbeta,\bs w}_{z_0}=\mathcal E^{\bbeta,\bs w,\bi}_{z_0}(T^{\bs w}_{\eta})\d \P^\bi_{z_0}$ on $\sigma(\ms Z_{t\wedge T^{\bs w}_\eta};t\geq 0)$. Hence, by Corollary~\ref{cor:superMG} and Girsanov's theorem \cite[(1.7) Theorem, p.329]{RY-2}, $\{\widetilde{W}^j_{t\wedge T_\eta^{\bs w}}\}$ and $\{\widetilde{W}^\bj_{t\wedge T_\eta^{\bs w}}\}$ are continuous local martingales under $\P^{\bbeta,\bs w}_{z_0}$ for all $1\leq j\leq N$ and $\bj\in\mc E_N$. 

The SDEs of  $\{Z^j_{t}\}_{1\leq j\leq N}$ and $\{Z^\bj_{t}\}_{\bj\in \mc E_N}$ for $0\leq t\leq T^{\bs w}_\eta$
under $\P_{z_0}^{\bbeta,\bs w}$ can be derived as follows. First, by the SDE of $\{Z^\bj_{t}\}$  under $\P^\bi_{z_0}$ in \eqref{SDE:Zbj-2} and by \eqref{def:Wcomtilde},
\begin{align}
Z^\bj_t&=Z_0^\bj-\frac{\sigma(\bj)\cdot \sigma(\bi)}{2}\int_0^t  \frac{ \K^{\bbeta,\bi}_{1}(s)}{ K^{\bbeta,\bi}_0(s)}
\biggl( \frac{1}{\overline{Z}^\bi_s}\biggr)\d s+\la W^\bj, N^{\bbeta,\bs w,\bi}_0\ra_t +\widetilde{W}^\bj_t\notag\\
&=Z_0^\bj-\frac{\sigma(\bj)\cdot \sigma(\bi)}{2}\int_0^t  \frac{ \K^{\bbeta,\bi}_{1}(s)}{ K^{\bbeta,\bi}_0(s)}
\biggl( \frac{1}{\overline{Z}^\bi_s}\biggr)\d s+\int_0^{t}\frac{\hK^{\bbeta,\bi}_1(s)[K^{\bbeta,\bs w}_0(s)-w_\bi K_0^{\bbeta,\bi}(s)]}{|Z^\bi_s|K^{\bbeta,\bi}_0(s)K^{\bbeta,\bs w}_0(s)}\d \la W^\bj,B^\bi\ra_s\notag\\
&\quad -\sum_{\bk\in \mc E_N\setminus\{ \bi\}}\int_0^{t}\frac{w_\bk\hK^{\bbeta,\bk}_1(s)}{|Z^\bk_s|K^{\bbeta,\bs w}_0(s)} \d \la W^\bj,B^\bk\ra_s+\widetilde{W}^\bj_t\notag\\
&=Z^\bj_0-\sum_{\bk\in \mc E_N}\frac{\sigma(\bj)\cdot \sigma(\bk)}{2} \int_0^t \frac{w_\bk\hK_1^{\bbeta,\bk}(s)}{K_0^{\bbeta,\bs w}(s)}\biggl(\frac{1}{\overline{Z}^\bk_s}\biggr)\d s+\widetilde{W}^\bj_t,\quad 0\leq t\leq T_\eta^{\bs w}.\label{SDE:1-100}
\end{align}
Here, the second equality uses \eqref{def:Nwi}--\eqref{def:Nwitilde}, and the last equality holds since by the definitions of $W^\bj$ and $B^\bk$ in \eqref{SDE:Zbj-2} and the formulas in \eqref{covar:UV-2}
we get
\begin{align*}
\d\la W^\bj, B^\bk\ra_t&=\frac{X^\bk_t\d \la U^\bj,U^\bk\ra_t}{|Z^\bk_t|}+\i\frac{Y^\bk_t\d \la V^\bj,V^\bk\ra_t}{|Z^\bk_t|}\\
&=\frac{\sigma(\bj)\cdot \sigma(\bk)}{2}\frac{X^\bk_t\d t+\i Y^\bk_t\d t}{|Z^\bk_t|}=\frac{\sigma(\bj)\cdot \sigma(\bk)}{2}\frac{Z^\bk_t\d t}{|Z^\bk_t|}=\frac{\sigma(\bj)\cdot \sigma(\bk)}{2}\biggl(\frac{|Z^{\bk}_t|}{\overline{Z}^\bk_t}\biggr)\d t .
\end{align*}
By \eqref{SDE:1-100}, $\{Z^\bj_{t}\}_{\bj\in \mc E_N}$ for $0\leq t\leq T^{\bs w}_\eta$ satisfy the SDEs in \eqref{SDE:Z1}.

As for the SDEs of $\{Z^j_{t}\}_{1\leq j\leq N}$ for $0\leq t\leq T^{\bs w}_\eta$, we can use \cite[Proposition~2.2]{C:SDBG4-2}. To this end, note that for the drift coefficients in \eqref{SDE:Z1}, we rewrite part of it as follows whenever $Z^\bk_s\neq 0$:
\[
 \frac{w_\bk\hK_1^{\bbeta,\bk}(s)}{2K_0^{\bbeta,\bs w}(s)}\left(\frac{1}{\overline{Z}^\bk_s}\right)= \frac{w_\bk\hK_1^{\bbeta,\bk}(s)}{2K_0^{\bbeta,\EN}(s)}\frac{1}{|Z_s^{\bk}|^2}Z^\bk_s.
\]
Accordingly, by \cite[Proposition~2.2]{C:SDBG4-2} with the identification 
\begin{align*}
a_\bk(s)&\,\equiv\, \frac{w_\bk\hK_1^{\bbeta,\bk}(s)}{2K_0^{\bbeta,\bs w}(s)}\frac{1}{|Z_s^{\bk}|^2}\1_{\{s\leq T^{\bs w}_\eta\}}= \frac{w_\bk\sqrt{2\beta_\bk} K_1(\sqrt{2\beta_\bk}|Z^\bk_s|)}{2K_0^{\bbeta,\bs w}(s)}\frac{1}{|Z_s^{\bk}|}\1_{\{s\leq T^{\bs w}_\eta\}}\\
&= \frac{w_\bk\sqrt{\beta_\bk} K_1(\sqrt{\beta_\bk}|Z^{k\prime}_s-Z^k_s|)}{K_0^{\bbeta,\bs w}(s)|Z_s^{k\prime}-Z^k_s|}\1_{\{s\leq T^{\bs w}_\eta\}}\quad \forall\;\bk=(k\prime,k)\in \mc E_N,
\end{align*}
$\{Z^j_{t}\}_{1\leq j\leq N}$ for $0\leq t\leq T^{\bs w}_\eta$ satisfy the equations in \eqref{SDE:Z1final} as soon as the following equations hold for some two-dimensional Brownian motion $\{\sqrt{N}\widetilde{W}^{\Sigma,\eta}_t\}$:
\begin{align}\label{BMsumBM}
\frac{1}{N}\sum_{j=1}^NZ^j_t=\frac{1}{N}\sum_{j=1}^N Z^j_0+\wt{W}^{\Sigma,\eta}_t,\quad \frac{1}{N}\sum_{j=1}^N\wt{W}^{j,\eta}_t=\wt{W}^{\Sigma,\eta}_t,\quad 0\leq t\leq T_\eta^{\bs w},
\end{align}
and $\{\wt{W}^{\Sigma,\eta}_t\}\ind \{\wt{W}^{\bj,\eta}_t\}_{\bj\in \mc E_N}$. Here,  we use the extension $\{\wt{W}^{j,\eta}_t\}_{\bj\in \mc E_N}$ of  $\{\wt{W}^{j}_t;0\leq t\leq T_\eta^{\bs w}\}_{\bj\in \mc E_N}$ beyond $T^{\bs w}_\eta$ such that $\{\wt{W}^{j,\eta}_t\}_{\bj\in \mc E_N}$ is family of independnet two-dimensional standard Brownian motions. [The existence of $\{\wt{W}^{j,\eta}_t\}_{\bj\in \mc E_N}$ uses
a concatenation with independent Brownian motions and L\'evy's characterization of Brownian motion \cite[3.16 Theorem, p.157]{KS:BM-2}.] To obtain the existence of such $\{\wt{W}^{\Sigma,\eta}_t\}$,
it suffices to show that $\{\wt{W}^{\Sigma,\eta}_t\}$
defined by the second equation in \eqref{BMsumBM} for all $t\geq 0$ satisfies all of the required properties. In this case, we only need to verify the first equality in \eqref{BMsumBM} for $0\leq t\leq T_\eta^{\bs w}$, but it holds upon noting that
\begin{align}\label{BMsumBM1}
\frac{1}{N}\sum_{j=1}^N Z^j_0+\wt{W}^{\Sigma,\eta}_t=\frac{1}{N}\sum_{j=1}^NZ^j_0+\frac{1}{N}\sum_{j=1}^N  W^j_t=\frac{1}{N}\sum_{j=1}^NZ^j_t,\quad 0\leq t\leq T^{\bs w}_\eta.
\end{align}
To see the first equality in \eqref{BMsumBM1}, note that $\{\frac{1}{N}\sum_{j=1}^N W^j_t\}\ind \{W^\bj_t\}$ for any $\bj\in \mc E_N$ under $\P^{\bi}_{z_0}$
so that $\la \frac{1}{N}\sum_{j=1}^N W^j_t,B^\bj\ra_{t\wedge T^{\bs w}_\eta}=0$ for all $\bj\in \mc E_N$ by the definition of $B^\bj$ in \eqref{def:Bj-2}. Hence, by \eqref{def:Wcom} and then \eqref{def:Nwi}--\eqref{def:Nwitilde}, the following holds for $0\leq t\leq T_\eta^{\bs w}$:
\[
\wt{W}^{\Sigma,\eta}_{t}=\frac{1}{N}\sum_{j=1}^NW^j_t-\left\langle\frac{1}{N}\sum_{j=1}^N W^j,N^{\bbeta,\bs w,\bi}\right\rangle_t=\frac{1}{N}\sum_{j=1}^NW^j_t.
\]
The second equality in \eqref{BMsumBM1} can be seen by using \eqref{def:ZSDE2-2}. We have proved that $\{Z^j_{t}\}_{1\leq j\leq N}$ for $0\leq t\leq T^{\bs w}_\eta$ satisfy the SDEs in \eqref{SDE:Z1final}.

The next step is to extend the SDEs of  $\{Z^j_{t}\}_{1\leq j\leq N}$ and $\{Z^\bj_{t}\}_{\bj\in \mc E_N}$ to $0\leq t\leq T^{\bs w}_0$ under $\P_{z_0}^{\bbeta,\bs w}$ and to specify the independent two-dimensional standard Brownian motions $\{W^j_t\}_{1\leq j\leq N}$. These tasks can first use the observation that
by \eqref{T0:approx} and the fact that $\{\ms Z_t\}$ is the coordinate process of $C_{\Bbb C^N}[0,\infty)$, \eqref{SDE:Z1final} and \eqref{SDE:Z1}  hold for all $0\leq t\leq T_0^{\bs w}$ by passing $\eta\searrow 0$ \emph{as soon as} \eqref{Lp:Z1drift} holds, in which case we define $\widetilde{W}^j_{t\wedge T_0^{\bs w}}$ as the limit of $\widetilde{W}^j_{t\wedge T_\eta^{\bs w}}$ as $\eta\searrow 0$. 

We prove \eqref{Lp:Z1drift} now. Let $1\leq p<2$ and $\bi\in \mc E_N$ such that $w_\bi>0$.
Then for all $\eta\in(0,\min_{\bj:w_\bj>0}|z_0^\bj|)$, we have
\begin{align}
  \E_{z_0}^{\bbeta,\bs w}\left[\int_0^{t\wedge T^{\bs w}_\eta}\left(\frac{w_\bj\hK^{\bbeta,\bj}_1(s)}{K_0^{\bbeta,\bs w}(s)|Z^\bj_s|} \right)^p\d s\right]
&= \int_0^t \E_{z_0}^{\bbeta,\bs w}\left[\left(\frac{w_\bj\hK^{\bbeta,\bj}_1(s)}{K_0^{\bbeta,\bs w}(s)|Z^\bj_s|} \right)^p\1_{\{s\leq t\wedge T_\eta^{\bs w}\}}\right]\d s\notag\\
&= \int_0^t  \E^{\bi}_{z_0}\left[\left(\frac{w_\bj\hK^{\bbeta,\bj}_1(s)}{K_0^{\bbeta,\bs w}(s)|Z^\bj_s|} \right)^p\1_{\{s\leq t\wedge T_\eta^{\bs w}\}}
\mathcal E^{\bbeta,\bs w,\bi}_{z_0}(t\wedge T^{\bs w}_\eta)\right]\d s\notag\\
&\leq \int_0^t  \E^{\bi}_{z_0}\left[\left(\frac{w_\bj \hK^{\bbeta,\bj}_1(s)}{K_0^{\bbeta,\bs w}(s)|Z^\bj_s|} \right)^p\1_{\{s\leq t\wedge T_\eta^{\bs w}\}}\mathcal E^{\bbeta,\bs w,\bi}_{z_0}(s)\right]\d s\notag\\
&\leq \int_0^t  \E^{\bi}_{z_0}\left[\left(\frac{w_\bj \hK^{\bbeta,\bj}_1(s)}{K_0^{\bbeta,\bs w}(s)|Z^\bj_s|} \right)^p\mathcal E^{\bbeta,\bs w,\bi}_{z_0}(s)\right]\d s<\infty.\label{SDE1:AF}
\end{align}
Here, the second equality uses \eqref{inhomo:2} with $\tau=t\wedge T^{\bs w}_\eta$; the first inequality follows from the supermartingale property of $\mathcal E^{\bbeta,\bs w,\bi}_{z_0}(\cdot)$ (Corollary~\ref{cor:superMG}) since $\{s\leq t\wedge T_\eta^{\bs w}\}\in \F_s^0$; the last inequality uses \eqref{eq:Qtight} for $r=s$. Since the right-hand side of \eqref{SDE1:AF} is independent of $\eta\in (0,\min_{\bj:w_\bj>0}|z_0^\bj|)$, \eqref{SDE1:AF}  implies \eqref{Lp:Z1drift} by passing $\eta\searrow 0$ and using \eqref{T0:approx} and Fatou's lemma. 

Finally, we show how $\{\widetilde{W}^j_{t\wedge T_0^{\bs w}}\}_{1\leq j\leq N}$ extends to a family of independent two-dimensional standard Brownian motions $\{\widetilde{W}^j_t\}_{1\leq j\leq N}$. Note that by \eqref{def:Wcom},
$\widetilde{U}^j\,\defeq\,\Re( \widetilde{W}^j)$ and  $\widetilde{V}^j\,\defeq\,\Im( \widetilde{W}^j)$ satisfy the following properties: \begin{align}\label{covar:UVtilde}
\la \wt{U}^i,\wt{U}^j\ra_t=\la \wt{V}^i,\wt{V}^j\ra_t=\delta_{ij}t,\quad 0\leq t<T_0^{\bs w},
\end{align}
where $\delta_{ij}$ are Kronecker's deltas, and, by It\^{o}'s formula,
\begin{align}\label{UV:qvbdd}
\left.
\begin{array}{ll}
\Big\{\Big((\widetilde{U}^j_{ t\wedge T_\eta^{\bs w}})^2- t\wedge T_\eta^{\bs w},(\widetilde{V}^j_{ t\wedge T_\eta^{\bs w}})^2- t\wedge T_\eta^{\bs w}\Big)_{j=1}^N; t\geq 0\Big\}\\
\mbox{is a vector continuous local martingale under $\P^{\bbeta,\bs w}_{z_0}$}.
\end{array}
\right.
\end{align}
By \eqref{covar:UVtilde}, $\{(\wt{U}^j_{t\wedge T^{\bs w}_\eta},\wt{V}^j_{t\wedge T^{\bs w}_\eta})_{j=1}^N;t\geq 0\}$ is a vector martingale. Also, by Fatou's lemma, \eqref{UV:qvbdd} implies that 
\begin{align}\label{UV:qvbdd1}
\max_{1\leq j\leq N}\max\{\E^{\bbeta,\bs w}_{z_0}[(\widetilde{U}^j_{ t\wedge T_\eta^{\bs w}})^2 ],\E^{\bbeta,\bs w}_{z_0}[(\widetilde{V}^j_{ t\wedge T_\eta^{\bs w}})^2 ]\}\leq  t,\quad \forall \; t\geq 0,\;\eta\in \left(0,\min_{\bj:w_\bj>0}|z^\bj_0|\right).
\end{align}
Hence, the uniform integrability of the family of random variables $\widetilde{U}^j_{ t\wedge T_\eta^{\bs w}}$ and $\widetilde{V}^j_{ t\wedge T_\eta^{\bs w}}$ holds. We conclude that $\{\widetilde{W}^j_{t\wedge T_0^{\bs w}}\}_{\bj\in \mc E_N}$ is a family of continuous martingales under $\P^{\bbeta,\bs w}_{z_0} $ such that \eqref{covar:UVtilde} holds up to and including $t=T_0^{\bs w}$. This is enough to get the family $\{\widetilde{W}^j_t\}_{1\leq j\leq N}$ of independent two-dimensional standard Brownian motions. The proof is complete.
\end{proof}

\subsection{The second class of strong Markov processes}\label{sec:SDE2}
In this subsection, we construct the second class of strong Markov processes described in Section~\ref{sec:intro}.
This class uses initial conditions in $\CNwni=\CNw\cup \,\CNwi $; recall the sets defined in \eqref{def:CNw}.
The preliminary forms of the required probability measures are defined as follows. Let $\bs w\in \R_+^{\mc E_N}$ be nonzero with $\#\{\bj;w_\bj>0\}\geq 2$, $\bi\in \mc E_N$ with $w_\bi>0$, $z_0\in \CNwni$, and $\eta\in (0,\min_{\bj:w_\bj>0,\bj\neq \bi}|z_0^\bj|)$. Then we can define probability measures $\P^{\bbeta,\bs w,\bi,\eta}_{z_0}$ on the measurable space $(C_{\Bbb C^N}[0,\infty),\B(C_{\Bbb C^N}[0,\infty)))$ such that $\P^{\bbeta,\bs w,\bi,\eta}_{z_0}$ are uniquely determined by the following expectations, using the martingale property stated in Corollary~\ref{cor:superMG}: for $0\leq F\in \B(C_{\Bbb C^N}[0,\infty))$ and $0<t_0<\infty$,
\begin{align}\label{def:Z2SDE}
\begin{split}
\E^{\bbeta,\bs w,\bi,\eta}_{z_0}[F(\ms Z_{t\wedge t_0};t\geq 0)]\,\defeq\,\E^{\bi}_{z_0}\left[F\left(\ms Z_{t\wedge t_0\wedge \Twi_\eta};t\geq 0\right)\mc E^{\bbeta,\bs w,\bi}_{z_0}\left(t_0\wedge \Twi_\eta\right) \right],
\end{split}
\end{align}
where $\Twi_\eta$ and $\mc E^{\bbeta,\bs w,\bi}_{z_0}(\cdot)$ are defined in \eqref{def:Twi} and \eqref{superMG}. Since $(C_{\Bbb C^N}[0,\infty),\B(C_{\Bbb C^N}[0,\infty)))$ serves as the underlying space, \eqref{def:Z2SDE} allows the definition of $\P^{\bbeta,\bs w,\bi,\eta}_{z_0}$ on the whole space $(C_{\Bbb C^N}[0,\infty),\B(C_{\Bbb C^N}[0,\infty)))$, not just on the measurable space generated by the set of $\CN$-valued continuous functions defined on $[0,t_0]$ for every $0<t_0<\infty$ \cite[1.3.5 Theorem, p.34]{SV-2}. 

Furthermore, the probability measures $\P_{z_0}^{\bbeta,\bs w,\bi,\eta}$ enjoy consistency in $\eta$, and the extensions to $\eta\searrow 0$ are the probability measures we look for to specify the second class of strong Markov processes. Here, the consistency in $\eta$ refers to the fact that for any $0<\eta_1<\eta_2<\min_{\bj:w_\bj>0,\bj\neq \bi}|z_0^\bj|$,  the uniqueness of $\P^{\bbeta,\bs w,\bi,\eta}_{z_0}$ by \eqref{def:Z2SDE} gives
\begin{align}\label{eq:consistency}
\E^{\bbeta,\bs w,\bi,\eta_1}_{z_0}[F(\ms Z_{t\wedge \Twi_{\eta_2}};t\geq 0)]=\E^{\bbeta,\bs w,\bi,\eta_2}_{z_0}\left[F(\ms Z_t;t\geq 0) \right],\;\;\forall\; 0\leq F\in \B(C_{\Bbb C^N}[0,\infty)).
\end{align}
On the other hand, proving the consistent extensions of $\P^{\bbeta,\bs w,\bi,\eta}_{z_0}$ to $\eta\searrow 0$ falls within the scope of constructing the exit measure of a supermartingale \cite{Follmer:exit-2,IW:Mult-2}. Roughly speaking, the central technical condition we need to verify is the continuous extensions to the time $\Twi_0$, if finite, of $\{\ms Z_t;t<\Twi_0\}$, since $\Twi_\eta\nearrow \Twi_0$ as $\eta\searrow 0$ as in \eqref{T0:approx}. This condition is handled in the proof of Proposition~\ref{prop:SDE2} (1$\cc$) below. We also summarize all the properties we need in this proposition. 

\begin{prop}\label{prop:SDE2}
Let $\bbeta\in (0,\infty)^{\mc E_N}$, $\bs w\in \R_+^{\mc E_N}$ with $\#\{\bj;w_\bj>0\}\geq 2$, and $\bi\in \mc E_N$ with $w_\bi>0$. Then the following properties hold.
\begin{itemize}
\item [\rm (1$\cc$)] For any $z_0\in \CNwni$, there exists a unique probability measure $\P^{\bbeta,\bs w,\bi}_{z_0}$ defined on $(C_{\Bbb C^N}[0,\infty),\B(C_{\Bbb C^N}[0,\infty)))$ such that
 $\{\ms Z_t\}$ stops at time $\Twi_0$ under $\P^{\bbeta,\bs w,\bi}_{z_0}$, and
\begin{align}\label{def:Q2}
\begin{aligned}
&\E^{\bbeta,\bs w,\bi}_{z_0}[F(\ms Z_{t\wedge \Twi_{\eta}};t\geq 0)]\\
& =
\E^{\bbeta,\bs w,\bi,\eta}_{z_0}[F(\ms Z_t;t\geq 0) ],\quad\forall\; 0\leq F\in \B(C_{\Bbb C^N}[0,\infty)),\;\eta\in \left(0,\min_{\bj:\bj\neq \bi;w_\bj>0}|z^\bj_0|\right).
\end{aligned}
\end{align}

\item [\rm (2$\cc$)] Under $\P^{\bbeta,\bs w,\bi}_{z_0}$ for any $z_0\in \CNwni$, $\{Z^j_t\}_{1\leq j\leq N}$ and $\{Z^\bj_t\}_{\bj\in \mc E_N}$ for $0\leq t \leq \Twi_0$  satisfy the SDEs in \eqref{SDE:Z1final} and~\eqref{SDE:Z1}. The Riemann-integral terms in these SDEs are absolutely integrable with
\begin{align}\label{Lp:Z2drift}
 \E_{z_0}^{\bbeta,\bs w,\bi}\biggl[\int_0^{t\wedge\Twi_0}\biggl(\frac{w_\bj\hK^{\bbeta,\bj}_1(s)}{K_0^{\bbeta,\bs w}(s)|Z^\bj_s|} \biggr)^p\d s\biggr]<\infty,\quad \forall\;t\geq 0,\;1\leq p<2,\;\bj\in \mc E_N.
\end{align}

\item [\rm (3$\cc$)] For any $z_0\in \CNwni$, $0\leq t_0<\infty$, and $0\leq F\in \B(C_{\Bbb C^N}[0,\infty))$,
\begin{align}
\begin{aligned}\label{def:Q21}
\E^{\bbeta,\bs w,\bi}_{z_0}\left[F(\ms Z_{t\wedge t_0};t\geq 0);t_0<\Twi_0\right] =\E^{\bi}_{z_0}\left[F(\ms Z_{t\wedge t_0};t\geq 0)\mathcal E^{\bbeta,\bs w,\bi}_{z_0}(t_0)
\right].
\end{aligned}
\end{align}
The same identity holds if we replace $t_0<\Twi_0$ on the left-hand side by $t_0\leq \Twi_0$. 

\item [\rm (4$\cc$)]
For any $z_0\in \CNw$, $\eta\in (0,\min_{\bj:w_\bj>0}|z_0^{\bj}|)$, $(\F_t^0)$-stopping time $\tau$ with $\tau\leq T^{\bs w}_0$, 
 and $0\leq F\in \B(C_{\Bbb C^N}[0,\infty))$,
\begin{align}\label{def:Q21000}
\E^{\bbeta,\bs w,\bi}_{z_0}\left[F(\ms Z_{t\wedge \tau};t\geq 0)\right]=\E^{\bbeta,\bs w}_{z_0}[F(\ms Z_{t\wedge \tau};t\geq 0)].
\end{align}

\item [\rm (5$\cc$)] For any $z_0\in \CNwni$, $\P^{\bbeta,\bs w,\bi}_{z_0}(\Twi_0<\infty)=1$ if $\beta_\bi=\min_{\bj:w_\bj>0}\beta_\bj$.  

\item [\rm (6$\cc$)] The family of probability measures 
$\{\P^{\bbeta,\bs w,\bi}_{z_0};z_0\in \CNwni\}$ defines $\{\ms Z_t\}$ 
as a strong Markov process stopped at $\Twi_0$.
\item [\rm (7$\cc$)] For any $z_0\in \CNwni$, it holds that
\begin{align*}
&\eqspace\P^{\bbeta,\bs w,\bi}_{z_0}\Big(\Twi_0<\infty, Z^\bi_{\Twi_0}\neq 0\;\;\&\;\;\exists\;!\; \bj\in\mc E_N\setminus\{\bi\}\mbox{ s.t. }w_\bj>0\;\&\; Z^\bj_{\Twi_0}=0\Big)\\
&=\P^{\bbeta,\bs w,\bi}_{z_0}\left(\Twi_0<\infty\right).
\end{align*}
 \end{itemize}
\end{prop}

\begin{rmk}\label{rmk:SDE2}
(1$\cc$) By Proposition~\ref{prop:SDE1} (2$\cc$) and \eqref{def:Q21000}, $\P^{\bbeta,\bs w,\bi}_{z_0}$ defined by \eqref{def:Z2SDE} for $z_0\in \CNw$ extends $\P^{\bbeta,\bs w}_{z_0}$ from Definition~\ref{def:SDE1ult}. The longer time $\Twi_0$ to stop is in use under $\P^{\bbeta,\bs w,\bi}_{z_0}$ since $\Twi_0\geq T^{\bs w}_0$. We also allow the additional set of initial conditions $\CNwi$ now. \medskip 

\noindent 
{\rm (2$\cc$)} In contrast to the others, Proposition~\ref{prop:SDE1} (5$\cc$) requires $\beta_\bi=\min_{\bj:w_\bj>0}\beta_\bj$. Presently, we do not know if $\P^{\bbeta,\bs w,\bi}_{z_0}(\Twi_0<\infty)=1$ for all the eligible $(\bbeta,\bs w,\bi)$, namely, all $(\bbeta,\bs w,\bi)$ such that $\bbeta\in (0,\infty)^{\mc E_N}$, $\bs w\in \R_+^{\mc E_N}$ with $\#\{\bj;w_\bj>0\}\geq 2$, and $\bi\in \mc E_N$ with $w_\bi>0$. \qed 
\end{rmk}
 
The proof of Proposition~\ref{prop:SDE2} (7$\cc$) is more involved than those of (1$\cc$)--(6$\cc$) and requires some additional definitions to fix ideas. We settle (1$\cc$)--(6$\cc$) first. \medskip 

\begin{proof}[Proof of Proposition~\ref{prop:SDE2} (1$\cc$)--(6$\cc$)]
\mbox{}\medskip 

\noindent {\bf (1$\cc$) and (2$\cc$)} 
Let us start the proof of (1$\cc$) first; the proof of (2$\cc$) will merge later. For the proof of (1$\cc$), it follows from continuity that the proofs of the existence part and \eqref{def:Q2}
 suffice. To this end, we first use a consistency theorem 
to construct a probability measure $\widehat{\P}^{\bbeta,\bs w,\bi}_{z_0}$ accommodating $\P_{z_0}^{\bbeta,\bs w,\bi,\eta}$ for any $\eta\in (0,\min_{\bj:\bj\neq \bi, w_\bj>0}|z_0^\bj|)$. The informal idea of a construction is to subdivide $\{\ms Z_t;t<\Twi_0\}$ into pieces of continuous paths according to the stopping times $\Twi_{(n_0+j)^{-1}}$, $j\in \Bbb N$, 
where we fix an integer $n_0\geq 1$ such that $0<n_0^{-1}<\min_{\bj:\bj\neq \bi, w_\bj>0}|z_0^\bj|$. Then the construction is complete by transferring those pieces of continuous paths to the probability space under $\widehat{\P}^{\bbeta,\bs w,\bi}_{z_0}$ and assembling them in the natural way. 

Specifically, by adding an extra point $\partial$ to $\CN$ and using \eqref{eq:consistency}, we can construct $\widehat{\P}_{z_0}^{\bbeta,\bs w,\bi}$ on $(C_{\CN\cup \{\partial\}}[0,\infty)^{\Bbb N},\B(C_{\CN\cup \{\partial\}}[0,\infty))^{\Bbb N})$ such that with $\widehat{\E}_{z_0}^{\bbeta,\bs w,\bi}\,\defeq\,\E^{\widehat{\P}_{z_0}^{\bbeta,\bs w,\bi}}$,
\begin{align}\label{eq:consistency0}
\begin{aligned}
&\eqspace\widehat{\E}_{z_0}^{\bbeta,\bs w,\bi}\left[F\left(\{\ms Z^{(1)}_t;t\geq 0\},\cdots,\{\ms Z^{(m)}_t;t\geq 0\}\right)\right]\\
&=\E^{\bbeta,\bs w,\bi,(n_0+m)^{-1}}_{z_0}\left[
F\left(\{\ms Z^{(1)*}_t;t\geq 0\},\cdots,\{\ms Z^{(m)*}_t;t\geq 0\}\right)
\right],\\
&\quad\quad \forall\;m\in \Bbb N,\; 0\leq F\in \B(C_{\Bbb C^N}[0,\infty)^m),
\end{aligned}
\end{align}
where $\{\ms Z^{(j)}_t\}$ maps to the coordinate process of the $j$-th copy of $C_{\CN}[0,\infty)$, and
\begin{align}\label{def:Zst}
\ms Z^{(j)*}_t\,\defeq\,
\begin{cases}
 \ms Z_{t\wedge \Twi_{n_0^{-1}}},&\mbox{if }j\geq 1,\\
 \ms Z_{\big(\Twi_{(n_0+j-2)^{-1}}+t\big)\wedge \Twi_{(n_0+j-1)^{-1}}},&\mbox{if }\Twi_{(n_0+j-2)^{-1}}<\infty,\;j\geq 2,\\
\partial, &\mbox{if }\Twi_{(n_0+j-2)^{-1}}=\infty,\;j\geq 2.
\end{cases}
\end{align}
See \cite[Theorem~6.14, p.114]{K:FMP-2} for the consistency theorem in use, and recall that $C_{\CN\cup \{\partial\}}[0,\infty)$ is a Polish space. By concatenating the nontrivial parts of $\{\ms Z^{(j)}_t;t\geq 0\}$ under $\widehat{\P}_{z_0}^{\bbeta,\bs w,\bi}$ and using \eqref{eq:consistency} and
\eqref{eq:consistency0}, we obtain a continuous process $\{\widehat{\ms Z}_t;t<\widehat{T}_0^{\bs w\setminus\{w_\bi\}}\}$
such that
\begin{align}\label{eq:consistency1}
\begin{aligned}
&\widehat{\E}^{\bbeta,\bs w,\bi}_{z_0}[F(\widehat{\ms Z}_{t\wedge \Twi_{\eta}};t\geq 0)]\\
&\quad =
\E^{\bbeta,\bs w,\bi,\eta}_{z_0}[F(\ms Z_t;t\geq 0) ],\quad  \forall\;0\leq F\in \B(C_{\Bbb C^N}[0,\infty)),\;\eta\in \left(0,\min_{\bj:\bj\neq \bi, w_\bj>0}|z_0^\bj|\right),
\end{aligned}
\end{align}
where $\widehat{T}_0^{\bs w\setminus\{w_\bi\}}$ is the equivalent of $\Twi_0$ when $\{\widehat{\ms Z}_t;t<\widehat{T}_0^{\bs w\setminus\{w_\bi\}}\}$ instead of $\{\ms Z_t\}$ is in use. Recall \eqref{def:Twi} for the definition of $\Twi_0$. To lighten notation, we will drop ``$\widehat{\mbox{\;\;\;}}$'' in $\widehat{\ms Z}_t$ and $\widehat{T}_0^{\bs w\setminus\{w_\bi\}}$ when working under $\widehat{\P}^{\bbeta,\bs w,\bi}_{z_0}$. Similarly, the notations defined on $C_{\CN}[0,\infty)$ earlier will be used with the same meanings under $\widehat{\P}^{\bbeta,\bs w,\bi}_{z_0}$.

The remaining of the proof of (1$\cc$) is similar to the proof of Proposition~\ref{prop:SDE1} (4$\cc$), and we can prove (2$\cc$) of the present proposition at the same time. Now, by comparing \eqref{inhomo:2} and \eqref{def:Z2SDE} and recalling the proof of Proposition~\ref{prop:SDE1} (4$\cc$), it should be clear that
\eqref{SDE:Z1final} and \eqref{SDE:Z1} hold for $0\leq t\leq \Twi_\eta$. The unique continuous extensions of the finite-variation parts of
 these SDEs to $\Twi_0$ when $\Twi_0<\infty$ also hold in the present case. This is because Lemma~\ref{lem:Qtight} holds for all initial conditions $z_0\in \CNwni$, and so,  \eqref{SDE1:AF} extends to the present case by \eqref{def:Z2SDE} and \eqref{def:Q2}.  
However, we now need to justify the existence of the limit $\ms Z_{t\wedge \Twi_\eta}$ in $\Bbb C$ as $\eta\searrow 0$ under $\widehat{\P}^{\bbeta,\bs w,\bi}_{z_0}$. This limit does not hold yet as $\widehat{\P}^{\bbeta,\bs w,\bi}_{z_0}$ is not a priori a probability measure defined on $C_{\CN}[0,\infty)$. But  as soon as we can prove the continuation of the Brownian motion parts of the SDEs of $Z^\bj$ under $\widehat{\P}^{\bbeta,\bs w,\bi}_{z_0}$ to $\Twi_0$, 
 the required continuation of $\{\ms Z_t;t<\Twi_0\}$ to $\Twi_0$ whenever $\Twi_0<\infty$ follows. Moreover, the proof of Proposition~\ref{prop:SDE1} (4$\cc$) then continues to apply on getting the required Brownian motions $\{\wt{W}^j_t\}_{1\leq j\leq N}$.

To complete the proofs of (1$\cc$) and (2$\cc$), we show that 
\begin{align}\label{W:limT0000}
\widehat{\P}^{\bbeta,\bs w,\bi}_{z_0}\left( \widetilde{W}^j_{t\wedge \Twi_0}\,\defeq\,
\lim_{\eta\searrow 0}
\widetilde{W}^j_{t\wedge \Twi_\eta}\mbox{ exists in $\Bbb C$}\right)=1,\quad\forall\;1\leq j\leq N.
\end{align}
To this end, note that by an argument similar to the one to get \eqref{UV:qvbdd1},
\begin{align}
&\quad\;\max_{1\leq j\leq N}\max\Big\{\widehat{\E}^{\bbeta,\bs w,\bi}_{z_0}\Big[\Big(\widetilde{U}^j_{ t\wedge \Twi_\eta}-\widetilde{U}^j_{ s\wedge \Twi_\eta}\Big)^4 \Big],\widehat{\E}^{\bbeta,\bs w,\bi}_{z_0}\Big[\Big(\widetilde{V}^j_{ t\wedge \Twi_\eta}-\widetilde{V}^j_{ s\wedge \Twi_\eta}\Big)^4 \Big]\Big\}\notag\\
&\leq  3(t-s)^2, \quad\forall \; 0\leq s\leq t<\infty,\;\eta\in \left(0,\min_{\bj:w_\bj>0}|z^\bj_0|\right).\label{UV:unifcont}
\end{align}
Given the $\eta$-independent bound in \eqref{UV:unifcont},  Kolmogorov's continuity theorem \cite[(2.1) Theorem, p.26]{RY-2} and Fatou's lemma imply the H\"older continuity of $\{\widetilde{W}^j_t;0\leq t<\Twi_0\wedge t_0\}$, $0<t_0<\infty$. Therefore, we obtain \eqref{W:limT0000} by the fact that any uniformly continuous real-valued function defined on $[0,t_1)$ for $0<t_1<\infty$ extends uniquely to a continuous function on $[0,t_1]$. The proofs of (1$\cc$) and (2$\cc$) are now complete simultaneously.\medskip 

\noindent {\bf (3$\cc$)} To obtain \eqref{def:Q21}, first, we combine \eqref{def:Z2SDE} and \eqref{def:Q2} to get, for any $\eta\in (0,\min_{\bj:w_\bj>0,\bj\neq \bi}|z_0^\bj|)$,
\begin{align*}
\E^{\bbeta,\bs w,\bi}_{z_0}\left[F(\ms Z_{t\wedge t_0};t\geq 0);t_0<\Twi_\eta\right]
&=\E^{\bi}_{z_0}\left[F(\ms Z_{t\wedge t_0};t\geq 0)\mc E^{\bbeta,\bs w,\bi}_{z_0}(t_0\wedge \Twi_\eta)
;t_0<\Twi_\eta\right]\\
&=\E^{\bi}_{z_0}\left[F(\ms Z_{t\wedge t_0};t\geq 0)
\mc E^{\bbeta,\bs w,\bi}_{z_0}(t_0);t_0<\Twi_\eta\right].
\end{align*}
Then \eqref{def:Q21} follows by passing $\eta\searrow 0$ on both sides of the last equality and using the monotone convergence theorem, since $\Twi_\eta\nearrow \Twi_0$, and $\Twi_0=\infty$ $\P^{\bi}_{z_0}$-a.s. by Proposition~\ref{prop:BMntc}. To see \eqref{def:Q21}  with ``$t_0<\Twi_0$'' replaced by ``$t_0\leq \Twi_0$,'' it suffices to repeat the same argument with ``$t_0<\Twi_\eta$'' changed to ``$t_0\leq \Twi_\eta$'' in the foregoing display. \medskip 

\noindent {\bf (4$\cc$)} It suffices to consider $F$ which are also bounded continuous. In this case, \eqref{def:Q21000} for $\tau\leq T^{\bs w}_\eta$, $\eta\in (0,\min_{\bj:w_\bj>0,\bj\neq \bi}|z_0^\bj|)$, holds by combining \eqref{inhomo:0}, \eqref{inhomo:2}, \eqref{def:Z2SDE} and \eqref{def:Q2} and using the bound $T^{\bs w}_\eta\leq \Twi_\eta$. The case of general $\tau\leq \Twi_0$ follows by using $\tau\wedge \Twi_\eta$ in the case just considered
and passing $\eta\searrow 0$. 
 \medskip

\noindent {\bf (5$\cc$)} 
To prove $\P^{\bbeta,\bs w,\bi}_{z_0}(\Twi_0<\infty)=1$, recall the formula \eqref{superMG} of $\mathcal E^{\bbeta,\bs w,\bi}_{z_0}(\cdot)$. By \eqref{def:Q21},
\begin{align*}
&\quad\;\P^{\bbeta,\bs w,\bi}_{z_0}(t<\Twi_0)\\
&=\E^{\bi}_{z_0}\left[\frac{w_\bi  K_0^{\bbeta,\bi}(0)}{K_0^{\bbeta,\bs w}(0)}\frac{\e^{-\mathring{A}^{\bbeta,\bs w,\bi}_0(t)-\widetilde{A}^{\bbeta,\bs w,\bi}_0(t)} K^{\bbeta,\bs w}_0(t)}{w_\bi K^{\bbeta,\bi}_0(t)}\right]\\
&=\frac{w_\bi K_0^{\bbeta,\bi}(0)}{K_0^{\bbeta,\bs w}(0)}\E^{\bi}_{z_0}[\e^{-\mathring{A}^{\bbeta,\bs w,\bi}_0(t)-\widetilde{A}^{\bbeta,\bs w,\bi}_0(t)}]+\sum_{\bj\in \mc E_N\setminus\{\bi\}}\frac{w_\bi K_0^{\bbeta,\bi}(0)}{K_0^{\bbeta,\bs w}(0)}\E^{\bi}_{z_0}\left[\e^{-\mathring{A}^{\bbeta,\bs w,\bi}_0(t)-\widetilde{A}^{\bbeta,\bs w,\bi}_0(t)}\frac{w_\bj K^{\bbeta,\bj}_0(t)}{w_\bi K^{\bbeta,\bi}_0(t)}\right]\\
&\leq \frac{w_\bi K_0^{\bbeta,\bi}(0)}{K_0^{\bbeta,\bs w}(0)}\E^{\bi}_{z_0}[\e^{-\mathring{A}^{\bbeta,\bs w,\bi}_0(t)-\widetilde{A}^{\bbeta,\bs w,\bi}_0(t)}]\\
&\quad\;+\sum_{\bj\in \mc E_N\setminus\{\bi\}}\frac{w_\bi K_0^{\bbeta,\bi}(0)}{K_0^{\bbeta,\bs w}(0)}
\left(\frac{w_\bj}{w_\bi}\right)
\E^{\bi}_{z_0}[\e^{-p_0\mathring{A}^{\bbeta,\bs w,\bi}_0(t)-p_0\widetilde{A}^{\bbeta,\bs w,\bi}_0(t)}]^{1/p_0}\E^{\bi}_{z_0}\left[\left(\frac{K^{\bbeta,\bj}_0(t)}{K^{\bbeta,\bi}_0(t)}\right)^{q_0}\right]^{1/q_0}
\end{align*}
for all pairs of H\"older conjugates $(p_0,q_0)$ such that $1<q_0<1+1/\two$.

To pass the limit of the right-hand side of the last inequality, first,
by using Fatou's lemma and \cite[Theorem~3.1 (3$\cc$)]{C:SDBG1-2} with the substitution $t'/q=t$, for all $\bj\neq \bi$,
\[
\int_0^\infty \e^{-t'}\liminf_{q\searrow 0}\E^{\bi}_{z_0}\left[\left(\frac{K^{\bbeta,\bj}_0(t'/q)}{K^{\bbeta,\bi}_0(t'/q)}\right)^{q_0}\right]\d t'\leq \liminf_{q\searrow 0}\int_0^\infty \e^{-t'}\E^{\bi}_{z_0}\left[\left(\frac{K^{\bbeta,\bj}_0(t'/q)}{K^{\bbeta,\bi}_0(t'/q)}\right)^{q_0}\right]\d t'<\infty.
\]
Also, we apply the assumption that $\beta_\bi=\min_{\bj:w_\bj>0}\beta_\bj$, which yields the fact that 
$\widetilde{A}^{\bbeta,\bs w,\bi}_0(t)\geq 0$ by Remark~\ref{rmk:NTC} (1$\cc$). Since $\mathring{A}^{\bbeta,\bs w,\bi}_0(\infty)=\infty$ $\P^{\bi}_{z_0}$-a.s. by \cite[(3.2)]{C:SDBG1-2}  and the definition \eqref{def:Aring} of $\mathring{A}^{\bbeta,\bs w,\bi}_0(\cdot)$, it follows from the last two displays that 
 we can choose $\{t_n\}$ with $0<t_n\nearrow \infty$ for
\[
0=
\lim_{n\to\infty}\P^{\bbeta,\bs w,\bi}_{z_0}(t_n<\Twi_0)=\P^{\bbeta,\bs w,\bi}_{z_0}(\Twi_0=\infty),
\]
which entails the required identity $\P^{\bbeta,\bs w,\bi}_{z_0}(\Twi_0<\infty)=1$.  \medskip

\noindent {\bf (6$\cc$)} 
The strong Markov property can be obtained by modifying the proof in \eqref{SMP:proof}, now via $\P^\bi$ using \eqref{def:Z2SDE} and \eqref{def:Q2} and approximating $\Twi_0$ with $t_0\wedge \Twi_\eta$ for $t_0\nearrow\infty$ and $\eta\searrow0 $. Specifically, for any $(\F_t^0)$-stopping time $\tau_0$ and bounded continuous $F_1,F_2$, we consider
\begin{align}
&\quad\;\E_{z_0}^{\bbeta,\bs w,\bi}\left[F_1(\ms Z_{t\wedge \tau_0};t\geq 0)F_2(\ms Z_{(\tau_0+t)\wedge \Twi_0};t\geq 0);\tau_0<\Twi_0\right]\notag\\
&=\lim_{\eta\searrow 0}\lim_{t_0\nearrow\infty} \E_{z_0}^{\bbeta,\bs w,\bi}\left[F_1(\ms Z_{t\wedge \tau_0};t\geq 0)F_2(\ms Z_{(\tau_0+t)\wedge t_0\wedge  T^{\bs w}_\eta};t\geq 0);\tau_0<t_0\wedge  \Twi_\eta\right]\notag\\
&=\lim_{\eta\searrow 0}\lim_{t_0\nearrow\infty}\E^{\bi}_{z_0}\left[F_1(\ms Z_{t\wedge \tau_0};t\geq 0)F_2(\ms Z_{(\tau_0+t)\wedge t_0\wedge  T^{\bs w}_\eta};t\geq 0)\mc E^{\bbeta,\bs w,\bi}_{z_0}\left(t_0\wedge \Twi_\eta\right);\tau_0<t_0\wedge  \Twi_\eta \right]\notag\\
&=\lim_{\eta\searrow 0}\lim_{t_0\nearrow\infty}\E^{\bi}_{z_0}\left[F_1(\ms Z_{t\wedge \tau_0};t\geq 0)\mc E^{\bbeta,\bs w,\bi}_{z_0}(\tau_0)\right.\notag\\
&\quad \times\E^{\bi}_{\ms Z_{\tau_0}}\left[F_2(\ms Z_{t\wedge (t_0-r)\wedge \Twi_\eta};t\geq 0)\mc E^{\bbeta,\bs w,\bi}_{z_0}\left((t_0-r)\wedge \Twi_\eta\right)\right]\Big|_{r=\tau_0};\tau_0<t_0\wedge \Twi_\eta\Big]\notag\\
&=\lim_{\eta\searrow 0}\lim_{t_0\nearrow\infty} \E_{z_0}^{\bbeta,\bs w,\bi}\left[F_1(\ms Z_{t\wedge \tau_0};t\geq 0)\E^{\bbeta,\bs w,\bi}_{\ms Z_{\tau_0}}[F_2(\ms Z_{t\wedge (t_0-r)\wedge \Twi_\eta};t\geq 0)]\big|_{r=\tau_0};\tau_0<t_0\wedge \Twi_\eta\right]\notag\\
&=\E_{z_0}^{\bbeta,\bs w,\bi}\left[F_1(\ms Z_{t\wedge \tau_0};t\geq 0)\E^{\bbeta,\bs w,\bi}_{\ms Z_{\tau_0}}[F_2(\ms Z_{t\wedge \Twi_0};t\geq 0)];\tau_0<\Twi_0\right].\notag
\end{align}
In more detail, the third equality uses the multiplicativity of $\mathcal E^{\bbeta,\bs w,\bi}_{z_0}(\cdot)$ and the property $\ms Z_{\tau_0}\in \CNwni$ on $\{\tau_0<t_0\wedge \Twi_\eta\}$, and the fourth equality uses the fact that $\{\mathcal E^{\bbeta,\bs w,\bi}_{z_0}(t\wedge \Twi_\eta)\}$ is a martingale by Corollary~\ref{cor:superMG}, as well as \eqref{def:Z2SDE} and
\eqref{def:Q2}. The last equality is enough to prove the required strong Markov property. 
\end{proof}

It remains to prove Proposition~\ref{prop:SDE2} (7$\cc$). The required property is comparable to Proposition~\ref{prop:BMntc}, but the situation here is more complicated. The following lemma specifies the SDEs comparable to the one in \eqref{SDE:NTC1}, now considering the presence of multiple particles.

\begin{lem}\label{lem:ZS}
Let $\bbeta\in (0,\infty)^{\mc E_N}$, $\bs w\in \R_+^{\mc E_N}$ with $\#\{\bj;w_\bj>0\}\geq 2$, $\bi\in \mc E_N$ with $w_\bi>0$, and
$z_0\in \CNwni$. For any $\mathcal S$ with $\varnothing\neq \mc S\subset\mc E_N$, it holds that under $\P^{\bbeta,\bs w,\bi}_{z_0}$, 
\begin{align}\label{SDE:ZS2}
\sum_{\bj\in \mc S}| Z^\bj_t|
 &=\sum_{\bj\in \mc S} |Z^\bj_0|+\sum_{\bj\in \mc S} \int_0^t\frac{1-\Phi^{\bbeta,\bs w,\bj}_{\mc S}(s)}{2|Z^\bj_s|}\d s+\sum_{\bj\in \mc S}\widetilde{B}^\bj_t,\quad 0\leq t\leq \Twi_0.
 \end{align}
Here, $\widetilde{B}^\bj$ are one-dimensional standard Brownian motions such that, with $Z^\bj=X^\bj+\i Y^\bj$ and $\widetilde{W}^\bj=\widetilde{U}^{\bj}+\i \widetilde{V}^\bj$ for real $X^\bj,Y^\bj,\widetilde{U}^\bj,\widetilde{V}^\bj$,  the following holds for all $\bj,\bk\in \mc E_N$ and $0\leq t\leq \Twi_0$:
\begin{align}\label{covar:tB}
\widetilde{B}^\bj_t\,\defeq\, \int_0^t \frac{ X^\bj_s\d \widetilde{U}^\bj_s+Y^\bj_s\d \widetilde{V}^\bj_s}{| Z^\bj_s|}\Longrightarrow 
\la \widetilde{B}^\bj,\widetilde{B}^\bk\ra_t=\frac{\sigma(\bj)\cdot \sigma(\bk)}{2}\int_0^t \frac{X^\bj_sX^\bk_s+Y^\bj_sY^\bk_s}{|Z^\bj_s||Z^\bk_s|}\d s.
\end{align}
 Also, for any $s\in (0,\Twi_0]$ such that $|Z^\bk_s|>0$ for all $\bk\in \mc E_N$,  
\begin{align}
\begin{aligned}
\Phi^{\bbeta,\bs w,\bj}_{\mc S}(s)&\,\defeq\, \sum_{\bk\in \mc S}
  \frac{|Z^\bj_s|}{|Z^\bk_s|}\Re\biggl(\frac{ Z^\bk_s}{ Z^\bj_s}\biggr) [\sigma(\bj)\cdot\sigma(\bk)]
 \frac{w_\bj \hK_1^{\bbeta,\bj}(s)}{K_0^{\bbeta,\bs w}(s)}\\
 &\quad\,+\sum_{\bk\in \mathcal E_N\setminus \mc S}\frac{|Z^\bj_s|^2}{|Z^\bk_s|^2}\Re\biggl(\frac{ Z^\bk_s}{ Z^{\bj}_s}\biggr)[\sigma(\bj)\cdot\sigma(\bk)]
 \frac{w_\bk \hK_1^{\bbeta,\bk}(s)}{K_0^{\bbeta,\bs w}(s)},\quad \forall\;\bj\in \mc S.\label{def:PhiNTC}
 \end{aligned}
\end{align}
Moreover,  
\begin{align}\label{bdd:NTC1}
\begin{aligned}
|\Phi^{\bbeta,\bs w,\bj}_{\mc S}(s)|
\leq C_{\ref{bdd:NTC1}}(\bbeta)
 \frac{w_\bj\#\mc S+
\sum_{\bk\in \mc E_N\setminus\mc S}w_\bk|Z^\bk_s|^{-1} |Z^\bj_s|}{K_0^{\bbeta,\bs w}(s)}.
\end{aligned}
\end{align}
\end{lem}

\begin{proof}
We derive \eqref{SDE:ZS2} first.
 By summing over $|Z^\bj_t|$ for $\bj\in \mc S$ and using \eqref{SDE:Z1} for $0\leq t\leq \Twi_0$, \eqref{Lp:Z2drift}, and \cite[Lemma~5.5]{C:SDBG1-2} with the choice of
\[
\mathcal Z^\bj\equiv Z^\bj,\quad \tau\equiv \Twi_0\wedge n,\quad \d A^{\bj,\bk}_s\equiv \frac{\sigma(\bj)\cdot \sigma(\bk)}{2}\frac{w_\bk\hK_1^{\bbeta,\bk}(s)}{K_0^{\bbeta,\bs w}(s)}\d s,
\]
for every integer $n\geq 1$,
we get
\begin{align}
\sum_{\bj\in \mc S}| Z^\bj_t|&=\sum_{\bj\in \mc S} |Z^\bj_0|+\sum_{\bj\in \mc S}\int_0^t \frac{\d s}{2|Z^\bj_s|}-\sum_{\bj\in \mc S}\sum_{\bk\in \mc E_N}\int_0^t\frac{1}{|Z^\bj_s|}\Re\biggl(\frac{ Z^\bj_s}{ Z^\bk_s}\biggr)\cdot \frac{\sigma(\bj)\cdot\sigma(\bk)}{2}
 \frac{w_\bk\hK_1^{\bbeta,\bk}(s)}{K_0^{\bbeta,\bs w}(s)}\d s\notag\\
 &\quad+\sum_{\bj\in \mc S}\widetilde{B}^\bj_t,\quad 0\leq t\leq \Twi_0,\label{SDE:ZS1}
\end{align}
where $\widetilde{B}^\bj$ satisfy the property stated below \eqref{SDE:ZS2}. To get \eqref{SDE:ZS2} from \eqref{SDE:ZS1}, note that the double sum on the right-hand side of \eqref{SDE:ZS1} 
satisfies 
\begin{align}
&\quad \;\sum_{\bj\in \mc S}\sum_{\bk\in \mc E_N}\int_0^t\frac{1}{|Z^\bj_s|}\Re\biggl(\frac{ Z^\bj_s}{ Z^\bk_s}\biggr)\cdot \frac{\sigma(\bj)\cdot\sigma(\bk)}{2}
 \frac{w_\bk\hK_1^{\bbeta,\bk}(s)}{K_0^{\bbeta,\bs w}(s)}\d s\notag\\
&=\Bigg(\sum_{\bj\in \mc S}\sum_{\bk\in \mc S}+\sum_{\bj\in \mc S}\sum_{\bk\in\mathcal E_N\setminus \mc S}\Bigg)\int_0^t\frac{1}{|Z^\bj_s|}\Re\biggl(\frac{ Z^\bj_s}{ Z^\bk_s}\biggr)\cdot \frac{\sigma(\bj)\cdot\sigma(\bk)}{2}
 \frac{w_\bk \hK_1^{\bbeta,\bk}(s)}{K_0^{\bbeta,\bs w}(s)}\d s\notag\\
 \begin{split}\label{bdd:NTC1000}
  &=\sum_{\bj\in \mc S} \sum_{\bk\in \mc S}\int_0^t
\frac{1}{2|Z^\bj_s|}\frac{|Z^\bj_s|}{|Z^\bk_s|}\Re\biggl(\frac{ Z^\bk_s}{ Z^\bj_s}\biggr)\cdot [\sigma(\bk)\cdot\sigma(\bj)]
 \frac{w_\bj\hK_1^{\bbeta,\bj}(s)}{K_0^{\bbeta,\bs w}(s)}\d s\\
 &\quad +\sum_{\bj\in \mc S}\sum_{\bk\in \mathcal E_N\setminus \mc S}\int_0^t\frac{1}{2|Z^\bj_s|}\frac{|Z^\bj_s|^2}{|Z^\bk_s|^2} \Re\biggl(\frac{ Z^\bk_s}{ Z^{\bj}_s}\biggr)\cdot [\sigma(\bj)\cdot\sigma(\bk)]
 \frac{w_\bk\hK_1^{\bbeta,\bk}(s)}{K_0^{\bbeta,\bs w}(s)}\d s.
 \end{split}
\end{align}
Here, on the right-hand side of \eqref{bdd:NTC1000}, the first double sum follows by using the change of variables that replaces $(\bj,\bk)$ with $(\bk,\bj)$ and then writing $1/|Z^\bk_s|$ as $(1/|Z^\bj_s|)(|Z^\bj_s|/|Z^\bk_s|)$, and the second double sum holds by using $\Re(z)=|z|^2\Re(1/z)$ with $z\equiv Z^\bj_s/Z^\bk_s$. The required equality \eqref{SDE:ZS2} follows by noting that $\sigma(\bk)\cdot\sigma(\bj)=\sigma(\bj)\cdot \sigma(\bk)$ and applying \eqref{bdd:NTC1000} to \eqref{SDE:ZS1}. 

It remains to prove \eqref{bdd:NTC1}. 
By using $|\sigma(\bj)\cdot\sigma(\bk)|\leq 2$ and $||z|\Re(1/z)|\leq 1$ with $z\equiv Z^\bk_s/Z^\bj_s$, \eqref{def:PhiNTC} implies that for all $s\in (0,\Twi_0]$ with $|Z^\bk_s|>0$ for all $\bk\in \mc E_N$,  
\begin{align}\label{bdd:NTC}
|\Phi^{\bbeta,\bs w,\bj}_{\mc S}(s)|\leq 2\#\mc S\cdot \frac{w_\bj\hK_1^{\bbeta,\bj}(s)}{K_0^{\bbeta,\bs w}(s)}+ 2
\Bigg(\sum_{\bk\in \mathcal E_N\setminus \mc S}\frac{w_\bk \hK^{\bbeta,\bk}_1(s)}{|Z^\bk_s|}\Bigg)
\frac{|Z^\bj_s|}{K_0^{\bbeta,\bs w}(s)}.
\end{align}
To finish the proof, recall the asymptotic representation \eqref{K10} of $K_1(x)$ as $x\searrow 0$, and note that $\hK_1(x)\,\defeq\,xK_1(x)$ is decreasing since $\widehat{K}_1'(x)=-x K_0(x)<0$ \cite[(5.7.9), p.110]{Lebedev-2}.  Hence, \eqref{bdd:NTC} is enough to get \eqref{bdd:NTC1}. The proof is complete.
\end{proof}

The following definition specifies the notion that we will use to reduce the possible number of simultaneous zeros in $\{|Z^\bj_t|;0<t\leq \Twi_0\}_{\bj\in \mc E_N}$ one by one. 

\begin{defi}\label{def:NTC}
Let $I\subset \R_+$ with $I\neq \varnothing$, and let
 $\ms A=\{f_1,\cdots,f_n\}$ with $n\geq 2$ and $0\leq f_\ell\in \C(I)$. Given $j\in \{2,\cdots,n\}$, 
{\bf $\bs j$-no-simultaneous-contacts ($\bs j$-NSC)} in $\ms A$ over $I$ is said to fail if there exist integers $\ell_1,\cdots,\ell_j$ with $1\leq \ell_1<\cdots<\ell_j\leq n$ such that $f_{\ell_1}(t_0)+\cdots+f_{\ell_j}(t_0)=0$ for some $t_0\in I$. Otherwise, $j$-NSC in $\ms A$ over $I$ is said to hold. If $2$-NSC in $\ms A$ over $I$ holds, we also say that {\bf no-triple-contacts (NTC)} in $\ms A$ over $I$ holds.  
\end{defi}

The informal picture here is that $f_\ell$ represents the distance between two particles, and different $f_\ell$'s are for different pairs of particles. In particular, the NTC in Definition~\ref{def:NTC} holds if and only if for any $\ell_1\neq \ell_2$, $f_{\ell_1}$ and $f_{\ell_2}$ do not attain a zero at the same time, which holds if and only if for any two different pairs of particles, contacts do not happen at the same time. Hence, the NTC in Definition~\ref{def:NTC} is an abstract formalization of the NTC described in Section~\ref{sec:intro}, and the notions of $j$-NSC generalize that NTC in Section~\ref{sec:intro}.

\begin{lem}\label{lem:RRprop}
Let $\bbeta\in (0,\infty)^{\mc E_N}$, $\bs w\in \R_+^{\mc E_N}$ with $\#\{\bj;w_\bj>0\}\geq 2$, and $\bi\in \mc E_N$ with $w_\bi>0$. Then it holds that
\begin{align}\label{step1:property}
\P^{\bbeta,\bs w,\bi}_{z_0}\left(\underline{R}<\overline{R} \mbox{ and $\#\mathcal S$-NSC in $\{| Z^\bj_t|\}_{\bj\in \mc E_N}$ over $(\underline{R},\overline{R}]$ fails}\right)=0
\end{align}
for any $\mathcal S\subset\mathcal E_N$ with $\#\mathcal S\geq 2$ and stopping times $\underline{R}, \overline{R}$ satisfying the following conditions with condition \eqref{R0R1:3} dropped if
$\mathcal S=\mathcal E_N$:
\begin{gather}
\overline{R}<\infty,\label{R0R1:1}\\
\underline{R}\leq \overline{R}\leq \Twi_0,\label{R0R1:2}\\
\min_{\bj\in \mathcal E_N\setminus\mc S}|Z_t^\bj|\geq \frac{1}{K+1}\min_{\bj\in \mathcal E_N\setminus\mc S}|Z_{\underline{R}}^\bj|>0,\quad \forall\;t\in [\underline{R}, \overline{R}],\label{R0R1:3}\\
\max_{\bj\in \mc S}|Z_t^\bj|\leq \max_{\bj\in \mc S}|Z_{\underline{R}}^\bj|+L,\quad \forall\;t\in [\underline{R}, \overline{R}],\label{R0R1:4}
\end{gather}
where $L,K\geq 1$ are non-random integers.
\end{lem}
\begin{proof}
By \eqref{R0R1:3}, the following identity is equivalent to \eqref{step1:property}:
\begin{align}\label{step1:property0}
\P^{\bbeta,\bs w,\bi}_{z_0}\left(\underline{R}<\overline{R} \mbox{ and $\#\mathcal S$-NSC in $\{| Z^\bj_t|\}_{\bj\in \mc S}$ over $(\underline{R},\overline{R}]$ fails}\right)=0.
\end{align}
It suffices to prove \eqref{step1:property0}.

To get \eqref{step1:property0}, we modify the proof of \eqref{BMntc} and use  Theorem~\ref{thm:NTC}. In this case, $\sum_{\bj\in \mc S}|Z^\bj_t|$ satisfies \eqref{SDE:ZS2} with a noise term given by  $\sum_{\bj\in \mc S}\widetilde{B}^\bj_t$, where $\widetilde{B}^\bj$ are one-dimensional standard Brownian motions satisfying \eqref{covar:tB}. 
To bound the finite-variation term in \eqref{SDE:ZS2}, for all $\theta>0$ and integers $M\geq 1$, we introduce $(\F_t^0)$-stopping times $S_n=S_n(\theta,M)$ and $T_n=T_n(\theta,M)$ defined inductively on $n\in \Bbb Z_+$ as follows:
$S_0=T_0\,\defeq\,\underline{R}$, 
and for all integers $n\geq 1$,
\begin{align*} 
S_n\,\defeq\,\inf\Bigg\{t\geq T_{n-1};\sum_{\bj\in \mc S}|Z^\bj_t|\leq \theta\Bigg\}\wedge \overline{R},\quad
T_n\,\defeq\,\inf\Bigg\{t\geq S_n;\sum_{\bj\in \mc S}|Z^\bj_t|=2\theta\Bigg\}\wedge \overline{R}.
\end{align*}
The crucial fact we need now is that by \eqref{bdd:NTC1}, \eqref{R0R1:2} and \eqref{R0R1:3}, for all $\eta>0$, there exists 
$0<\theta_0=\theta_0(\omega)\in \F^0_{\underline{R}}$ such that on $\{S_n(\theta_0,M)<\overline{R}\}$, 
\[
\sup_{\bj\in \mathcal S}\esssup\left\{\left|\Phi^{\bbeta,\bs w,\bj}_{\mc S}(s)\right|;s\in [S_n(\theta_0,M),T_n(\theta_0,M)]\right\}<\eta.
\]
Note that $S_n(\theta_0,M)$ and $T_n(\theta_0,M)$ are still stopping times since $\theta_0\in \F^0_{\underline{R}}$, and
 $\{S_n(\theta_0,M)<T_n(\theta_0,M)\}\subset \{S_n(\theta_0,M)<\overline{R}\}$. Recall Remark~\ref{rmk:ntc} (1$\cc$).
Therefore, by choosing $\eta$ sufficiently small, Theorem~\ref{thm:NTC} with $\mc J\equiv \mc S$ implies that for all $n\in \Bbb N$,
\[
\P^{\bbeta,\bs w,\bi}_{z_0}\Biggl(S_n(\theta_0,M)<T_n(\theta_0,M)\mbox{ and }\exists\; t\in (S_n(\theta_0,M),T_n(\theta_0,M)]\mbox{ s.t. }\sum_{\bj\in \mc S}|Z^\bj_t|=0\Biggr)=0.
\]
Moreover, similar to the argument below \eqref{NTCt0bdd}, the continuity of $t\mapsto \sum_{\bj\in \mc S}|Z^\bj_t|$ ensures that 
we can find an integer $n=n(\omega)\geq 0$ such that $S_n(\theta_0,M)=\overline{R}$. Hence, by the union bound, the foregoing equality gives \eqref{step1:property0}. The proof is complete.
\end{proof}

The proof of Proposition~\ref{prop:SDE2} (7$\cc$) also uses excursion intervals of $(\F_t^0)$-adapted nonnegative continuous processes $\varrho=\{\varrho_t;t\in \R_+\}$. For such $\varrho$, we introduce random times $g_n(a;\varrho)$ and $d_n(a;\varrho)$, $0<a<\infty$, in the following manner, where $t\in \R_+$ and $n\in \Bbb N$ with $n\geq 2$:
\begin{align*}
g(t,a;\varrho)&\,\defeq\,\begin{cases}
\sup\{s\in [0,t];\varrho_s=0\},&\inf\{s\in [t,\infty);\varrho_s=0\}-\sup\{s\in [0,t];\varrho_s=0\}\geq a,\\
\infty,&\inf\{s\in [t,\infty);\varrho_s=0\}-\sup\{s\in [0,t];\varrho_s=0\}<a,
\end{cases}\\ 
d(t,a;\varrho)&\,\defeq\,\begin{cases}
\inf\{s\in [t,\infty);\varrho_s=0\},&\inf\{s\in [t,\infty);\varrho_s=0\}-\sup\{s\in [0,t];\varrho_s=0\}\geq a,\\
\infty,&\inf\{s\in [t,\infty);\varrho_s=0\}-\sup\{s\in [0,t];\varrho_s=0\}<a,
\end{cases}\\
g_1(a;\varrho)&\,\defeq\, \inf\{g(s,a;\varrho);s\in \Bbb Q_+\},\\
 d_1(a;\varrho)&\,\defeq\, \inf\{d(s,a;\varrho);s\in \Bbb Q_+\},\\
g_n(a;\varrho)&\,\defeq\, 
\begin{cases}
\inf\{g(s,a;\varrho);s\in \Bbb Q_+,g(s,a;\varrho)\neq g_j(a;\varrho),\forall \;1\leq j\leq
n-1\},&\mbox{ if $g_{n-1}(a;\varrho)<\infty$},\\
\infty,&\mbox{ if $g_{n-1}(a;\varrho)=\infty$},
\end{cases}
\\
d_n(a;\varrho)&\,\defeq\, \begin{cases}
\inf\{d(t,a;\varrho);s\in \Bbb Q_+, d(s,a;\varrho)\neq d_j(a;\varrho),\forall\; 1\leq j\leq
n-1\},&\mbox{ if $d_{n-1}(a;\varrho)<\infty$},\\
\infty,&\mbox{ if $d_{n-1}(a;\varrho)=\infty$}.
\end{cases}
\end{align*}
Note that the above definitions use the usual convention $\sup\varnothing=0$. 

\begin{lem}\label{lem:stoppingtimes}
For $n\in \Bbb N$ and $0<a<\infty$,
$d_n(a;\varrho)$ and $g_n(a;\varrho)+a$ are $(\F_t^0)$-stopping times.
\end{lem}
\begin{proof}
To see that $d_n(a;\varrho)$ is an $(\F_t^0)$-stopping time, simply note that if $d_n(a;\varrho)<\infty$, the infimum of $d(s,a;\varrho)$ in the definition of $d_n(a;\varrho)$ is a minimum. To show that $g_n(a;\varrho)+a$ is a stopping time, set 
\[
\tau\,\defeq\,\inf\{t\geq d_{n-1}(a;\varrho)+a;\varrho_s> 0,\, \forall \, s\in (t-a,t)\}.
\]
Since $\varrho_s\neq 0$ on $(g_n(a;\varrho),g_n(a;\varrho)+a)$ and $g_n(a;\varrho)\geq d_{n-1}(a;\varrho)$, it follows that $g_n(a;\varrho)+a\geq \tau$. Conversely, the existence of some $t$ such that $d_{n-1}(a;\varrho)+a\leq t<g_n(a;\varrho)+a$ and $\varrho_s>0$ for all $s\in (t-a,t)$ would contradict the definition of $g_n(a;\varrho)$. Hence, $g_n(a;\varrho)+a=\tau$. Note that $\tau$ is a stopping time, since on $\{\tau<\infty\}$, the infimum in the definition of $\tau$ becomes a minimum.
We conclude that $g_n(a;\varrho)+a$ is an $(\F_t^0)$-stopping time, as required. 
\end{proof}

\begin{proof}[Proof of Proposition~\ref{prop:SDE2} (7$\cc$)]
Fix $z_0\in \CNwni$ and an integer $N\geq 3$.
The proof is divided into three steps, using an induction on $j={N\choose 2}, {N\choose 2}-1,\cdots, 2$ in the same order of $j$ to establish the following property:
\begin{align}\label{NTC:goal}
\P^{\bbeta,\bs w,\bi}_{z_0}\left(\Twi_0<\infty \mbox{ and $j$-NSC in $\{|Z^\bj_t|\}_{\bj\in \mc E_N}$ over $(0,\Twi_0]$ fails}\right)=0.
\end{align}
The reader may recall the terms for $j$-NSC from Definition~\ref{def:NTC}.\medskip

\noindent {\bf Step~1.} We show that \eqref{NTC:goal} for $j={N\choose 2}$ holds. To use Lemma~\ref{lem:RRprop}, we set $\underline{R}\equiv 0$ and $\overline{R}\equiv R(L,M)$, where
\[
R(L,M)\,\defeq\,\inf\left\{t\geq 0; \max_{\bj\in \mc E_N}|Z_t^\bj|\geq \max_{\bj\in \mc E_N}|Z_{0}^\bj|+L\right\}\wedge \Twi_0\wedge M,\quad L,M\in \Bbb N.
\]
Since $R(L,M)>0$, it follows from Lemma~\ref{lem:RRprop} that
\begin{align}\label{R0R1:bdd}
\P^{\bbeta,\bs w,\bi}_{z_0}\left(\mbox{${N\choose 2}$-NSC in $\{|Z^\bj_t|\}_{\bj\in \mc E_N}$ over $(0,R(L,M)]$ fails}\right)=0.
\end{align}
Moreover, since
$\Twi_0=R(L,M)$ for large enough $L,M$ on $\{\Twi_0<\infty\}$, we have
\begin{align*}
&\quad\;\P^{\bbeta,\bs w,\bi}_{z_0}\left(\Twi_0<\infty \mbox{ and ${N\choose 2}$-NSC in $\{| Z^\bj_t|\}_{\bj\in \mc E_N}$ over $(0,\Twi_0]$ fails}\right)\\
&\leq \P^{\bbeta,\bs w,\bi}_{z_0}\left(\bigcup_{L=1}^\infty\bigcup_{M=1}^\infty \left\{\mbox{${N\choose 2}$-NSC in $\{|Z^\bj_t|\}_{\bj\in \mc E_N}$ over $(0,R(L,M)]$ fails}\right\}\right)=0.
\end{align*}
The last equality holds by \eqref{R0R1:bdd} and the union bound. We have proved \eqref{NTC:goal} for $j={N\choose 2}$. \medskip

\noindent {\bf Step~2.} Assume that for some integer $0\leq j<{N\choose 2}-2$, \eqref{NTC:goal} holds with $j$ replaced by ${N\choose 2}-\ell$ for all integers $0\leq \ell\leq j$. Our goal here is to prove \eqref{NTC:goal} for $j$ replaced by ${N\choose 2}-j-1$.

We first show that 
\begin{align}\label{NTC:step4-1}
\P^{\bbeta,\bs w,\bi}_{z_0}\left(
\left.
\begin{array}{cc}
\mbox{$\Twi_0<\infty$ and $\exists\;\mathcal L:\mathcal L\subset \mc E_N$ with $\#\mathcal L=j+1$ s.t.}\\
\mbox{$({N\choose 2}-j-1)$-NSC in $\{| Z^\bj_t|\}_{\bj\in \mc E_N}$ over }\\
\mbox{$\{0<t\leq \Twi_0;\min_{\bj\in \mc L}|Z^{\bj}_t|>0\}$ fails}
\end{array}
\right.\right)=0.
\end{align}
To this end, fix $\mathcal L$ with $\mathcal L\subset \mc E_N$ and $\#\mathcal L=j+1$. With
 the random times $g_n(a;\varrho),d_n(a;\varrho)$ defined before Lemma~\ref{lem:stoppingtimes}, we set 
\[
g_n^\mathcal L(m)\,\defeq\,g_n(\tfrac{1}{m};\varrho)\wedge \Twi_0,\; d_n^\mathcal L(m)\,\defeq\,d_n(\tfrac{1}{m};\varrho)\wedge \Twi_0,\; m\in \Bbb N,\mbox{ for }\varrho_t\equiv\min_{\bj\in \mc E_N}\Big|Z^{\bj}_{t\wedge \Twi_0}\Big|.
\] 
To apply Lemma~\ref{lem:RRprop}, for all $n,m,L,K,M\in \Bbb N$,
we choose $\mathcal S\equiv \mathcal E_N\setminus \mathcal L$ and 
\begin{align*}
\underline{R}&\equiv R(n,m,M)\,\defeq\, \left(g_n^{\mathcal L}(m)+\frac{1}{m}\right)\wedge \Twi_0\wedge M,\\
\overline{R}&\equiv R'(n,m,L,K,M)
\,\defeq\,
\inf\left\{t\geq R(n,m,M);\min_{\bj\in \mc L}|Z_t^{\bj}|\leq \frac{1}{K+1}\min_{\bj\in \mc L}|Z_{R(n,m,M)}^{\bj}|\right\} \\
&\quad\quad\wedge \inf\left\{t\geq R(n,m,M); \max_{\bj\in \mc E_N\setminus\mathcal L}|Z_t^\bj|\geq \max_{\bj\in \mc E_N\setminus\mathcal L}|Z_{R(n,m,M)}^\bj|+L\right\} \wedge d_n^{\mathcal L}(m)\wedge \Twi_0\wedge M.
\end{align*}
Note that $\#\mathcal S={N\choose 2}-j-1>1$ by the assumption on $j$, and the random times defined in the foregoing display are indeed stopping times by Lemma~\ref{lem:stoppingtimes}. Hence, by Lemma~\ref{lem:RRprop},
\begin{align*}
\P^{\bbeta,\bs w,\bi}_{z_0}\left(
\left.
\begin{array}{cc}
R(n,m,M)<R'(n,m,L,K,M) \mbox{ and (${N\choose 2}-j-1$)-NSC in  }\\
\mbox{$\{| Z^\bj_t|\}_{\bj\in \mc E_N}$ over $(R(n,m,M),R'(n,m,L,K,M)]$ fails}
\end{array}
\right.\right)=0
\end{align*}
for all $n,m,L,K,M\in \Bbb N$. Furthermore, by the union bound, \eqref{NTC:step4-1} follows from the foregoing equality
since on $\{\Twi_0<\infty\}$ the two events in the next display are $\P^{\bbeta,\bs w,\bi}_{z_0}$-a.s. equal:
\begin{align*}
\Big\{0<t\leq \Twi_0;\min_{\bj\in \mc L}|Z^{\bj}_t|>0\Big\}=\bigcup_{n,m,L,K,M\in \Bbb N}(R(n,m,M),R'(n,m,L,K,M)].
\end{align*}

We also need the following property:
\begin{align}
&\quad\;\P^{\bbeta,\bs w,\bi}_{z_0}\Bigg(\Twi_0<\infty \mbox{ and }\bigcup_{\mathcal L:\mathcal L\in \EN, \#\mathcal L=j+1}\Big\{0<t\leq \Twi_0;\min_{\bj\in \mc L}|Z^{\bj}_t|>0\Big\}=(0,\Twi_0]\Bigg)\notag\\
&=\P^{\bbeta,\bs w,\bi}_{z_0}(\Twi_0<\infty),\label{NTC:interval}
\end{align}
that is, $\P^{\bbeta,\bs w,\bi}_{z_0}$-a.s. on $\{\Twi_0<\infty\}$, 
 any $t_0\in (0,\Twi_0]$ satisfies $\min_{\bj\in \mc L}|Z^\bj_{t_0}|>0$ for some $\mathcal L\subset \mc E_N$ with $\#\mathcal L=j+1$. This set $\mathcal L$ can be chosen inductively as follows. First, use the ${N\choose 2}$-NSC to pick $\mathbf j_1\in \mc E_N$ such that $| Z^{\bj_1}_{t_0}|>0$. 
 If $j=0$, then we simply take $\mathcal L=\{\bj_1\}$ and stop. For $j\geq 1$, once distinct $\bj_1,\cdots,\bj_\ell$ for $1\leq \ell\leq j$ have been chosen, use the $({N\choose 2}-\ell)$-NSC from the induction assumption to choose $\bj_{\ell+1}$ from $\mathcal E_N\setminus\{\bj_1,\cdots,\bj_\ell\}$ such that $|Z^{\bj_{\ell+1}}_{t_0}|>0$. [Specifically, the assumption states that \eqref{NTC:goal} holds with $j$ replaced by ${N\choose 2}-\ell$ for all integers $0\leq \ell\leq j$.]  Continuing until $\ell=j$ gives the required set $\mathcal L$ as $\{\bj_1,\cdots,\bj_{j+1}\}$, so \eqref{NTC:interval} holds. 

Finally, we obtain \eqref{NTC:goal} for $j$ replaced by ${N\choose 2}-j-1$ upon combining \eqref{NTC:step4-1} and \eqref{NTC:interval}.\medskip 

\noindent {\bf Step~3.} By Steps~1--2 and mathematical induction, we obtain \eqref{NTC:goal} for $j=2$. This property is enough to get Proposition~\ref{prop:SDE2} (7$\cc$).
\end{proof}

\subsection{End of the proof of Theorem~\ref{thm:main1}}\label{sec:SDEcont}
To sum up, the existence of $\ms P=\{\P^{\bbeta,\bs w}_{z_0};z_0\in \CNw\}\cup 
\{ \P^{\bbeta,\bs w,\bi}_{z_0};z_0\in\CNwi\}_{\bi\in \mc E_N,w_\bi>0}$ satisfying (1$\cc$)--(3$\cc$) follows by using a general theorem of concatenation of strong Markov processes \cite[Section~14 in Chapter~II, pp.77+]{Sharpe-2} to concatenate the processes constructed in Sections~\ref{sec:SDE1} and~\ref{sec:SDE2}. We also use the properties in Proposition~\ref{prop:SDE1} and Proposition~\ref{prop:SDE2}.
 
Let us give more details of the above summary. First, for any $z_0\in \CNw$,
 we extend the probability measure $\P^{\bbeta,\bs w}_{z_0}$ from Definition~\ref{def:SDE1ult} beyond 
$T_0^{\bs w}=T_0^1$ by concatenation inductively over $[T_0^{m},T_0^{m+1}]$, $m\in \Bbb N$, 
using the probabilty measures from Proposition~\ref{prop:SDE2} (1$\cc$) such that the path of $\{\ms Z_t\}$ over $[0,T_0^\infty)$ is continuous. We define $\ms Z_t\equiv \partial$ for $t\geq T_0^\infty$ on $\{T_0^\infty<\infty\}$. In this concatenation, the existence of the random indices $\bs J_m$ in Theorem~\ref{thm:main1} (1$\cc$) has used 
Proposition~\ref{prop:SDE1} (3$\cc$) and Proposition~\ref{prop:SDE2} (7$\cc$), and $\P^{\bbeta,\bs w}_{z_0}(T_0^1<\infty)=1$ by Proposition~\ref{prop:SDE1} (2$\cc$). We denote the extension of $\P^{\bbeta,\bs w}_{z_0}$ thus obtained by $\P^{\bbeta,\bs w}_{z_0}$. This extension gives the required probability measure  $\P^{\bbeta,\bs w}_{z_0}$ in $\ms P$. In a similar fashion, for any $z_0\in \CNwi$ with $w_\bi>0$, we extend the probability measure $\P^{\bbeta,\bs w,\bi}_{z_0}$ from Proposition~\ref{prop:SDE2} (1$\cc$) beyond $\Twi_0=T^{\bi,1}_0$ and get the required probability measure $\P^{\bbeta,\bs w,\bi}_{z_0}\in \ms P$. For both cases of $\P^{\bbeta,\bs w,\bi}_{z_0}$ and $\P^{\bbeta,\bs w}_{z_0}$, the stopping times $T^m_0$ and $T^{\bi,m}_0$ are a.s. finite for all $m\in\Bbb N$ when $\bbeta$ is $\bs w$-homogeneous and $\#\{w_\bj;w_\bj>0\}\geq 2$ by Proposition~\ref{prop:SDE2} (5$\cc$). We have verified all the properties stated in Theorem~\ref{thm:main1} (1$\cc$). 
 
Theorem~\ref{thm:main1} (2$\cc$) can be justified as follows. By Proposition~\ref{prop:SDE2} (6$\cc$), $\ms P$ defines $\{\ms Z_t\}$ as a time-homogeneous Markov process. Here, we use Remark~\ref{rmk:SDE2} (1$\cc$) to justify the strong Markov property under $\P^{\bbeta,\bs w}_{z_0}$ for $z_0\in \CNw$. 
Also, \eqref{pathproperty} is due to the construction in (1$\cc$).
 
Finally, Theorem~\ref{thm:main1} (3$\cc$) is the consequence of the above concatenation construction of $\P\in \ms P$, Proposition~\ref{prop:SDE1} (4$\cc$), and Proposition~\ref{prop:SDE2} (2$\cc$). The proof of Theorem~\ref{thm:main1} is complete.
 
\section{The no-simultaneous-contacts (NSC) phenomenon}\label{sec:NTC}
Our goal in Section~\ref{sec:NTC} is to prove Theorem~\ref{thm:NTC}. It suffices to consider the case of
\begin{align}\label{ass:S<T}
\P(0=S<T<\infty)=1.
\end{align}
The additional assumption in \eqref{ass:S<T} can be justified as follows. First, we can assume $\P(S<\infty,T<\infty)=1$. Otherwise, note that
\begin{gather*}
\P(\exists\;t\in (S,T]\mbox{ s.t. }\varrho_t=0,\;S<T)\leq \sum_{\scriptstyle p<q\atop \scriptstyle p,q\in \Bbb Q_+} \P(\exists\;t\in (S\vee p ,T\wedge q]\mbox{ s.t. }\varrho_t=0,\;S\vee p<T\wedge q),\\
\{S\vee p<T\wedge q\}\subset \{S<T\},\quad \forall \;p,q\in \Bbb Q_+\mbox{ with }p<q,\\
[S\vee p,T\wedge q)\subset [S,T)\mbox{ on }\{S\vee p<T\wedge q\},\quad \forall \;p,q\in \Bbb Q_+\mbox{ with }p<q.
\end{gather*}
Hence, we can first work with these finite stopping times $S\vee p$ and $T\wedge q$ and use the assumed properties of $\{\varrho^\bj_t;S\leq t<T\}_{\bj\in \mc J}$ over $[S\vee p,T\wedge q)$ on $\{S\vee p<T\wedge q\}$.
Also, we can assume $\P(S<T)=1$ since whenever $\P(S<T)>0$, we can use the fact that $\{S<T\}\in \F_S$, provided that $(\F_t)$ is the underlying filtration, to transfer the assumed properties of $\{\varrho^\bj_t;S\leq t<T\}_{\bj\in \mc J}$ on $\{S<T\}$ to $\P(\,\cdot\,|S<T)$. Finally, the assumption that $\P(S=0)=1$ can be justified by a time shift. 

In the remaining of Section~\ref{sec:NTC}, we impose the additional condition in \eqref{ass:S<T} and take the following steps to prove Theorem~\ref{thm:NTC}. \medskip 
 
\noindent {\bf Step~1 (Emulating a Bessel process).}     
We start with the transformation of the martingale part of $\{\varrho_t;0\leq t<T\}$ by a time change.  By the Dambis--Dubins--Schwarz theorem \cite[(1.6) Theorem on p.181]{RY-2}, we can find a one-dimensional Brownian motion $\widetilde{\mc B}$ such that $\widetilde{\mc B}_0=0$ and the martingale part of $\{\varrho_t;0\leq t<T\}$ is expressible as $\widetilde{\mc B}_{\int_0^t \varsigma(s)\d s}$. That is, we have 
\begin{align}\label{extend:varrho}
\varrho_{t}=\varrho_0+\int_0^{t} \sum_{\bj\in \J}
 \frac{\mu_\bj(r)}{\varrho_r^\bj}\d r
 +\widetilde{\mc B}_{\int_0^t \varsigma(s)\d s},\quad 0\leq t< T.
\end{align}
Specifically, the clock process for $\widetilde{\mathcal B}$ is the continuous process given by
\begin{align}\label{def:varsigma}
\begin{aligned}
\int_0^t \varsigma(s)\d s
=\la \varrho,\varrho\ra_t=\left\langle\sum_{\bj\in \J} \mc B^\bj,\sum_{\bj\in\J}\mc B^\bj\right\rangle_{t}
,\quad 0\leq t<T.
\end{aligned}
\end{align}
With $n_0\,\defeq\,\#\mc J\in \{2,3,\cdots\}$ by assumption, $\varsigma(s)=\varsigma(s,\omega)$ satisfies the following for a.s.-$\omega$:
\begin{align}
0\leq \esssup \{\varsigma(s,\omega);s\in [0,T)\}&\leq n_0+\sum_{\scriptstyle \bj, \bk\in \mc J\atop \scriptstyle \bj\neq \bk}\frac{\esssup \{|\sigma_{\bj,\bk}(s,\omega)|;s\in [0,T)\}}{2}\notag\\
&\leq n_0+\frac{n_0(n_0-1)}{2},\label{bdd:varsigma}
\end{align}
where the second inequality in \eqref{bdd:varsigma} uses \eqref{W:qcv} with $S=0$ and the assumption that each $\mathcal B^\bj$ is a one-dimensional standard Brownian motion, and the third inequality uses the assumption that $\esssup \{|\sigma_{\bj,\bk}(s,\omega)|;s\in [0,T)\}\leq 1$ a.s. for all $\bj\neq \bk$. In particular, by the present assumption $\P(T<\infty)=1$ and \eqref{bdd:varsigma},
\[
\la \varrho,\varrho\ra_T\defeq \lim_{t\nearrow T}\ua \la \varrho,\varrho\ra_t<\infty.
\]

We now define a time change by
\[
\gamma(\ell)\;\defeq\, \inf\{t\geq 0;\la \varrho,\varrho\ra_{t\wedge T}>\ell\}\wedge T.
\]
That is, $\gamma(\ell)$ is the right-continuous inverse of $t\mapsto \la \varrho,\varrho\ra_{t\wedge T}$, but levelled out at level $T$. 
To use this time change, we set
\begin{align}\label{NTC:def}
L\,\defeq \,\la \varrho,\varrho\ra_{T},\quad \gamma(0-)\,\defeq \,0.
\end{align}
The first definition in \eqref{NTC:def} implies 
\begin{align}\label{NTC:gamma}
\gamma(\ell)= T,\quad \forall\;\ell\geq L, 
\end{align}
and the second definition in \eqref{NTC:def} is chosen for the convenience of the following lemma when we use \eqref{NTC:rho} for $\ell=0$. {Together with the particular form of the finite-variation part in \eqref{eq:ntc},
this lemma will help handle $\varrho$ over possible intervals of constancy of $\gamma$.}

\begin{lem}\label{lem:NTCinterval}
It holds that  
\begin{align}\label{NTC:rho}
t\mapsto \varrho_t\mbox{ is strictly increasing on $[\gamma(\ell-),\gamma(\ell)]$,}\quad \forall\; 0\leq \ell\leq L.
\end{align}
In particular, if $L=0$, then $\varrho_t>0$ for all $t\in (0,T]$.
\end{lem}
\begin{proof}
This strict monotonicity is trivial when $\gamma(\ell-)=\gamma(\ell)$.
To see \eqref{NTC:rho} when $\gamma(\ell-)<\gamma(\ell)$, note that by the right-continuity of $\gamma(\cdot)$, 
$\la \varrho,\varrho\ra_{\gamma(\ell-)}=\la \varrho,\varrho\ra_{\gamma(\ell)}$, that is,
\begin{align}\label{BM:jump}
\left\langle \sum_{\bj\in \J} \mc B^\bj,\sum_{\bj\in \J} \mc B^\bj\right\rangle_{\gamma(\ell-)}=\left\langle \sum_{\bj\in \J} \mc B^\bj,\sum_{\bj\in \J} \mc B^\bj\right\rangle_{\gamma(\ell)}. 
\end{align}
This implies $\sum_{\bj\in \J} \mc B^\bj_{t}=\sum_{\bj\in \J} \mc B^\bj_{\gamma(\ell-)}$ for all $\gamma(\ell-)\leq t\leq \gamma(\ell)$ \cite[(1.13) Proposition on p.125]{RY-2}. Hence, \eqref{NTC:rho} holds by  \eqref{eq:ntc} along with the nonnegativity of $\varrho^\bj$ and the assumed $s$-a.e. inequality $\mu_\bj(s,\omega)\geq (1-2\alpha)/2>0$. 
\end{proof}

Since the remaining of the proof of Theorem~\ref{thm:NTC} uses only pathwise arguments, {\bf we assume $\bs L\bs >\bs 0$ from now on.} By Lemma~\ref{lem:NTCinterval}, $L=0$ implies that $t\mapsto \varrho_t$ is strictly increasing on $[\gamma(0-),\gamma(0)]=[0,T]$, so the required conclusion of Theorem~\ref{thm:NTC} plainly holds.

Having transformed the noise term of $\varrho$, we now define for all $0\leq \ell\leq L$, 
\begin{align}\label{def:rho2}
\varrho^{(2)}_\ell\,\defeq \,\varrho_{\gamma(\ell)}=\sum_{\bj\in \J}\varrho^\bj_{\gamma(\ell)},\quad A_\ell\,\defeq\int_{\gamma(0)}^{\gamma(\ell)}\sum_{\bj\in \J}
 \frac{\mu_\bj(r)}{\varrho_r^\bj}\d r,
\end{align}
where we define $\varrho_T$ according to Remark~\ref{rmk:ntc} (2$\cc$). 
 In the following lemma, \eqref{eq:rho21} specifies a basic form by which we will compare $\varrho^{(2)}$ with Bessel processes
$\varrho^{(1),k}$ to be defined in Step~3. 

\begin{lem}
It holds that 
\begin{gather}
\varrho^{(2)}_\ell=\varrho^{(2)}_0+A_\ell+\widetilde{\mc B}_\ell,\quad 0\leq \ell\leq L,\label{eq:rho2}\\
\label{eq:rho21}
\varrho_L^{(2)}=\varrho_T\quad A_L=\int_{\gamma(0)}^T\sum_{\bj\in \J}\frac{\mu_\bj(r)}{\varrho^\bj_r} \d r\in [0,\infty].
\end{gather}
\end{lem}
\begin{proof}
By \eqref{extend:varrho} and the fact that $\wt{\mc B}_0=0$, 
\eqref{eq:rho2} as soon as we have
\[
\widetilde{\mc B}_{\int_0^{\gamma(\ell)}\varsigma(s)\d s}=\widetilde{\mc B}_\ell,\quad 0\leq \ell\leq L.
\]
This equality can be seen by considering two cases: (1) for $\ell<L$, $t\mapsto \int_0^t\varsigma(s)\d s$ is a continuous function, and so,  $\int_0^{\gamma(\ell)}\varsigma(s)\d s=\ell$; (2) for $\ell=L$, 
\[
\int_0^{\gamma(L)}\varsigma(s)\d s=\int_0^T\varsigma(s)\d s=\la \varrho,\varrho\ra_{T}=L,
\] 
where the three equalities follow from \eqref{NTC:gamma}, the first equality in \eqref{def:varsigma}, and the first equality in \eqref{NTC:def}, respectively. Finally, \eqref{eq:rho21} follows immediately from \eqref{NTC:gamma}. 
\end{proof}

To emulate the comparing Bessel process based on \eqref{eq:rho2} requires a bound on $\{A_\ell;0\leq \ell\leq L\}$. 
The following lemma serves as the starting point.  

\begin{lem}\label{lem:igbdd}
For all integers $n\geq 2$, 
\[
\frac{1}{x_1}+\cdots+\frac{1}{x_n}\geq \frac{n^2}{x_1+\cdots+x_n},\quad \forall\;x_1,\cdots,x_n\geq 0.
\]
Moreover, equality holds if and only if (1) $x_1=\cdots=x_n$ or (2) some of the $x_j$'s is zero.
\end{lem}
\begin{proof}
I suffices to consider the case of $x_1,\cdots,x_n>0$. 
We begin by rewriting the inequality in question as the equivalent inequalities in the following five lines:
\begin{gather*}
\frac{\sum_{j=1}^n \prod_{i\neq j}x_i}{\prod_{i=1}^n x_i}\geq \frac{n^2}{\sum_{k=1}^nx_k},\\
\sum_{j,k=1}^n x_k\prod_{i\neq j}x_i\geq n^2\prod_{i=1}^n x_i,\\
\sum_{j<k}x_k\prod_{i\neq j}x_i+\sum_{j>k}x_k\prod_{i\neq j}x_i-(n^2-n)\prod_{i=1}^n x_i\geq 0,\\
\sum_{j<k}x_k^2\prod_{i\neq j,k}x_i+\sum_{k>j}x_j^2\prod_{i\neq j,k}x_i-\sum_{j<k}2\prod_{i=1}^n x_i\geq 0,\\
\sum_{j<k}(x_k^2-2x_jx_k+x_j^2)\prod_{i\neq j,k}x_i\geq 0.
\end{gather*}
In more detail, we have changed variables by replacing $(j,k)$ with $(k,j)$ to get the second term in the fourth line.  
Since the last inequality obviously holds true, the required inequality of the lemma holds. Moreover, the last equality shows that equality holds if and only if $x_1=\cdots=x_n$. This completes the proof of the lemma. 
\end{proof}

The following lemma concludes Step~1 by bounding $\{A_\ell;0\leq \ell\leq L\}$. 

\begin{lem}\label{lem:NTCAbdd}
For all $0\leq \ell_1\leq \ell_2\leq L$, it holds that 
\begin{align}\label{nc:bdd}
\begin{aligned}
&\int_{(\ell_1,\ell_2]}\d A_{\ell'}\geq  \int_{(\ell_1,\ell_2]} \frac{\mathfrak d-1}{2\varrho^{(2)}_{\ell'}}\d \ell',  
\quad\mbox{where}\quad \mathfrak d\,\defeq\, \frac{n_0^2(1-2\alpha)}{n_0+\frac{n_0(n_0-1)}{2}}+1\geq 2.
\end{aligned}
\end{align}
\end{lem}
\begin{proof}
The lower bound for $\mathfrak d$ in \eqref{nc:bdd} follows immediately from \eqref{def:alpha}.
To see the bound for $\int_{(\ell_1,\ell_2]}\d A_{\ell'}$ in \eqref{nc:bdd}, first, we use the assumed $r$-a.e. lower bound $\mu_\bj(r,\omega)\geq (1-2\alpha)/2$, Lemma~\ref{lem:igbdd} and the second inequality 
in \eqref{bdd:varsigma}
to get the second inequality below:
\begin{align*}
\int_{(\ell_1,\ell_2]}\d A_{\ell'}=A_{\ell_2}-A_{\ell_1}=\int_{\gamma(\ell_1)}^{\gamma(\ell_2)} \sum_{\bj\in \J}\frac{\mu_\bj(r)}{\varrho_r^\bj}\d r&\geq 
\frac{n_0^2(1-2\alpha)}{2
(n_0+\frac{n_0(n_0-1)}{2})
}\int_{\gamma(\ell_1)}^{\gamma(\ell_2)}\frac{1}{\varrho_r}\varsigma(r)\d r.
\end{align*}
To proceed, note that $r\leq \gamma_{\la \varrho,\varrho\ra_r}$ for all $0\leq r\leq T$. The strict inequality $r< \gamma_{\la \varrho,\varrho\ra_r}$ holds only if $r\in [\gamma(\ell-),\gamma(\ell))$ when $\gamma(\ell)-\gamma(\ell-)>0$, in which case
$1/\varrho_r\geq 1/\varrho_{\gamma_{\la \varrho,\varrho\ra_r}}$ by Lemma~\ref{lem:NTCinterval}. In any case, 
\[
\int_{\gamma(\ell_1)}^{\gamma(\ell_2)}\frac{1}{\varrho_r}\varsigma(r)\d r\geq \int_{\gamma(\ell_1)}^{\gamma(\ell_2)}\frac{1}{\varrho_{\gamma_{\la \varrho,\varrho\ra_r}}}\varsigma(r)\d r=\int_{\gamma(\ell_1)}^{\gamma(\ell_2)}\frac{1}{\varrho_{\gamma_{\la \varrho,\varrho\ra_r}}}\d \la \varrho,\varrho\ra_r.
\]
Also, if we denote $\widetilde{A}_{\ell'}\equiv \ell' $, then $\widetilde{A}_{\la \varrho,\varrho\ra_r}=\la\varrho,\varrho\ra_r$, and so, 
\begin{align*}
\int_{\gamma(\ell_1)}^{\gamma(\ell_2)}\frac{1}{\varrho_{\gamma_{\la \varrho,\varrho\ra_r}}}\d \la \varrho,\varrho\ra_r&=\int_{[\gamma(\ell_1),\gamma(\ell_2)]}\frac{1}{\varrho_{\gamma_{\la \varrho,\varrho\ra_r}}}\d \widetilde{A}_{\la \varrho,\varrho\ra_r}.
\end{align*}
Third, since $t\mapsto \la \varrho,\varrho\ra_t$ is continuous and non-decreasing, we can apply a change of variables formula suitable for time changes \cite[(4.10) Proposition, p.9]{RY-2} to get 
\begin{align*}
\int_{[\gamma(\ell_1),\gamma(\ell_2)]}\frac{1}{\varrho_{\gamma_{\la \varrho,\varrho\ra_r}}}\d \widetilde{A}_{\la \varrho,\varrho\ra_r}
=\int_{[\la \varrho,\varrho\ra_{\gamma(\ell_1)},\la \varrho,\varrho\ra_{\gamma(\ell_1)}] }\frac{1}{\varrho_{\gamma(\ell')}}\d \widetilde{A}_{\ell'}=\int_{[\ell_1,\ell_2] }\frac{1}{\varrho_{\gamma(\ell')}}\d \ell'=\int_{\ell_1}^{\ell_2}\frac{1}{\varrho^{(2)}_{\ell'}}\d \ell',
\end{align*}
where the second equality uses the continuity of $t\mapsto \la \varrho,\varrho\ra_t$ to get 
$\la \varrho,\varrho\ra_{\gamma(\ell_j)}=\ell_j$. Combining the last four displays proves the bound for $\int_{(\ell_1,\ell_2]}\d A_{\ell'}$ in \eqref{nc:bdd}.
\end{proof}

 Lemma~\ref{lem:NTCAbdd} has considered the fact that under the current assumptions, $\{A_\ell;0\leq \ell\leq L\}$ is not continuous and has only c\`adl\`ag paths. At every point of discontinuity, it only jumps upward. A closely related property is that although the sample paths of $\varrho^{(2)}$ are only c\`adl\`ag, Lemma~\ref{lem:NTCinterval} gives
\begin{align}\label{rho:jump}
\varrho^{(2)}_{\ell}-\varrho^{(2)}_{\ell-}\geq 0,\quad \forall\;0\leq \ell\leq L.
\end{align}
This property will be be useful in Steps~4 and 5.

\noindent {\bf Step~3 (Choosing lower-bound Bessel processes).} 
For all integers $k\geq 1$, define
\[
\sigma_{k}\defeq \inf\{\ell\geq 0;\varrho_\ell^{(2)}\geq k^{-1}\}\wedge L,
\]
where $L$ is defined in \eqref{NTC:def} and has been assumed to be strictly positive after Lemma~\ref{lem:NTCinterval}. (The effect we want from $\sigma_k$ is nonzero initial conditions, as will be made precise in Lemma~\ref{lem:sigmaklbd}.)  
Keeping in mind the first inequality in \eqref{nc:bdd}, we now choose lower-bound Bessel processes for $\varrho^{(2)}$, denoted by $\varrho^{(1),k}$, to be of dimension $\mathfrak d$ starting from $\varrho_{\sigma_k}^{(2)}$ at time $\sigma_{k}$ as follows: 
\begin{align}\label{def:varrho1}
\varrho^{(1),k}_{\ell}=\varrho_{\sigma_{k}}^{(2)}+\int_{\sigma_{k}}^{\ell} \frac{\mathfrak d-1}{2\varrho^{(1),k}_{\ell'}}\d \ell'+\widetilde{\mc B}_{\ell}-\widetilde{\mc B}_{\sigma_{k}},\quad\sigma_{k}\leq \ell<\infty.
\end{align}
The strong existence of these processes $\varrho^{(1),k}_{\ell}$ results from the strong existence of the SDE of the Bessel squared process of dimension $\mathfrak d$ since $\mathfrak d\geq 2$ by \eqref{nc:bdd}. 

Two basic properties we need from the above lower-bound processes are as follows. First, with probability one, since $\mathfrak d\geq 2$, we have
\begin{align}\label{NTC:lbd}
\varrho^{(1),k}_{\ell}>0,\quad \forall\;\ell>\sigma_k,\;\forall\;k\geq 1.
\end{align}
This property can be seen by using the polarity of one-point sets of the two-dimensional Brownian motion \cite[(2.7) Proposition, p.191]{RY-2} and the comparison theorem of SDEs \cite[(2.18) Proposition, p.293]{KS:BM-2} specialized to the SDEs of the Bessel squared processes of dimenesion two and of dimension $\mathfrak d$.
The second property is that by \eqref{eq:rho2} and \eqref{def:varrho1},
\[
\Delta_{\ell}(k)\,\defeq\, \varrho^{(1),k}_{\ell}-\varrho^{(2)}_{\ell},\quad \sigma_{k}\leq \ell\leq L, 
\]
satisfies the following \emph{ordinary} integral equation:
\begin{align}\label{eq:Delta}
\Delta_{\ell}(k)=\int_{(\sigma_{k},\ell]}\left(\frac{\mathfrak d-1}{2\varrho^{(1),k}_{\ell'}}\d \ell'-\d A_{\ell'}\right),\quad \sigma_{k}\leq \ell\leq L.
\end{align}

\noindent {\bf Step~4 (Comparing processes).}
We first specify a smooth regularization $\{\varphi_n\}$ of the positive-part function $x\mapsto x^+$:
\[
\varphi_n(\cdot)\,\defeq\,\psi_n(\cdot)\1_{(0,\infty)}(\cdot), 
\]
where $\{\psi_n\}$ is chosen to be the smooth regularization of the absolute-value function in the Yamada--Watanabe pathwise uniqueness method \cite[Theorem~1]{YW:PU-2}. These functions $\{\psi_n\}$ are defined as follows. Take a strictly decreasing sequence $\{a_n\}_{n=0}^\infty\subset(0,1]$ with $a_0=1$, $\lim_{n\to\infty}a_n=0$ and $\int_{(a_n,a_{n-1})}u^{-2}\d u=n$ for every $n\geq 1$. For each $n\geq 1$, there exists a continuous function $\rho_n(\cdot)$ on $\R$ with support in $(a_n,a_{n-1})$, so that $0\leq \rho_n(x)\leq 2(nx^{2})^{-1}$ for every $x>0$ and $\int_{(a_n,a_{n-1})}\rho_n(x)\d x=1$. Then 
\[
\psi_n(x)\,\defeq\int_0^{|x|}\int_0^y \rho_n(u)\d u\d y, \quad x\in \R,
\]
is even and twice continuous differentiable with $|\psi_n'(x)|\leq 1$ and $\lim_{n\to\infty}\psi_n(x)=|x|$ for all $x$. 

Now, by the change of variables for Stieltjes integrals \cite[(4.6) Proposition on p.6]{RY-2}, applying \eqref{eq:Delta} and the first inequality in \eqref{nc:bdd} in the same order gives
\begin{align}
\varphi_n(\Delta_\ell(k))&=\int_{(\sigma_k,\ell]} \varphi_n'(\Delta_{\ell'}(k))\left(\frac{\mathfrak d-1}{2\varrho^{(1),k}_{\ell'}}\d \ell'-\d A_{\ell'}\right)\notag\\
&\leq \int_{(\sigma_k,\ell]} \varphi_n'(\Delta_{\ell'}(k))\left(\frac{\mathfrak d-1}{2\varrho^{(1),k}_{\ell'}}-\frac{\mathfrak d-1}{2\varrho^{(2)}_{\ell'}}\right)\d {\ell}',\quad \forall\;\sigma_k\leq \ell\leq L.\label{YW:ineq}
\end{align}
We have also used $\varphi_n(0)=0$ to get the equality and $\varphi_n'\geq 0$ over $\R$ to get the inequality. 

To close the rightmost side of the first inequality in \eqref{YW:ineq}, we first prove the following lemma. Recall once again that we have assumed $L>0$ after Lemma~\ref{lem:NTCinterval}.

\begin{lem}\label{lem:sigmaklbd}
It holds that 
\begin{align}\label{def:sigmak}
\sigma_k<L\mbox{ and $\varrho^{(2)}_{\sigma_k}\geq k^{-1}$ for all large $k=k(\omega)$, and $\displaystyle \lim_{k\to\infty}\da \sigma_k=0$}.
\end{align}
\end{lem}
\begin{proof}
We consider separately the two cases: $\gamma(0)=0$ and $\gamma(0)>0$. If $\gamma(0)=0$, it follows from the lower bound for $\int_{(\ell_1,\ell_2]}\d A_{\ell'}$ in \eqref{nc:bdd} that $\max_{\ell'\in [0,\vep]}\varrho^{(2)}_{\ell'}>0$ for all $\vep>0$. Since $L>0$, the required properties in \eqref{def:sigmak} follow. If $\gamma(0)>0$,  Lemma~\ref{lem:NTCinterval} implies $\varrho_{\gamma(0)}>0$ so that \eqref{def:sigmak} plainly holds. 
\end{proof}

Define, for all integers $m\geq k+1$,
\begin{align}\label{def:taumk}
\tau_m(k)\,\defeq \,\inf\left\{\ell\geq \sigma_k;
\min\{\varrho^{(2)}_\ell,\varrho^{(1),k}_\ell\}\leq m^{-1}\right\}\wedge L.
\end{align}
Then it follows from \eqref{def:sigmak} that for all large $k=k(\omega)$ and all $m\geq k+1$, 
\begin{align}\label{mk:bdd}
\sigma_k<\tau_m(k)\mbox{ and }\min\{\varrho^{(2)}_\ell,\varrho^{(1),k}_\ell\}\geq m^{-1},\; \forall\; \sigma_k\leq \ell\leq \tau_m(k).
\end{align}
Here, we use \eqref{rho:jump} to get $\varrho^{(2)}_{\tau_m(k)}\geq  m^{-1}$ when $\varrho^{(2)}$ jumps at $\tau_m(k)$ and the sample-path continuity of $\varrho^{(1),k}$ to get  $\varrho^{(1),k}_{\tau_m(k)}\geq  m^{-1}$. {\bf We only consider these $\bs k$ and $\bs m$ in the remaining of the proof of Theorem~\ref{thm:NTC}.} 

Let $L_m$ be the Lipschitz constant of $x\mapsto (\mathfrak d-1)/(2x)$ over $[m^{-1},\infty)$. Note that $\varphi_n$ is supported in $(0,\infty)$ so $\varphi_n'(\Delta_\ell(k))\neq 0$ only if $\Delta_\ell(k)>0$,  that is, only if $\varrho^{(1),k}_\ell-\varrho^{(2)}_\ell> 0$, in which case $\varphi_n'(\Delta_\ell(k))=\psi_n'(\Delta_\ell(k))\leq 1$ by the choice of $(\varphi_n,\psi_n)$. By \eqref{YW:ineq} and \eqref{mk:bdd},
\begin{align*}
\varphi_n(\Delta_{\sigma_{k}+\ell}(k))
&\leq \int_{0}^{\ell} \varphi_n'(\Delta_{\sigma_k+\ell'}(k))L_m\left(\varrho^{(1),k}_{\sigma_k+\ell'}-\varrho^{(2)}_{\sigma_k+\ell'}\right)^+\d \ell'\\
&\leq \int_{0}^{\ell} L_m\Delta_{\sigma_k+\ell'}(k)^+\d \ell',\quad \forall\;0\leq \ell\leq \tau_m(k)-\sigma_{k}.
\end{align*}
Passing $n\to\infty$ for the leftmost side yields the closed inequality
\[
\Delta_{\sigma_{k}+\ell}(k)^+\leq L_m\int_{0}^{\ell} \Delta_{\sigma_k+\ell'}(k)^+\d \ell',\quad\forall\; 0\leq \ell\leq \tau_m(k)-\sigma_{k}.
\]
Gr\"onwall's lemma is applicable to the foregoing inequality due to the boundedness of $\Delta_\ell(k)^+$ on compacts in $[\sigma_k,\infty)$ by the sample-path continuity of $\varrho^{(1),k}$ and the nonnegativity of $\varrho^{(2)}$. Hence, 
\begin{align}\label{mk:bddconclusion}
&\Delta_\ell(k)=\varrho^{(1),k}_\ell-\varrho^{(2)}_\ell\leq 0,\quad \sigma_k\leq \ell\leq \tau_m(k),\quad \forall\;\mbox{large }k=k(\omega)\;\forall\;m\geq k+1.
\end{align}

\noindent {\bf Step~5 (Extending to all the required times).} In this step, we transfer \eqref{mk:bddconclusion} to the following inequalities for $\varrho^{(2)}$ and $\varrho$, one after the other: 
\begin{align}
\inf_{\ell\in [\vep,L]}\varrho^{(2)}_\ell> 0\quad &\forall\; 0<\vep< L,
\label{NTC:ineq1}\\
\varrho_t>0,\quad &\forall\;0<t\leq T.
\label{NTC:ineq2}
\end{align}

To obtain \eqref{NTC:ineq1}, let $\tau_\star\,\defeq\,\lim_m\ua\,\tau_m(k)$. If $\tau_\star<L$, then the following holds for all large $k$:
 \begin{align}
0=\varrho^{(2)}_{\tau_\star-}\geq  \varrho^{(1),k}_{\tau_\star}>0&\mbox{ if }\tau_m(k)<\tau_\star\mbox{ for all $m\geq k+1$},\label{NTC:c2-1}\\
0=\varrho^{(2)}_{\tau_\star}\geq  \varrho^{(1),k}_{\tau_\star}>0&\mbox{ if }\tau_m(k)=\tau_\star\mbox{ for some $m\geq k+1$}.\label{NTC:c2-2}
 \end{align} 
Here,  \eqref{NTC:lbd} and the bound on $\varrho^{(2)}_{\sigma_k}$ in \eqref{def:sigmak} imply $\varrho^{(1),k}_{\tau_\star}>0$; \eqref{def:taumk} implies the first equalities in \eqref{NTC:c2-1} and \eqref{NTC:c2-2} now that $\tau_\star<L$; \eqref{mk:bddconclusion} implies the first inequalities in \eqref{NTC:c2-1} and \eqref{NTC:c2-2}.
Since each of \eqref{NTC:c2-1} and \eqref{NTC:c2-2} shows a contradiction, we must have $\tau_\star=L$. 
Then by \eqref{rho:jump}, \eqref{NTC:lbd} and \eqref{mk:bddconclusion}, the following holds for all large $k$:
\begin{align}\label{NTC:c2000}
\varrho^{(2)}_\ell\geq \varrho^{(1),k}_\ell>0,\quad \forall\;\sigma_k\leq \ell\leq L. 
\end{align}
By \eqref{NTC:c2000}, we obtain  \eqref{NTC:ineq1} from the limit of $\sigma_k$ in \eqref{def:sigmak}.

To get \eqref{NTC:ineq2}, first, it is immediate from \eqref{NTC:ineq1} that $\varrho_t>0$ whenever $t=\gamma(\ell)$ or $\gamma(\ell-)$ for some $0<\ell\leq L$. To handle the remaining times $t\in (0,T]$, note that $\gamma(L)=T$ by \eqref{NTC:gamma}, and so, the increasing monotonicity of $\gamma(\cdot)$ gives
 \[
 (0,T]\setminus\{\gamma (\ell),\gamma(\ell-);0<\ell\leq L\}=(\gamma(0-),\gamma(0)]\cup \bigcup_{0<\ell\leq L}(\gamma(\ell-),\gamma(\ell)).
 \] 
By Lemma~\ref{lem:NTCinterval}, we get $\varrho_t>0$ when $t\in(0,T]$ and $t$ is not given by $\gamma(\ell)$ or $\gamma(\ell-)$ for some $0<\ell\leq L$. We have proved that $\varrho_t>0$ for all $0<t\leq T$. The proof of Theorem~\ref{thm:NTC} is complete.

\section{It\^{o}'s formulas for the log-sums}\label{sec:logsum}
Our goal in this section is to prove Proposition~\ref{prop:logsum}. For the proof, we write 
$R^\bj_t\,\defeq\,|Z^\bj_t|^2$ under $\P^\bi_{z_0}$ for any fixed $z_0\in \CN$ so that the following SDEs hold \cite[(4.4)]{C:SDBG1-2}:
\begin{align}\label{ItoR}
R^\bj_t
&=R^\bj_0+\int_0^t2\Biggl(1-\frac{\sigma(\bj)\cdot \sigma(\bi)}{2}
 \Re\left(\frac{Z^\bj_s}{Z^\bi_s}\right)\frac{ \K^{\bbeta,\bi}_{1}(s)}{ K^{\bbeta,\bi}_0(s)}\Biggr) \d s+\int_0^t 2|Z^\bj_s|\d B^\bj_s,\quad \forall\;\bj\in \mc E_N,
\end{align}
where $B^\bj$ follows the definition in \eqref{def:Bj-2}.
Also, $\{\bs R_t\}\,\defeq\,\{R^\bj_t\}_{\bj\in \mc E_N}$ so that $R^\bj_t$ is the $\bj$-th component of $\bs R_t$, whereas $R,R_\bj$ denote $\R_+$-valued variables and $\bs R=\{R_\bj\}_{\bj\in \mc E_N}$. Finally, with $\hK_\nu(x)\,\defeq\, x^\nu K_\nu(x)$,
\begin{align}\label{def:Gnu}
G_\nu(R)\,\defeq \,\widehat{K}_\nu(\sqrt{R}\,),\quad R>0,
\end{align}
which is extended to $R=0$ for any $\nu\in (0,\infty)$ by the following limit \cite[(5.16.4), p.136]{Lebedev-2}:
\begin{align}\label{def:hK0}
\widehat{K}_\nu(0)\,\defeq\,\lim_{x\searrow 0}\widehat{K}_\nu(x)=2^{\nu-1} \Gamma (\nu).
\end{align}

\subsection{Formulas at the approximate level}\label{sec:logsum-1}
Let $\bs \vep=\{\vep_\bj\}_{\bj\in \mc E_N}\in (0,\infty)^{\mc E_N}$, and let $F^\bi_{\vep_\bi}$ and $F^{\bs w}_{\bs \vep}$ be $\C^2$-functions defined in open sets containing $[0,\infty)$ and $[0,\infty)^{\mc E_N}$ by
\begin{align}
 F^\bi_{ \vep_\bi}(R_\bi)&\;\defeq \;\log H^\bi_0(\vep_\bi+R_\bi)\quad\mbox{for } H^\bj_0(R_\bj)\,\defeq\,  w_\bj   G_0(2\beta_\bj R_\bj),\;\forall\;\bj\in \mc E_N,\label{FUNC1}\\
 F^{\bs w}_{\bs \vep}(\bs R)&\;\defeq \;\log H^{\bs w}_0(\bs \vep+\bs R)\quad\mbox{for } H^{\bs w}_0(\bs R)\,\defeq\, \sum_{\bj\in \mathcal E_N} w_\bj G_0(2\beta_\bj R_\bj).\label{FUNC2}
\end{align}
Note that by Proposition~\ref{prop:BMntc}, Proposition~\ref{prop:logsum} (1$\cc$) and Remark~\ref{rmk:ntc} (2$\cc$)-(a), 
\begin{align}\label{Ito:scheme}
\begin{aligned}
\P^\bi_{z_0}\left(\lim_{\bs \vep\searrow \bs 0}F^{\bs w}_{\bs \vep}(\bs R_t)-F^\bi_{\vep_\bi}(R^\bi_t)=\log \frac{K_0^{\bbeta,\bs w}(t)}{w_\bi K_0^{\bbeta,\bi}(t)},\; \forall\;t\geq 0\right)=1,\quad \forall\;z_0\in \CNwni,
\end{aligned}
\end{align} 
where $\bs 0\in [0,\infty)^{\mathcal E_N}$ denotes the zero vector.

The following two lemmas apply the standard version of It\^{o}'s formula and specify decompositions of 
$F^\bi_{\vep_\bi}(R^\bi_t)$ and $F^{\bs w}_{\bs \vep}(\bs R_t)$ finer than the semimartingale decompositions from It\^{o}'s formula. These finer decompositions will induce suitable limits in Section~\ref{sec:logsumA}--\ref{sec:logsumC}. 

\begin{lem}\label{lem:ItoKi}
For all $\vep_\bi\in (0,\infty)$ and $z_0\in \Bbb C$, it holds that under $\P^\bi_{z_0}$,
\begin{align}\label{eq:ItoKi}
F^\bi_{\vep_\bi}(R^\bi_t)=F^\bi_{\vep_\bi}(R^\bi_0)+\sum_{i=1}^5I_{\ref{eq:ItoKi}}(i)_t.
\end{align}
Here, the terms $I_{\ref{eq:ItoKi}}(i)_t=I_{\ref{eq:ItoKi}}(i)$ are defined as follows:
\begin{align*}
I_{\ref{eq:ItoKi}}(1)&\;\defeq\,-w_\bi(2\beta_\bi)\int_0^t \frac{G_{1}(R)}{2R}\Bigg|_{R=2\beta_\bi(\vep_\bi+R^\bi_s)}\times \frac{2}{H^\bi_0(\vep_\bi +R^\bi_s)}\d s,\\
I_{\ref{eq:ItoKi}}(2)&\;\defeq\, -w_\bi (2\beta_\bi)\int_0^t \frac{G_{1}(R)}{2R}\Bigg|_{R=2\beta_\bi(\vep_\bi+R^\bi_s)}\times \frac{2|Z^\bi_s|}{H_0^\bi(\vep_\bi+R^\bi_s)}\d B^\bi_s,\\
I_{\ref{eq:ItoKi}}(3)&\;\defeq\, \frac{1}{2}w_\bi(2\beta_\bi)^2\int_0^t \frac{2G_1(R)+RG_0(R)}{4R^{2}}\Bigg|_{R=2\beta_\bi(\vep_\bi+R^\bi_s)}\times \frac{4R_s^\bi}{H^\bi_0( \vep_\bi+R^\bi_s)}\d s,\\
I_{\ref{eq:ItoKi}}(4)&\;\defeq\, -\frac{1}{2}w_\bi^2(2\beta_\bi)^2\int_0^t \frac{G_{1}(R)^2}{4R^{2}}\Bigg|_{R=2\beta_\bi(\vep_\bi+R^\bi_s)}\times \frac{4R^\bi_s}{H^\bi_0(\vep_\bi+R^\bi_s)^2}\d s,\\
I_{\ref{eq:ItoKi}}(5)&\;\defeq\,w_\bi(2\beta_\bi)\int_0^t \frac{G_{1}(R)}{2R}\Bigg|_{R=2\beta_\bi(\vep_\bi+R^\bi_s)}\times \frac{2}{H^\bi_0(\vep_\bi +R^\bi_s)}\left( \frac{ \K^{\bbeta,\bi}_{1}(s)}{ K^{\bbeta,\bi}_0(s)}\right)\d s.
\end{align*} 
\end{lem}
\begin{proof}
We use the following derivative formulas:
\begin{align*}
\frac{\d}{\d R}\log G_0(R)=-\frac{G_1(R)}{2RG_0(R)},\quad \frac{\d^2}{\d R^2}\log G_0(R)=\frac{G_1(R)}{2R^2G_0(R)}+\frac{1}{4R}-\frac{G_{1}(R)^2}{4R^2G_0(R)^2};
\end{align*}
see \cite[(4.12) and (4.13)]{C:BES-2} with $\alpha=0$ or the proof of Lemma~\ref{lem:ItoKw}.
Hence, by \eqref{ItoR},  
\begin{align*}
I_{\ref{eq:ItoKi}}(1)+I_{\ref{eq:ItoKi}}(2)+I_{\ref{eq:ItoKi}}(5)&=\int_0^t \frac{\d F^\bi_{\vep_\bi}}{\d R_\bi}(R^\bi_s) \d R^\bi_s,\quad
I_{\ref{eq:ItoKi}}(3)+I_{\ref{eq:ItoKi}}(4)=\frac{1}{2}\int_0^t \frac{\d^2F^\bi_{ \vep_\bi}}{\d R_\bi^2}(R^\bi_s)\d \langle R^\bi,R^\bi\rangle_s,
\end{align*}
which is enough to get \eqref{eq:ItoKi}.
\end{proof}

\begin{lem}\label{lem:ItoKw}
For all $\bs\vep\in (0,\infty)^{\mc E_N}$ and $z_0\in \CN$, it holds that under $\P^{\bi}_{z_0}$, 
\begin{align}\label{eq:ItoKw}
F^{\bs w}_{\bs \vep}(\bs R_t)=F^{\bs w}_{\bs \vep}(\bs R_0)+\sum_{i=1}^9I_{\ref{eq:ItoKw}}(i)_t.
\end{align}
Here, the terms $I_{\ref{eq:ItoKw}}(i)_t=I_{\ref{eq:ItoKw}}(i)$ are defined as follows, using Kronecker's deltas  $\delta_{\bi,\bl}$:
\begin{align*}
I_{\ref{eq:ItoKw}}(1)&\;\defeq\,-\sum_{\bj\in \mathcal E_N}w_\bj(2\beta_\bj)\int_0^t \frac{G_{1}(R)}{2R}\Bigg|_{R=2\beta_\bj(\vep_\bj+R^\bj_s)}\times \frac{2}{H_0^{\bs w}(\bs \vep +\bs R_s)}\d s,\\
I_{\ref{eq:ItoKw}}(2)&\;\defeq\, -\sum_{\bj\in \mathcal E_N}w_\bj (2\beta_\bj)\int_0^t \frac{G_{1}(R)}{2R}\Bigg|_{R=2\beta_\bj(\vep_\bj+R^\bj_s)}\times \frac{2|Z^\bj_s|}{H_0^{\bs w}(\bs \vep+\bs R_s)}\d B^\bj_s,\\
I_{\ref{eq:ItoKw}}(3)&\;\defeq\, \frac{1}{2}\sum_{\bj\in \mathcal E_N}w_\bj(2\beta_\bj)^2\int_0^t \frac{2G_1(R)+RG_0(R)}{4R^{2}}\Bigg|_{R=2\beta_\bj(\vep_\bj+R^\bj_s)}\times \frac{4R_s^\bj}{H_0^{\bs w}(\bs \vep+\bs R_s)}\d s,\\
I_{\ref{eq:ItoKw}}(4)&\;\defeq\,  -\frac{1}{2}\sum_{\bj\in \mc E_N\setminus\{\bi\}}w_\bj^2(2\beta_\bj)^2\int_0^t \frac{G_{1}(R)^2}{4R^{2}}\Bigg|_{R=2\beta_\bj(\vep_\bj+R^\bj_s)}\times \frac{4R^\bj_s}{H_0^{\bs w}(\bs \vep+\bs R_s)^2}\d s,\\
I_{\ref{eq:ItoKw}}(5)&\;\defeq\, -\frac{1}{2}w_\bi^2(2\beta_\bi)^2\int_0^t \frac{G_{1}(R)^2}{4R^{2}}\Bigg|_{R=2\beta_\bi(\vep_\bi+R^\bi_s)}\times \frac{4R^\bi_s}{H_0^{\bs w}(\bs \vep+\bs R_s)^2}\d s,\\
I_{\ref{eq:ItoKw}}(6)&\;\defeq\, -\frac{1}{2}\sum_{\stackrel{\scriptstyle \bj,\bk\in \mathcal E_N\setminus\{\bi\}}{\bj\neq \bk}}w_\bj w_\bk(2\beta_\bj )(2\beta_\bk)\int_0^t \frac{G_{1}(R)}{2R}\frac{G_{1}(R')}{2R'}\Bigg|_{\stackrel{\scriptstyle R=2\beta_\bj(\vep_\bj+R^\bj_s)}{R'=2\beta_\bk(\vep_\bk+R^\bk_s)}}\\ &\quad\quad \times  \frac{4|Z^\bj_s| |Z^\bk_s| }{H_0^{\bs w}(\bs\vep+\bs R_s)^2}\d \la B^\bj,B^\bk\ra_s,\\
I_{\ref{eq:ItoKw}}(7)&\;\defeq\, -\sum_{\bj\in \mc E_N\setminus\{ \bi\}}w_\bj w_\bi(2\beta_\bj )(2\beta_\bi)\int_0^t \frac{G_{1}(R)}{2R}\frac{G_{1}(R')}{2R'}\Bigg|_{\stackrel{\scriptstyle R=2\beta_\bj(\vep_\bj+R^\bj_s)}{R'=2\beta_\bi(\vep_\bi+R^\bi_s)}}\times \frac{4|Z^\bj_s| |Z^\bi_s| }{H_0^{\bs w}(\bs\vep+\bs R_s)^2}\d \la B^\bj,B^\bi\ra_s,\\
I_{\ref{eq:ItoKw}}(8)&\;\defeq\, \sum_{\bj\in \mc E_N\setminus\{\bi\}}w_\bj(2\beta_\bj)\int_0^t \frac{G_{1}(R)}{2R}\Bigg|_{R=2\beta_\bj(\vep_\bj+R^\bj_s)}\times \frac{\sigma(\bj)\cdot \sigma(\bi)}{H_0^{\bs w}(\bs \vep +\bs R_s)}\left(
 \Re\left(\frac{Z^\bj_s}{Z^\bi_s}\right)\frac{ \K^{\bbeta,\bi}_{1}(s)}{ K^{\bbeta,\bi}_0(s)}\right)\d s,\\
I_{\ref{eq:ItoKw}}(9)&\;\defeq\, w_\bi(2\beta_\bi)\int_0^t \frac{G_{1}(R)}{2R}\Bigg|_{R=2\beta_\bi(\vep_\bi+R^\bi_s)}\times \frac{2}{H_0^{\bs w}(\bs \vep +\bs R_s)}\left(
\frac{ \K^{\bbeta,\bi}_{1}(s)}{ K^{\bbeta,\bi}_0(s)}\right)\d s.
\end{align*} 
\end{lem}

\begin{proof}
To find the partial derivatives of $\log \sum_{\bl\in \mc E_N}w_\bl G_0(R_\bl)$ up to the second order, we use
\begin{linenomath*}\begin{align}\label{hK:der}
(\d/\d x)[x^\nu K_\nu(x)]=-x^\nu K_{\nu-1}(x)\quad\forall\;\nu\in \R
\end{align}\end{linenomath*}
\cite[the third identity in (5.7.9) on p.110]{Lebedev-2} and the even parity of $\nu\mapsto K_{\nu}(x)$ \cite[(5.7.10) on p.110]{Lebedev-2}. Also,
note $(\d /\d R)F(f(R))=F'(f(R))f'(R)$ and $(\d^2 /\d R^2)F(f(R))=F''(f(R))f'(R)^2+F'(f(R))f''(R)$. Hence,
for $\bj\neq \bk$,
\begin{align}
\frac{\partial}{\partial R_\bj}\log \sum_{\bl\in \mc E_N}w_\bl G_0(R_\bl)&=\left.\left(\frac{-w_\bj x^{-1}\widehat{K}_{1}(x)}{ \sum_{\bl\in \mc E_N}w_\bl G_0(R_\bl)}\right)\right|_{x=\sqrt{R_\bj}}\left(\frac{1}{2\sqrt{R_\bj}}\right)=\frac{-w_\bj G_{1}(R_\bj)}{2R_\bj\sum_{\bl\in \mc E_N}w_\bl G_0(R_\bl)},\notag\\
\frac{\partial^2}{\partial R_\bj^2}\log \sum_{\bl\in \mc E_N}w_\bl G_0(R_\bl)&=\left(\frac{w_\bj x^{-2}\widehat{K}_{1}(x)+w_\bj K_{0}(x)}{\sum_{\bl\in \mc E_N}w_\bl G_0(R_\bl)}\right. \left.+\frac{w_\bj x^{-1}\widehat{K}_{1}(x)[-w_\bj x^{-1}\widehat{K}_{1}(x)]
}{[\sum_{\bl\in \mc E_N}w_\bl G_0(R_\bl)]^2}\right)\Bigg|_{x=\sqrt{R_\bj}}\notag\\
&\quad \times \left(\frac{1}{2\sqrt{R_\bj}}\right)^2+\left(\frac{-w_\bj G_{1}(R_\bj)}{2R_\bj\sum_{\bl\in \mc E_N}w_\bl G_0(R_\bl)}\right)
\Biggl(\frac{1}{-4R_\bj^{3/2}}\Biggr)\notag\\
&= \frac{2w_\bj G_{1}(R_\bj)+w_\bj R_\bj G_0(R_\bj)}{4R_\bj^{2}\sum_{\bl\in \mc E_N}w_\bl G_0(R_\bl)}
 -\frac{w_\bj^2 G_{1}(R_\bj)^2}{4R_\bj^{2}[\sum_{\bl\in \mc E_N}w_\bl G_0(R_\bl)]^2},\notag\\
 \frac{\partial^2}{\partial R_\bj\partial R_\bk}\log \sum_{\bl\in \mc E_N}w_\bl G_0(R_\bl)&=\left.
 \left(\frac{-w_\bj G_{1}(R_\bj)\cdot w_\bk x^{-1}\widehat{K}_{1}(x)}{2R_\bj[\sum_{\bl\in \mc E_N}w_\bl G_0(R_\bl)]^2}\right)
 \right|_{x=\sqrt{R_\bk}}\left(\frac{1}{2\sqrt{R_\bk}}\right)\notag\\
 &=\frac{-w_\bj w_\bk G_1(R_\bj)G_1(R_\bk)}{4R_\bj R_\bk[\sum_{\bl\in \mc E_N}w_\bl G_0(R_\bl)]^2 }.
\end{align}
By \eqref{ItoR}, it follows that 
\begin{align}
I_{\ref{eq:ItoKw}}(1)&=\sum_{\bj\in \mc E_N}\int_0^t \frac{\partial F^{\bs w}_{\bs \vep}}{\partial R_\bj}(\bs R_s)2\d s,\notag\\
I_{\ref{eq:ItoKw}}(2)&=\sum_{\bj\in \mc E_N}\int_0^t  \frac{\partial F^{\bs w}_{\bs \vep}}{\partial R_\bj}(\bs R_s)2|Z^\bj_s|\d B^\bj_s,\notag\\
I_{\ref{eq:ItoKw}}(8)+I_{\ref{eq:ItoKw}}(9)&=\sum_{\bj\in \mc E_N}\int_0^t  \frac{\partial F^{\bs w}_{\bs \vep}}{\partial R_\bj}(\bs R_s)\sigma(\bj)\cdot \sigma(\bi) \left(
 \Re\left(\frac{Z^\bj_s}{Z^\bi_s}\right)\frac{ \K^{\bbeta,\bi}_{1}(s)}{ K^{\bbeta,\bi}_0(s)}\right)\d s,\notag\\
I_{\ref{eq:ItoKw}}(3)+I_{\ref{eq:ItoKw}}(4)+I_{\ref{eq:ItoKw}}(5)&=\frac{1}{2}\sum_{\bj\in \mathcal E_N}\int_0^t \frac{\partial^2F^{\bs w}_{\bs \vep}}{\partial R_\bj^2}(\bs R_s)\d \langle R^\bj,R^\bj\rangle_s,\notag\\
I_{\ref{eq:ItoKw}}(6)+I_{\ref{eq:ItoKw}}(7)
&=\frac{1}{2}\sum_{\stackrel{\scriptstyle \bj,\bk\in \mc E_N}{ \bj\neq \bk}}\int_0^t \frac{\partial^2 F^{\bs w}_{\bs \vep}}{\partial R_\bj\partial R_{\bk}}(\bs R_s)\d \langle R^\bj,R^\bk\rangle_s,\notag
\end{align}
and these equations are enough to get \eqref{eq:ItoKw}.
\end{proof}

The remaining of Section~\ref{sec:logsum} will derive the limit of $\sum_{i=1}^9I_{\ref{eq:ItoKw}}(i)_t-\sum_{i=1}^5 I_{\ref{eq:ItoKi}}(i)_t$ as we consider  \eqref{Ito:scheme} for the proof of \eqref{RN}. This part of the proof is where we begin to use particular initial conditions $z_0$ for $\P^{\bi}_{z_0}$, as imposed in Proposition~\ref{prop:logsum}. 
 
\subsection{Limiting formulas: local-time and Riemann-integral terms}\label{sec:logsumA}
The following proposition shows where the term $A^{\bbeta,\bs w,\bi}_0(t)$ in \eqref{RN} comes from.

\begin{prop}\label{prop:lim1}
Let $\bs w\in \R_+^{\mc E_N}$ with $w_\bi> 0$. Under $\P^{\bi}_{z_0}$ for any $z_0\in\CNwni$, the following limit holds with probability one: for all $t>0$, 
\begin{align}
\begin{split}
&I_{\ref{eq:ItoKw}}(1)_t+I_{\ref{eq:ItoKw}}(3)_t-I_{\ref{eq:ItoKi}}(1)_t-I_{\ref{eq:ItoKi}}(3)_t\\
&\quad \xrightarrow[\bs\vep\searrow \bs 0]{} 
\sum_{\bj\in \mc E_N\setminus\{\bi\}}2\left(\frac{w_\bj}{w_\bi} \right)\int_0^t  K_0(\sqrt{2\beta_\bj} |Z_s^\bj|)\d L^\bi_s+\sum_{\bj\in \mathcal E_N} \int_0^t \beta_\bj  \frac{w_\bj K^{\bbeta,\bj}_0(s)}{K^{\bbeta,\bs w}_0(s)}\d s- \beta_\bi t.\label{Ilim:1wi}
\end{split}
\end{align}
That is, with probability one, $I_{\ref{eq:ItoKw}}(1)_t+I_{\ref{eq:ItoKw}}(3)_t-I_{\ref{eq:ItoKi}}(1)_t-I_{\ref{eq:ItoKi}}(3)_t$ converges to $A^{\bbeta,\bs w,\bi}_0(t)$ for all $t>0$ as $\bs \vep\searrow \bs 0$, where $A^{\bbeta,\bs w,\bi}_0(t)$ is defined in \eqref{def:A}. 
\end{prop}

We need a few tools to prove Proposition~\ref{prop:lim1}. The first tool is
a special approximation of $\{L_t^\bi\}$ under $\P^{\bi}_{z_0}$, where $\{L_t^\bi\}$  is the Markovian local time  of $\{|Z^\bi_t|\}$ at level $0$ subject to the normalization used by Donati-Martin--Yor~\cite[(2.10)]{DY:Krein-2}. To specify the approximation, we will use the occupation times formula for $\{|Z^\bi_t|\}$ under $\P_{z_0}^\bi$ for any $z_0\in \CN$:
\begin{align}\label{oct}
\int_0^tg(|Z^\bi_t|)\d s=\int_0^\infty g(r)L^{\bi}_t(r)m^{\beta_\bi\da }_0(\d r),\quad \forall\;g\in \B_+(\R_+).
\end{align}
Here, $\{L^{\bi}_t(r)\}$ are the Markovian local times of $\{|Z^\bi_t|\}$ at level $r\geq 0$ such that $(r,t)\mapsto L^{\bi}_t(r)$, $(r,t)\in \R_+^2$, is continuous and $L^{\bi}_t(0)=L^\bi_t$. Also, $m^{\beta\da }_0(\d r)$ is the speed measure of $\BES (0,\beta\da)$ \cite[Section~(3.3)]{DY:Krein-2} with the following normalization:
\begin{align}\label{BES0b:speed}
m_0^{\beta\da }(\d r)=4r K_0(\sqrt{2\beta }r)^2 \d r,\quad 0<r<\infty.
\end{align}
The following lemma proves \eqref{lt:consistent} to explain
the relation between the normalization in \eqref{oct} and the normalization in $\{L^\bi_t\}$ in \eqref{def:DYLT-2}. This relation is obtained independently by using \cite[Theorem~2.1]{C:SDBG1-2} and comparing \eqref{oct} with the following identity: 
\[
\E^\bi_0\left[\int_0^t\frac{\1_{[0,r]}(|Z^\bi_s|)\d s}{4|Z^\bi_s|K_0(\sqrt{2\beta_\bi}|Z^\bi_s|)^2}\right]
=\int_0^t \int_0^r \frac{f_{|Z^\bi_s|}(r')}{4r'K_0(\sqrt{2\beta_\bi}r')^2}\d r'\d s,\quad r>0,
\]
where $f_{|Z^\bi_t|}(r)$ is a version of the PDF of $|Z^\bi_t|$ under $\P^{\bi}_0$. 

\begin{lem}
The PDF of $|Z_t^\bi|$, $t>0$, under $\P^\bi_0$ can be chosen as
\begin{align}
f_{|Z_t^\bi|}(r)\,\defeq\,4rK_0(\sqrt{2\beta}r) \int_0^t \d \tau\left(\e^{-\beta (t-\tau)}\int_0^\infty \d u \frac{\beta^u(t-\tau)^{u-1}}{\Gamma(u)}\right)\left(\frac{1}{2\tau}\e^{-\beta\tau-\frac{r}{2\tau}}\right).\label{Z:den}
\end{align}
In this case, we have
\begin{align}\label{lt:consistent}
\lim_{r\to 0}\frac{f_{|Z_t^\bi|}(r)}{4rK_0(\sqrt{2\beta}r)^2}=\lim_{r\to 0}\frac{f_{|Z_t^\bi|}(r)}{4r(\log r)^2}=\frac{\d}{\d t}\E_0^{\bi}[L^\bi_t],\quad t>0.  
\end{align}
\end{lem}
\begin{proof}
First, the formula in \eqref{Z:den} and its joint continuity in $(r,t)\in (0,\infty)^2$ can be obtained immediately by using \cite[Theorem~2.1 (1$\cc$)]{C:SDBG1-2} with $z^0=0$ and the polar coordinates. To prove \eqref{lt:consistent}, note that 
on the one hand, \cite[Theorem~2.1 (3$\cc$)]{C:SDBG1-2} with $z^0=0$ gives
\begin{align}
\frac{\d}{\d t}\E_0^{\bi}[L^\bi_t]&=g(t)\,\defeq\, \e^{-\beta t}\int_0^\infty \d u \frac{\beta^u t^{u-1}}{\Gamma(u)},\quad \forall\;t\in (0,\infty).\label{L:den}
\end{align}
On the other hand, to obtain the limit of the leftmost side of \eqref{lt:consistent}, note that by \cite[Proposition~4.5 (2$\cc$)]{C:SDBG1-2}, $g\in L^1_{\loc}([0,\infty))$, and $g\in \C((0,\infty))$ so that given $\vep>0$ and $t>0$, we can find $0<\delta<t$ such that 
$|g(t)-g(t-\tau)|<\vep$ for all $0\leq \tau\leq\delta$. 
By \eqref{Z:den}, we get, as $r\to 0$, 
\begin{align*}
\frac{f_{|Z_t^\bi|}(r)}{4rK_0(\sqrt{2\beta}r)}=\int_0^t \d \tau g(t-\tau)\left(\frac{1}{2\tau}\e^{-\beta \tau-\frac{r}{2\tau}}\right)
&=g(t)\int_0^\delta \d \tau\left(\frac{1}{2\tau}\e^{-\beta \tau-\frac{r}{2\tau}}\right)+\mathcal O(1)\\
&=g(t)\int_0^\infty \d \tau\left(\frac{1}{2\tau}\e^{-\beta \tau-\frac{r}{2\tau}}\right)+\mathcal O(1)\\
&=g(t)K_0(\sqrt{2\beta}r)+\mathcal O(1),
\end{align*}
where the last equality uses the integral representation of $K_0(\cdot)$~\cite[(5.10.25) on p.119]{Lebedev-2}. The last equality implies that the leftmost limit in \eqref{lt:consistent} equals $g(t)$, and hence, the rightmost side of \eqref{lt:consistent} by \eqref{L:den}. Finally, the first equality in \eqref{lt:consistent} holds by using \eqref{K00} again. 
\end{proof}

The following lemma gives the special approrximation of $\{L^\bi_t\}$ under $\P^{\bi}_{z_0}$ mentioned above. 
See also \cite[Section~(5.4)]{DY:Krein-2} on approximations of $\{L^\bi_t\}$.

\begin{lem}\label{lem:LTnorm1}
For all $\vep>0$, define
\begin{align}\label{def:kbe}
\kappa^{\beta\da}_\vep (r)\,\defeq\, \frac{\vep}{(\vep+r^2)^2K_0(\sqrt{2\beta(\vep+r^2)})^2},\quad r>0.
\end{align}
Then for any bounded $f:\R_+\to \R_+$ which is continuous at $r=0$ and has a compact support, 
\begin{align}\label{kbvep:asymp}
\lim_{\vep\to 0}\int_0^\infty f(r)\kappa^{\beta\da}_\vep(r)m^{\beta\da}_0(\d r)= 2f(0),
\end{align}
and under $\P^\bi_{z_0}$ for any $z_0\in \CN$, 
\begin{align}\label{eq:LTnorm1-2}
\sup_{s\leq t}\left|\int_0^s \kappa^{\beta\da}_\vep(|Z^\bi_r|)\d r-2L^{\bi}_s\right|\xrightarrow[\vep\to 0]{\rm a.s.}0,\quad \forall\;0<t<\infty.
\end{align}
\end{lem}
\begin{proof}
We first prove \eqref{kbvep:asymp}.
For any $0<\vep<1$, the first equality below follows from the definitions \eqref{BES0b:speed} and \eqref{def:kbe} of $m^{\beta\da}_{0}$ and $\kappa^{\beta\da}_\vep$, and the second equality is obtained by using the change of variables $\sqrt{\vep}r'=r$:
\begin{align}
\int_0^\infty f(r)\kappa^{\beta\da}_\vep(r)m^{\beta\da}_0(\d r)
&=\int_0^\infty f(r)\frac{4r\vep}{(\vep+r^2)^2}\cdot \frac{K_0(\sqrt{2\beta }r)^2}{K_0(\sqrt{2\beta(\vep+r^2)})^2} \d r\notag\\
&=\int_{0}^{\infty} f(\sqrt{\vep}r')\frac{4\sqrt{\vep}r'\vep}{(\vep+\vep (r')^2)^2}\cdot \frac{K_0(\sqrt{2\beta \vep}r')^2}{K_0(\sqrt{2\beta(\vep+\vep (r')^2)})^2}\sqrt{\vep} \d r'\notag\\
&=\int_{0}^{\infty}  f(\sqrt{\vep}r')\frac{4r'}{(1+ (r')^2)^2}\cdot \frac{K_0(\sqrt{2\beta \vep}r')^2}{K_0(\sqrt{2\beta(\vep+\vep (r')^2)})^2} \d r'.\label{eq:km}
\end{align}

For the right-hand side of \eqref{eq:km}, the dominated convergence theorem applies to pass $\vep\to 0$ under the integral. To see this, 
 we use the asymptotic representations \eqref{K00} and \eqref{K0infty} of $K_0(x)$ as $x\to 0$ and as $x\to\infty$ to get the following bounds for all $0<\vep<\frac{1}{4\cdot 2\beta}$:
\begin{align*}
\forall\;r':\sqrt{\vep}r'\leq \frac{1}{\sqrt{4\cdot 2\beta}},\; 
\frac{K_0(\sqrt{2\beta \vep}r')^2}{K_0(\sqrt{2\beta(\vep+\vep (r')^2)})^2}&\less\frac{\log^2(2\beta\vep (r')^2)}{\log^2(2\beta(\vep+\vep (r')^2))}\leq \frac{\log^2(2\beta\vep (r')^2)}{\log^2(2\beta(\vep+\frac{1}{4\cdot 2\beta}))}\\
&\less \left(\frac{\log (2\beta\vep)+\log (r')^2}{\log (2\beta(\vep+\frac{1}{4\cdot 2\beta}))}\right)^2
\leq C(\beta)[1+\log(r')^2]^2,
\end{align*}
where the second inequality is obtained by the decreasing monotonicity of $x\mapsto \log^2x$ over $0<x\leq 1$, and for all $M>0$,
\begin{align*}
\forall\; r':\sqrt{\vep}r'>\frac{1}{\sqrt{4\cdot 2\beta}},\; &\quad\;\frac{K_0(\sqrt{2\beta \vep}r')^2}{K_0(\sqrt{2\beta(\vep+\vep (r')^2)})^2}\1_{\{\sqrt{\vep}r'\leq M\}}\\
&\leq C(\beta) \frac{\sqrt{2\beta(\vep+\vep (r')^2)}}{\sqrt{2\beta\vep}r'}\e^{-2\sqrt{2\beta \vep(r')^2}+2\sqrt{2\beta(\vep+\vep (r')^2)} }\1_{\{\sqrt{\vep}r'\leq M\}}
\leq C(M,\beta).
\end{align*}
By the last two displays, we obtain 
\begin{align}\label{eq:km1}
\frac{K_0(\sqrt{2\beta \vep}r')^2\1_{\{\sqrt{\vep }r'\leq M\}}}{K_0(\sqrt{2\beta(\vep+\vep (r')^2)})^2}\leq C(M,\beta)[1+\log(r')^2]^2,\;\forall\;r'>0,\;M>0,\;0<\vep<\frac{1}{4(2\beta)},
\end{align}
so the dominated convergence theorem applies to the right-hand side of \eqref{eq:km}.

Thanks to the applicability of the dominated convergence theorem, we obtain from \eqref{eq:km}, the assumed continuity of $f$ at $0$, and the asymptotic representation \eqref{K00} of $K_0(x)$ as $x\to 0$ the following limit:
\begin{align}\label{oct:limit}
 \lim_{\vep\to 0}\int_0^\infty f(r)\kappa^{\beta\da}_\vep(r)m^{\beta\da}_0(\d r)
=\int_{0}^{\infty}  f(0)\frac{4r'}{(1+ (r')^2)^2}\d r'=2f(0)
\end{align}
by the change of variables $r=1+(r')^2$. The last equality proves  \eqref{kbvep:asymp}.

Finally, to obtain \eqref{eq:LTnorm1-2}, we use the occupation times formula \eqref{oct} and get
\begin{align*}
\int_0^s \kappa^{\beta\da}_\vep(|Z^\bi_r|)\d r=\int_0^\infty L^{\bi}_s(r)\kappa^{\beta\da}_\vep(r)m^{\beta\da }_0(\d r).
\end{align*}
Since $(r,s)\mapsto L^{\bi}_s(r)$ is continuous and $r\mapsto L^{\bi}_t(r)$ has a compact support a.s., \eqref{eq:LTnorm1-2} follows by using \eqref{eq:km}, \eqref{eq:km1}, and a slight modification of \eqref{oct:limit}. The proof is complete.
\end{proof}

Lemmas~\ref{lem:unif1}--\ref{lem:bullet} below give the other tools for proving Proposition~\ref{prop:lim1} and will also be used to derive the other limits for Proposition~\ref{prop:logsum}. First, for Lemma~\ref{lem:unif1}, we omit the elementary proof of (1$\cc$) and refer to \cite[Section~2.8, pp.121+]{Ash-2} for (2$\cc$).
 
\begin{lem}\label{lem:unif1}
{\rm (1$\cc$)} Let $f:\R_+\to \R_+$ be continuous and $F_n\to F$ uniformly on compacts as $n\to\infty$. Then $F_n(f(s))\to F(f(s))$ uniformly on compacts as $n\to\infty$.\medskip 

\noindent {\rm (2$\cc$)} If $F_n$ and $F$ are increasing right-continuous functions on $\R_+$ such that $F_n$ converges weakly to $F$ and if $\{g_n\}$ converges uniformly to a continuous function $g$ on compacts, then $\int_0^t g_n(s)\d F_n(s)\to \int_0^t g(s)\d F(s)$ for all $t\geq 0$. 
\end{lem}

The second lemma determines the three basic convergences which arise from the integrands of the terms $I_{\ref{eq:ItoKi}}(i)_t$ and $I_{\ref{eq:ItoKw}}(i)_t$. 

\begin{lem}\label{lem:unif2}
For all $0<\delta_0<\delta_1<\infty$, set
\[
\mathcal R(\delta_0,\delta_1)\,\defeq\, \{\bs R\in \R_+^{\mathcal E_N}:R_\bj\in [\delta_0,\delta_1],\;\forall\;\bj\neq \bi,\;\&\;\; R_\bi\in [0,\delta_1]\}.
\]
Then as $\bs \vep\searrow \bs 0$, the following convergences hold uniformly in $ \bs R\in \mathcal R(\delta_0,\delta_1)$:
\begin{align}\label{Gvep:convwi}
\begin{split}
&\frac{1}{G_0(2\beta_\bi (\vep_\bi+ R_\bi))}\searrow \frac{1}{G_0(2\beta_\bi R_\bi)},\quad  G_0(2\beta_\bj(\vep_\bj+R_\bj))\nearrow G_0(2\beta_\bj R_\bj), \\
&G_1(2\beta_\bi(\vep_\bi+R_\bi))\nearrow G_1(2\beta_\bi R_\bi).
\end{split}
\end{align}
\end{lem}
\begin{proof}
First, the monotonic pointwise convergences hold by the definition \eqref{def:Gnu} of $G_\nu(\cdot)$ and the decreasing monotonicity of $K_0(\cdot)$ and $\widehat{K}_1(\cdot)$ due to \eqref{hK:der}. Moreover, the limiting functions in \eqref{Gvep:convwi} are continuous. Hence, by Dini's theorem, the convergences in  \eqref{Gvep:convwi} are also uniform in $\bs R\in \mc R(\delta_0,\delta_1)$.
\end{proof}

The last lemma is stated for convenience. It collects three simple properties that will be used altogether frequently in the remaining of Section~\ref{sec:logsum}.

\begin{lem}\label{lem:bullet}
{\rm (1$\cc$)} For $z_0\in \CNwni$, $\P^{\bi}_{z_0}(R^\bj_s\neq 0\;\forall\;s\geq 0,\bj\neq \bi)=1$.\medskip

\noindent {\rm (2$\cc$)} $G_1(R)$ is continuous up to and including  $R=0$. \medskip 

\noindent {\rm (3$\cc$)} $K_0(\cdot)$ is bounded away from zero on compacts of $\R_+$. 
\end{lem}
\begin{proof}
(1$\cc$) follows from Proposition~\ref{prop:BMntc} and the assumption that $z_0\in \CNwni$; (2$\cc$) uses \eqref{def:hK0}; (3$\cc$) can be seen by using the integral representation  of $K_0(\cdot)$~\cite[(5.10.25) on p.119]{Lebedev-2}. 
\end{proof}

\begin{proof}[Proof of Proposition~\ref{prop:lim1}]
Write
\begin{align}
&\quad \;I_{\ref{eq:ItoKw}}(1)+I_{\ref{eq:ItoKw}}(3)-I_{\ref{eq:ItoKi}}(1)-I_{\ref{eq:ItoKi}}(3)\notag\\
&
=\sum_{\bj\in \mathcal E_N} \int_0^t \left( \frac{w_\bj}{H_0^{\bs w}(\bs\vep +\bs R_s)}-\frac{w_\bj\delta_{\bi,\bj}}{H_0^{\bi}(\vep_\bi+R_s^\bi)}\right)\left[\frac{-2\beta_\bj G_{1}(R)}{R}\right]_{R=2\beta_\bj(\vep_\bj+R^\bj_s)}\d s\notag\\
&\quad +\sum_{\bj\in \mc E_N} \int_0^t \left( \frac{w_\bj}{H_0^{\bs w}(\bs\vep +\bs R_s)}-\frac{w_\bj\delta_{\bi,\bj}}{H_0^{\bi}(\vep_\bi+R_s^\bi)}\right)\left[\frac{\beta_\bj(R-2\beta_\bj\vep_\bj) [2G_1(R)+RG_0(R)]}{R^{2}}\right]_{R=2\beta_\bj(\vep_\bj+R^\bj_s)}\d s\notag\\
&=
\widetilde{I}^{1,\bs\vep}_t+\widetilde{I}^{2,\bs\vep}_t
,\label{I1I3}
\end{align}
where the two terms on the right-hand side are defined by
\begin{align}
\widetilde{I}^{1,\bs\vep}_t&\,\defeq \,\sum_{\bj\in \mathcal E_N} \int_0^t \left( \frac{w_\bj}{H_0^{\bs w}(\bs\vep +\bs R_s)}-\frac{w_\bj\delta_{\bi,\bj}}{H_0^{\bi}(\vep_\bi+R_s^\bi)}\right) \notag\\
&\quad \times\left[\frac{-2\beta_\bj G_{1}(R)}{R}+\frac{2\beta_\bj G_{1}(R)(R-2\beta_\bj \vep_\bj)}{R^{2}}\right]_{R=2\beta_\bj(\vep_\bj+R^\bj_s)}\d s\notag\\
&\,=\,\sum_{\bj\in \mathcal E_N} \int_0^t \left( \frac{w_\bj}{H_0^{\bs w}(\bs\vep +\bs R_s)}-\frac{w_\bj\delta_{\bi,\bj}}{H_0^{\bi}(\vep_\bi+R_s^\bi)}\right)\left[\frac{-(2\beta_\bj)^2 G_{1}(R)\vep_\bj}{R^{2}}\right]_{R=2\beta_\bj(\vep_\bj+R^\bj_s)}\d s,\notag\\
\begin{split}
\widetilde{I}^{2,\bs \vep}_t&\,\defeq \,\sum_{\bj\in \mc E_N} \int_0^t \left( \frac{w_\bj}{H_0^{\bs w}(\bs\vep +\bs R_s)}-\frac{w_\bj\delta_{\bi,\bj}}{H_0^{\bi}(\vep_\bi+R_s^\bi)}\right)\left[\frac{\beta_\bj(R-2\beta_\bj\vep_\bj) G_{0}(R)}{R}\right]_{R=2\beta_\bj(\vep_\bj+R^\bj_s)}\d s.\label{I1I3-I2}
\end{split}
\end{align}
We handle the limits of $\widetilde{I}^{1,\bs\vep}_t$ and $\widetilde{I}^{2,\bs\vep}_t$ in Steps~1 and~2, respectively, an conclude in Step~3. \medskip 

\noindent {\bf Step 1.} To find the limit of $\widetilde{I}^{1,\bs\vep}_t$, we separate the summand indexed by $\bi$ from the others. Therefore, by the definition \eqref{def:Gnu} of $G_0(R)$, we get
\begin{align}
\widetilde{I}^{1,\bs\vep}_t
&=\sum_{\bj\in \mc E_N\setminus\{ \bi\}}\int_0^t \left( \frac{w_\bj}{H_0^{\bs w}(\bs\vep +\bs R_s)}\right)\times [- G_{1}(2\beta_\bj(\vep_\bj+R^\bj_s))] \times \frac{ \vep_\bj}{(\vep_\bj +R^\bj_s)^{2}}\d s\notag\\
&\quad +\int_0^t \left( \frac{w_\bi}{H_0^{\bs w}(\bs\vep +\bs R_s)}-\frac{w_\bi}{H_0^{\bi}(\vep_\bi+R_s^\bi)}\right)\times [- G_{1}(2\beta_\bi(\vep_\bi+R^\bi_s))]\times G_0(2\beta_\bi (\vep_\bi+R^\bi_s))^2\notag\\
&\quad \times \frac{ \vep_\bi}{(\vep_\bi +R^\bi_s)^{2}K_0(\sqrt{2\beta_\bi (\vep_\bi+R^\bi_s)})^2}\d s\notag\\
\begin{split}
&=\sum_{\bj\in \mc E_N\setminus\{\bi\}}\int_0^t \left( \frac{w_\bj}{H_0^{\bs w}(\bs\vep +\bs R_s)}\right)\times [- G_{1}(2\beta_\bj(\vep_\bj+R^\bj_s))] \times \frac{ \vep_\bj}{(\vep_\bj +R^\bj_s)^{2}}\d s\\
&\quad +\int_0^t \left( \frac{w_\bi}{H_0^{\bs w}(\bs\vep +\bs R_s)}-\frac{w_\bi}{H_0^{\bi}(\vep_\bi+R_s^\bi)}\right)\\
&\quad \times [- G_{1}(2\beta_\bi(\vep_\bi+R^\bi_s))]\times G_0(2\beta_\bi (\vep_\bi+R^\bi_s))^2\times \kappa^{\beta_\bi\da}_\vep(|Z^\bi_s|)\d s,\label{I1:limwi0}
\end{split}
\end{align}
where the last equality uses the definition \eqref{def:kbe} of $\kappa^{\beta\da}_{\vep}$ and the definition of $R^\bi_s$. 

Now, the following limits hold with probability one:
\begin{align}
\widetilde{I}^{1,\bs\vep}_t
\xrightarrow[\bs \vep\searrow \bs 0]{}&\;\int_0^t 
\frac{\sum_{\bj\in \mc E_N\setminus\{\bi\}}w_\bj G_0(2\beta_\bj R_s^\bj)
G_{1}(2\beta_\bi R_s^\bi)
}{w_\bi+\sum_{\bj\in\mc E_N\setminus\{ \bi\}}\frac{w_\bj G_0(2\beta_\bj R^\bj_s)}{G_0(2\beta_\bi R_s^\bi )}}2\d L^\bi_s\label{lim:I1wi0}\\
&=\sum_{\bj\in \mc E_N\setminus\{\bi\}}2\left(\frac{w_\bj}{w_\bi}\right) \int_0^t 
 K_0(\sqrt{2\beta_\bj} |Z_s^\bj|)\d L^\bi_s,\quad \;\forall\;t>0.\label{lim:I1wi}
\end{align}
Let us explain \eqref{lim:I1wi0} and \eqref{lim:I1wi}. To obtain \eqref{lim:I1wi0}, we consider the sum and the last integral in \eqref{I1:limwi0} separately. To handle that sum, we use Lemma~\ref{lem:bullet} and  recall that $H_0^{\bs w}$ is defined in \eqref{FUNC2}. 
Hence, with probability one, the sum over $\bj\in \mc E_N\setminus\{\bi\}$ in  \eqref{I1:limwi0} converges to zero by dominated convergence for all $t$. 
Also, the limit of the last integral in \eqref{I1:limwi0} is equal to the limit in \eqref{lim:I1wi0}. This limit can be seen by 
 using Lemma~\ref{lem:unif1} (1$\cc$) and Lemma~\ref{lem:unif1} (2$\cc$) with \eqref{eq:LTnorm1-2} and the following limit proven at the end of Step~1: for all $0<\delta_0<\delta_1<\infty$,
\begin{align}
&\quad\;\lim_{\bs\vep\searrow \bs 0}\left( \frac{w_\bi}{H_0^{\bs w}(\bs\vep +\bs R)}-\frac{w_\bi}{H_0^{\bi}(\vep_\bi+R_\bi)}\right)\times [- G_{1}(2\beta_\bi(\vep_\bi+R_\bi))]\times  G_0(2\beta_\bi (\vep_\bi+R^\bi))^2\notag\\
& =
\frac{\sum_{\bj\in \mc E_N\setminus\{\bi\}}w_\bj G_0(2\beta_\bj R_\bj)
G_{1}(2\beta_\bi R_\bi)
}{w_\bi+\sum_{\bj\in \mc E_N\setminus\{\bi\}}\frac{w_\bj G_0(2\beta_\bj R_\bj)}{G_0(2\beta_\bi R_\bi )}}\quad \mbox{uniform in $\bs R\in \mathcal R(\delta_0,\delta_1)$.}\label{lim:Gwi}
\end{align}
Note that we can choose $\delta_0,\delta_1$ to justify the use of Lemma~\ref{lem:unif1} (2$\cc$) due to Lemma~\ref{lem:bullet} (1$\cc$).
We have proved \eqref{lim:I1wi0}. Also, \eqref{lim:I1wi} holds by the following three facts: (i) $G_1(0)=\widehat{K}_1(0)=1$ by the definition \eqref{def:Gnu} of $G_\nu(\cdot)$ and \eqref{def:hK0}, (ii)
\begin{align}\label{def:contR}
\frac{w_\bj G_0(2\beta_\bj R_\bj)}{G_0(2\beta_\bi R_\bi)}=
\frac{w_\bj K_0(\sqrt{2\beta_\bj R_\bj})}{K_0(\sqrt{2\beta_\bi R_\bi })}=0,\quad \forall\;\bj\neq \bi,\mbox{ if $R_\bi=0$ and $R_\bj>0$,}
\end{align}
and (iii) Lemma~\ref{lem:bullet} (1$\cc$) holds now since $z_0\in \CNwni$.

To complete Step~1, it remains to show the uniform convergence in \eqref{lim:Gwi}. By the definitions of $H_0^{\bi}$ and $H_0^{\bs w}$ in  \eqref{FUNC1} and \eqref{FUNC2},
\begin{align}
&\quad\;\left( \frac{w_\bi}{H_0^{\bs w}(\bs\vep +\bs R)}-\frac{w_\bi}{H_0^{\bi}(\vep_\bi+R_\bi)}\right)\times [- G_{1}(2\beta_\bi(\vep_\bi+R_\bi))\times  G_0(2\beta_\bi (\vep_\bi+R^\bi))^2]\notag\\
&=\left[ \frac{-w_\bi\sum_{\bj\in \mc E_N\setminus\{ \bi\}}w_\bj G_0(2\beta_\bj(\vep_\bj+R_\bj))}{H_0^{\bs w}(\bs\vep +\bs R)H_0^{\bi}(\vep_\bi+R_\bi)}\times G_0(2\beta_\bi (\vep_\bi+R_\bi))^2\right]
\times [- G_{1}(2\beta_\bi(\vep_\bi+R_\bi))]\notag\\
&=\left[ 
\frac{-\sum_{\bj\in \mc E_N\setminus\{ \bi\}}w_\bj G_0(2\beta_\bj(\vep_\bj+R_\bj))}{
w_\bi+\sum_{\bj\in \mc E_N\setminus\{ \bi\}}\frac{w_\bj G_0(2\beta_\bj(\vep_\bj+R_\bj))}{G_0(2\beta_\bi(\vep_\bi+R_\bi))}}
\right]
\times [- G_{1}(2\beta_\bi(\vep_\bi+R_\bi))].\notag
\end{align}
Since $w_\bi>0$ by assumption, 
the last equality and Lemma~\ref{lem:unif2} are enough to get \eqref{lim:Gwi}.\medskip

\noindent {\bf Step 2.}
This step is to obtain the limit of $\widetilde{I}^{2,\bs \vep}_t$, which is defined in \eqref{I1I3-I2}. We have
\begin{align}
\widetilde{I}^{2,\bs \vep}_t
&=\sum_{\bj\in \mathcal E_N} \int_0^t \frac{w_\bj}{H_0^{\bs w}(\bs\vep+\bs R_s)}\beta_\bj \times G_0 (R)
\times \left(\frac{R-2\beta_\bj \vep_\bj}{R}\right)\Bigg|_{R=2\beta_\bj (\vep_\bj+R^\bj_s)}\d s\notag\\
&\quad -\int_0^t \frac{w_\bi}{H_0^{\bi}(\vep_\bi+R^\bi_s)}\beta_\bi \times G_0 (R)
\times \left(\frac{R-2\beta_\bi \vep_\bi}{R}\right)\Bigg|_{R=2\beta_\bi (\vep_\bi+R^\bi_s)}\d s\notag\\
&\xrightarrow[\bs \vep\searrow \bs 0]{} \int_0^t \sum_{\bj\in \mathcal E_N}\beta_\bj \frac{w_\bj K^{\bbeta,\bj}_0(s)}{K^{\bbeta,\bs w}_0(s)}\1_{\{R^\bj_s>0\}}\d s- \beta_\bi \int_0^t\1_{\{R^\bi_s>0\}}\d s\notag\\
&=\int_0^t \sum_{\bj\in \mathcal E_N}\beta_\bj \frac{w_\bj K^{\bbeta,\bj}_0(s)}{K^{\bbeta,\bs w}_0(s)}\d s- \beta_\bi t,\label{lim:I2wi}
\end{align}
where the convergence holds by dominated convergence and the definitions \eqref{def:Gnu}, \eqref{FUNC1}, and \eqref{FUNC2} of $G_0$, $H_0^\bi$, and $H_0^{\bs w}$, and the last equality uses Lemma~\ref{lem:bullet} (1$\cc$) and the fact that $\{R^\bi_t\}$ is instantaneously reflecting at $0$~\cite[Theorem~2.1]{DY:Krein-2}.  \medskip

\noindent {\bf Step~3.} By \eqref{I1I3}, \eqref{lim:I1wi} and \eqref{lim:I2wi}, we obtain the required equality in \eqref{Ilim:1wi}.
\end{proof}

\subsection{Limiting formulas: martingale terms}
Recall that $\Twi_\eta$, $\{N^{\bbeta,\bs w,\bi}_0(t)\}$ and $\{\mathring{N}^{\bbeta,\bs w,\bi}_0(t)\}$
 are defined in \eqref{def:Twi}, \eqref{def:Nwi} and \eqref{def:Nwiring}.
The following proposition borrows the following mode of convergence from \cite[p.57]{Protter-2}. A sequence of processes $\{H^n_t\}_{n\geq 1}$ converges to a process $\{H_t\}$ {\bf uniformly on compacts in probability} (abbreviated {\bf u.c.p.}) if, for each $t> 0$, $\sup_{0\leq s\leq t}|H^n_s-H_s|$ converges to zero in probability.  

\begin{prop}\label{prop:lim2}
Let $\bs w\in \R_+^{\mc E_N}$ with $w_\bi> 0$. The following properties hold under $\P^{\bi}_{z_0}$ for any $z_0\in \CNwni$.  \medskip 

\noindent {\rm (1$\cc$)} $\{N^{\bbeta,\bs w,\bi}_0(t)\}$ is a well-defined continuous local martingale, and for $\eta\in (0,\min_{\bj:w_\bj>0,\bj\neq \bi}|z^\bj_0|)$,
$\{\mathring{N}^{\bbeta,\bs w,\bi}_0(t\wedge T^{\bs w\setminus\{w_\bi\}}_{\eta})\}$
is a continuous $L^2$-martingale satisfying
\begin{align}\label{expbdd:Nring}
\E^{\bi}_{z_0}\left[\exp\left\{\lambda \la \mathring{N}^{\bbeta,\bs w,\bi}_0,\mathring{N}^{\bbeta,\bs w,\bi}_0\ra_{t\wedge T^{\bs w\setminus\{w_\bi\}}_{\eta}}\right\}\right]<\infty,\quad \forall\;\lambda\in \R,\;t\geq 0.
\end{align}

\noindent {\rm (2$\cc$)} The following limit holds in the sense of u.c.p. as $\bs \vep\searrow \bs 0$:
\begin{align}
&\{I_{\ref{eq:ItoKw}}(2)_t-I_{\ref{eq:ItoKi}}(2)_t\} \xrightarrow[\bs\vep\searrow \bs 0]{\rm u.c.p.} \{ N^{\bbeta,\bs w,\bi}_0(t)\}.\label{Ilim:2wi}
\end{align}
\end{prop}

\begin{proof}
(1$\cc$) By Lemma~\ref{lem:bullet} (2$\cc$), $N^{\bbeta,\bs w,\bi}_0-\mathring{N}^{\bbeta,\bs w,\bi}_0$ is a well-defined continuous local martingale. Since $T_{\eta}^{\bs w\setminus\{w_\bi\}}\nearrow\infty$ as $\eta\searrow 0$ under $\P^{\bi}_{z_0}$ by Proposition~\ref{prop:BMntc}, it remains to prove the required properties of $\mathring{N}^{\bbeta,\bs w,\bi}_0$, that is, the $L^2$-martingale property and \eqref{expbdd:Nring}. 
To this end, rewrite the stochastic integrand of $\mathring{N}^{\bbeta,\bs w,\bi}_0$ from \eqref{def:Nwiring} as follows:
\[
\frac{\hK^{\bbeta,\bi}_1(s)[K^{\bbeta,\bs w}_0(s)-w_\bi K_0^{\bbeta,\bi}(s)]}{|Z^\bi_s|K^{\bbeta,\bi}_0(s)K^{\bbeta,\bs w}_0(s)}
=\frac{\hK^{\bbeta,\bi}_1(s)\sum_{\bj\in \mc E_N\setminus\{\bi\}}w_\bj K^{\bbeta,\bj}_0(s)}{|Z^\bi_s|K^{\bbeta,\bi}_0(s)K^{\bbeta,\bs w}_0(s)}.
\] 
We bound the right-hand side in the following manner. Whenever $|Z^\bi_s|\geq 1$, 
\begin{align*}
\frac{\hK^{\bbeta,\bi}_1(s)\sum_{\bj\in \mc E_N\setminus\{\bi\}}w_\bj K^{\bbeta,\bj}_0(s)}{|Z^\bi_s|K^{\bbeta,\bi}_0(s)K^{\bbeta,\bs w}_0(s)}=\frac{\hK^{\bbeta,\bi}_1(s)}{|Z^\bi_s|K^{\bbeta,\bi}_0(s)}\left(\frac{\sum_{\bj\in \mc E_N\setminus\{\bi\}}w_\bj K^{\bbeta,\bj}_0(s)}{K^{\bbeta,\bs w}_0(s)}\right)
\leq \frac{\hK^{\bbeta,\bi}_1(s)}{|Z^\bi_s|K^{\bbeta,\bi}_0(s)}\leq C(\beta_\bi)
\end{align*}
by the asymptotic representations \eqref{K0infty} and \eqref{K1infty} of $K_0(x)$ and $K_1(x)$ as $x\to\infty$. Also, whenever $|Z^\bi_s|\leq 1$ and $s\leq T^{\bs w\setminus\{w_\bi\}}_\eta$, 
\begin{align*}
\frac{\hK^{\bbeta,\bi}_1(s)\sum_{\bj\in \mc E_N\setminus\{\bi\}}w_\bj K^{\bbeta,\bj}_0(s)}{|Z^\bi_s|K^{\bbeta,\bi}_0(s)K^{\bbeta,\bs w}_0(s)}&= \frac{\hK^{\bbeta,\bi}_1(s)}{|Z^\bi_s|K^{\bbeta,\bi}_0(s)^2}\left(\frac{K_0^{\bbeta,\bi}(s)\sum_{\bj\in \mc E_N\setminus\{\bi\}}w_\bj K^{\bbeta,\bj}_0(s)}{K^{\bbeta,\bs w}_0(s) }\right)\\
&\leq  \frac{\hK^{\bbeta,\bi}_1(s)}{|Z^\bi_s|K^{\bbeta,\bi}_0(s)^2}\left(\frac{K_0^{\bbeta,\bi}(s)\sum_{\bj\in \mc E_N\setminus\{\bi\}}w_\bj K^{\bbeta,\bj}_0(s)}{w_\bi K_0^{\bbeta,\bi}(s) }\right)\\
&\leq \frac{C(\bbeta,\eta,w_\bi)}{|Z^\bi_s|K^{\bbeta,\bi}_0(s)^2},
\end{align*}
where the last inequality uses the asymptotic representations \eqref{K10} and \eqref{K1infty} of $K_1(x)$ as $x\to 0$ and as $x\to\infty$, the definition \eqref{def:Twi} of $T^{\bs w\setminus\{w_\bi\}}_\eta$, and
the asymptotic representation \eqref{K0infty} of $K_0(x)$ as $x\to\infty$. It follows from the last two displays that
\[
\la \mathring{N}^{\bbeta,\bs w,\bi}_0,\mathring{N}^{\bbeta,\bs w,\bi}_0\ra_{t\wedge T^{\bs w\setminus\{w_\bi\}}_\eta}\leq C(\bbeta,\eta,w_\bi)\left(t+\int_0^t \frac{\1_{\{|Z_s^\bi|\leq 1\}}\d s}{|Z^\bi_s|^2K^{\bbeta,\bi}_0(s)^4}\right).
\]
The exponential integrability in \eqref{expbdd:Nring} now follows from \cite[Proposition~4.2 (1$\cc$)]{C:SDBG1-2} and  implies the $L^2$-martingale property of $\{\mathring{N}^{\bbeta,\bs w,\bi}_0(t)\}$ \cite[(1.23) Proposition, p.129]{RY-2}. \medskip 

\noindent {\rm (2$\cc$)} We begin by rewriting $I_{\ref{eq:ItoKw}}(2)_t-I_{\ref{eq:ItoKi}}(2)_t$ as follows:
\begin{align}
&\quad \;I_{\ref{eq:ItoKw}}(2)_t-I_{\ref{eq:ItoKi}}(2)_t\\
&=
-\sum_{\bj\in \mathcal E_N\setminus\{\bi\}}w_\bj (2\beta_\bj)\int_0^t \frac{G_{1}(R)}{2R}\Bigg|_{R=2\beta_\bj(\vep_\bj+R^\bj_s)}\times \frac{2|Z^\bj_s|}{H_0^{\bs w}(\bs \vep+\bs R_s)}\d B^\bj_s\notag\\
&\quad -w_\bi (2\beta_\bi)
\int_0^t \frac{G_{1}(R)}{2R}\Bigg|_{R=2\beta_\bi(\vep_\bi+R^\bi_s)}\left(\frac{1}{H^{\bs w}_0(\bs \vep+\bs R_s)}-\frac{1}{H^{\bi}_0( \vep_\bi+R^\bi_s)}\right)2|Z^\bi_s|\d B^\bi_s\notag\\
\begin{split}\label{MN:proof}
&=-\sum_{\bj\in \mathcal E_N\setminus\{\bi\}}w_\bj \int_0^t \frac{G_{1}(2\beta_\bj(\vep_\bj+R^\bj_s))}{2(\vep_\bj+R^\bj_s)}\times \frac{2|Z^\bj_s|}{H_0^{\bs w}(\bs \vep+\bs R_s)}\d B^\bj_s\\
&\quad -w_\bi 
\int_0^t \frac{G_{1}(2\beta_\bi(\vep_\bi+R^\bi_s))}{2(\vep_\bi+R^\bi_s)}\left(\frac{1}{H^{\bs w}_0(\bs \vep+\bs R_s)}-\frac{1}{H^{\bi}_0( \vep_\bi+R^\bi_s)}\right)2|Z^\bi_s|\d B^\bi_s.
\end{split}
\end{align}
Below we derive the limits of the sum using $\bj\in \mc E_N\setminus\{\bi\}$ and the last integral in \eqref{MN:proof} separately using the following stopping times: for $\delta_0\in (0,\min_{\bj:w_\bj>0,\bj\neq \bi}|z^\bj_0|)$ and $\delta_1>\max_{\bj\in \mc E_N}|z_0^\bj|$,
\[
T_{\delta_0,\delta_1}\,\defeq\,\inf\left\{t\geq 0;\min_{\bj\in \mc E_N\setminus\{\bi\}}R^\bj_t\leq \delta_0\mbox{ or }\max_{\bj\in \mc E_N}R^\bj_t\geq \delta_1\right\}.
\]

First, when stopped at $T_{\delta_0,\delta_1}$, the limit of the sum using $\bj\in \mc E_N\setminus\{\bi\}$ in \eqref{MN:proof} can be determined as follows by Lemma~\ref{lem:unif1} (1$\cc$), Lemma~\ref{lem:unif2}, Lemma~\ref{lem:bullet} and the dominated convergence theorem for stochastic integrals \cite[Theorem~32, p.176]{Protter-2}:
\begin{align}
\begin{split}\label{MN:proof1}
&\Bigg\{-\sum_{\bj\in \mathcal E_N\setminus\{\bi\}}w_\bj \int_0^{t\wedge T_{\delta_0,\delta_1}} \frac{G_{1}(2\beta_\bj(\vep_\bj+R^\bj_s))}{2(\vep_\bj+R^\bj_s)}\times \frac{2|Z^\bj_s|}{H_0^{\bs w}(\bs \vep+\bs R_s)}\d B^\bj_s\Bigg\}\\
&\xrightarrow[\bs \vep\searrow \bs 0]{\rm u.c.p.}\Bigg\{-\sum_{\bj\in \mathcal E_N\setminus\{\bi\}}w_\bj \int_0^{t\wedge T_{\delta_0,\delta_1}} \frac{G_{1}(2\beta_\bj R^\bj_s)}{2 R^\bj_s}\times \frac{2|Z^\bj_s|}{H_0^{\bs w}(\bs R_s)}\d B^\bj_s\Bigg\}\\
&\quad\quad \quad =\{N^{\bbeta,\bs w,\bi}_0(t\wedge T_{\delta_0,\delta_1})-\mathring{N}^{\bbeta,\bs w,\bi}_0(t\wedge T_{\delta_0,\delta_1})\}.
\end{split}
\end{align}
Here, the equality uses the definitions \eqref{def:K1bj}, \eqref{def:Gnu} and \eqref{FUNC2} of $\widehat{K}^{\bbeta,\bj}_1(s)$, $G_1$ and $H_0^{\bs w}$. Recall once again that $\{N^{\bbeta,\bs w,\bi}_0(t)\}$ and $\{\mathring{N}^{\bbeta,\bs w,\bi}_0(t)\}$ are defined in \eqref{def:Nwi} and \eqref{def:Nwiring}.

For the last stochastic integral in \eqref{MN:proof}, we show that 
\begin{align}
&\left\{ -w_\bi 
\int_0^{t\wedge T_{\delta_0,\delta_1}} \frac{G_{1}(2\beta_\bi(\vep_\bi+R^\bi_s))}{2(\vep_\bi+R^\bi_s)}\left(\frac{1}{H^{\bs w}_0(\bs \vep+\bs R_s)}-\frac{1}{H^{\bi}_0( \vep_\bi+R^\bi_s)}\right)2|Z^\bi_s|\d B^\bi_s\right\}\notag\\
&\xrightarrow[\bs \vep\searrow \bs 0]{\rm u.c.p.}\{\mathring{N}^{\bbeta,\bs w,\bi}_0(t\wedge T_{\delta_0,\delta_1})\}.\label{MN:proof2}
\end{align}
To justify this limit,  first, note that by the definitions of $H_0^{\bi}$ and $H_0^{\bs w}$ in  \eqref{FUNC1} and \eqref{FUNC2}, 
\begin{align}
&\quad \;\frac{G_{1}(2\beta_\bi(\vep_\bi+R^\bi_s))}{2(\vep_\bi+R^\bi_s)}\left(\frac{1}{H^{\bs w}_0(\bs \vep+\bs R_s)}-\frac{1}{H^{\bi}_0( \vep_\bi+R^\bi_s)}\right)2|Z^\bi_s|\notag\\
&=\frac{G_{1}(2\beta_\bi(\vep_\bi+R^\bi_s))}{2(\vep_\bi+R^\bi_s)} \left(\frac{-\sum_{\bj\in \mc E_N\setminus\{\bi\}}w_\bj G_0(2\beta_\bj(\vep_\bj+R^\bj_s))}{H^{\bs w}_0(\bs \vep+\bs R_s)H^\bi_0(\vep_\bi+R^\bi_s)}\right)2|Z^\bi_s|\notag\\
\begin{split}
&=\frac{|Z^\bi_s| }{(\vep_\bi+R^\bi_s)G_0(2\beta_\bi(\vep_\bi+R^\bi_s))^2} \times
\frac{H^\bi_0(\vep_\bi+R^\bi_s)G_{1}(2\beta_\bi(\vep_\bi+R^\bi_s))[-\sum_{\bj\in \mc E_N\setminus\{\bi\}}w_\bj G_0(2\beta_\bj(\vep_\bj+R^\bj_s))]}{w^2_\bi H_0^{\bs w}(\bs \vep+\bs R_s)}\notag
\end{split}\\
\begin{split}
&\leq \frac{C(\delta_0,\delta_1,\bbeta,\bs w)}{|Z^\bi_s|G_0(2\beta_\bi R^\bi_s)^2},\quad \forall\;s\leq T_{\delta_0,\delta_1}.
\label{MN:proof21}
\end{split}
\end{align}
The last inequality follows by using the asymptotic representation \eqref{K10} of $K_1(x)$ as $x\to 0$ and the increasing monotonicity of $x\mapsto \sqrt{x}K_0(x)$ for all small $x>0$: $K_0'(x)=-K_1(x)$ by \eqref{hK:der}, so 
\[
(\d/\d x)\sqrt{x}K_0(x)=[K_0(x)-2xK_1(x)]/(2\sqrt{x}), 
\]
which is positive for all small $x>0$ by \eqref{K00} and \eqref{K10}. Given \eqref{MN:proof21}, now we can use the dominated convergence theorem for stochastic integrals \cite[Theorem~32, p.176]{Protter-2}
and \cite[Proposition~4.5 (1$\cc$)]{C:SDBG1-2} to pass the limit under the stochastic integrals. 
It follows from Lemma~\ref{lem:unif1} (1$\cc$) and Lemma~\ref{lem:unif2} that \eqref{MN:proof2} holds.

Finally, note that the u.c.p. convergences in \eqref{MN:proof1} and \eqref{MN:proof2} hold even if we drop the stopping by $T_{\delta_0,\delta_1}$ since $T_{\delta_0,\delta_1}\to\infty$ as $\delta_1\nearrow\infty$ and $\delta_0\searrow 0$ by  the sample path continuity of $\{N^{\bbeta,\bs w,\bi}_0(t)\}$ and 
$\{\mathring{N}^{\bbeta,\bs w,\bi}_0(t)\}$ and Proposition~\ref{prop:BMntc}. Applying these extensions of \eqref{MN:proof1} and \eqref{MN:proof2} to \eqref{MN:proof} proves \eqref{Ilim:2wi}.
\end{proof}

\subsection{Limiting formulas: quadratic-variation terms}\label{sec:logsumC}
The limits of the remaining terms in \eqref{eq:ItoKi} and~\eqref{eq:ItoKw} may be less obvious. 

\begin{prop}\label{prop:lim3}
Let $\bs w=\{w_\bj\}_{\bj\in \mc E_N}\in \R_+^{\mc E_N}$ with $w_\bi> 0$. Under $\P^{\bi}_{z_0}$ for any $z_0\in\CNwni$, the following limits hold with probability one: for all $t>0$, 
\begin{align}
&- I_{\ref{eq:ItoKi}}(4)-I_{\ref{eq:ItoKi}}(5)+I_{\ref{eq:ItoKw}}(4)+I_{\ref{eq:ItoKw}}(5)+I_{\ref{eq:ItoKw}}(6)+I_{\ref{eq:ItoKw}}(7)+I_{\ref{eq:ItoKw}}(8)+I_{\ref{eq:ItoKw}}(9)\notag\\
&\quad \xrightarrow[\bs \vep\searrow\bs 0]{} -\frac{1}{2}\la N^{\bbeta,\bs w,\bi}_0,N^{\bbeta,\bs w,\bi}_0\ra_t .\label{Ilim:4wi}
\end{align}
\end{prop}

To motivate the proof of Proposition~\ref{prop:lim3} and clarify the direction, we first consider the following lemma. 
Note that by Lemma~\ref{lem:bullet}, the implication in \eqref{def:Bj-2}, and \cite[Proposition~4.5 (1$\cc$)]{C:SDBG1-2}, the integrals $\I_{\ref{match:1}}$, $\II_{\ref{match:1}}$ and $\III_{\ref{match:1}}$ defined below are absolutely convergent.

\begin{lem}\label{lem:MNQV}
Let $\bs w\in \R_+^{\mc E_N}$ with $w_\bi> 0$. Under $\P^{\bi}_{z_0}$ for any $z_0\in\CNwni$, the following limits hold with probability one: for all $t>0$, 
\begin{align}
\la N^{\bbeta,\bs w,\bi}_0,N^{\bbeta,\bs w,\bi}_0\ra_t=
\la \mathring{N}^{\bbeta,\bs w,\bi}_0,\mathring{N}^{\bbeta,\bs w,\bi}_0\ra_t+\I_{\ref{match:1}}+\II_{\ref{match:1}}+\III_{\ref{match:1}},\label{match:1}
\end{align}
where 
\begin{align*}
\I_{\ref{match:1}}&\;\defeq\,  \sum_{\bj\in \mc E_N\setminus\{\bi\}}w_\bj^2\int_0^t \frac{\hK^{\bbeta,\bj}_1(s)^2}{|Z^\bj_s|^2 K_0^{\bbeta,\bs w}(s)^2}\d s,\\
\II_{\ref{match:1}}&\;\defeq\,   \sum_{\stackrel{\scriptstyle \bj,\bk\in \mathcal E_N\setminus\{\bi\}}{\bj\neq \bk}}w_\bj w_\bk\int_0^{t}\frac{\hK^{\bbeta,\bj}_1(s)\hK^{\bbeta,\bk}_1(s)}{|Z^\bj_s||Z^\bk_s|K_0^{\bbeta,\bs w}(s)^2}\d \la B^{\bj},B^{\bk}\ra_s,\\
\III_{\ref{match:1}}&\;\defeq\,  -\sum_{\bj\in \mc E_N\setminus\{ \bi\}}w_\bi w_\bj\int_0^t \frac{\hK^{\bbeta,\bj}_1(s)\hK_1^{\bbeta,\bi}(s)
\sum_{\bk\in \mc E_N\setminus\{\bi\}}w_\bk K^{\bbeta,\bk}_0(s)
}{|Z^\bj_s|^2 K_0^{\bbeta,\bs w}(s)^2\cdot w_\bi K_0^{\bbeta,\bi}(s)}\Re\left(\frac{Z^\bj_s}{Z^\bi_s}\right)\sigma(\bj)\cdot \sigma(\bi)\d s.
\end{align*}
\end{lem}
\begin{proof}
We consider the following computation, where
the second equality below applies \eqref{def:Nwi} and \eqref{def:Nwiring}:
\begin{align}
&\quad\;\la N^{\bbeta,\bs w,\bi}_0,N^{\bbeta,\bs w,\bi}_0\ra_t-\la \mathring{N}^{\bbeta,\bs w,\bi}_0,\mathring{N}^{\bbeta,\bs w,\bi}_0\ra_t\notag\\
&=\la N^{\bbeta,\bs w,\bi}_0-\mathring{N}^{\bbeta,\bs w,\bi}_0,N^{\bbeta,\bs w,\bi}_0-\mathring{N}^{\bbeta,\bs w,\bi}_0\ra_t+2\la N^{\bbeta,\bs w,\bi}_0-\mathring{N}^{\bbeta,\bs w,\bi}_0, \mathring{N}^{\bbeta,\bs w,\bi}_0\ra_t\notag\\
&=\sum_{\bj\in \mc E_N\setminus\{\bi\}}w_\bj^2\int_0^t \frac{\hK^{\bbeta,\bj}_1(s)^2}{|Z^\bj_s|^2 K_0^{\bbeta,\bs w}(s)^2}\d s+\sum_{\stackrel{\scriptstyle \bj,\bk\in \mathcal E_N\setminus\{\bi\}}{\bj\neq \bk}}w_\bj w_\bk\int_0^{t}\frac{\hK^{\bbeta,\bj}_1(s)\hK^{\bbeta,\bk}_1(s)}{|Z^\bj_s||Z^\bk_s|K_0^{\bbeta,\bs w}(s)^2}\d \la B^{\bj},B^{\bk}\ra_s\notag\\
&\quad -2\sum_{\bj\in\mc E_N\setminus\{\bi\}}w_\bj\int_0^t \frac{\hK^{\bbeta,\bj}_1(s)\hK_1^{\bbeta,\bi}(s)\sum_{\bk\in \mc E_N\setminus\{\bi\}}w_\bk K^{\bbeta,\bk}_0(s)}{|Z^\bj_s||Z^\bi_s| K_0^{\bbeta,\bs w}(s)^2K_0^{\bbeta,\bi}(s)}\d \la B^\bj,B^\bi\ra_s\notag\\
&=\sum_{\bj\in \mc E_N\setminus\{\bi\}}w_\bj^2\int_0^t \frac{\hK^{\bbeta,\bj}_1(s)^2}{|Z^\bj_s|^2 K_0^{\bbeta,\bs w}(s)^2}\d s +\sum_{\stackrel{\scriptstyle \bj,\bk\in \mathcal E_N\setminus\{\bi\}}{\bj\neq \bk}}w_\bj w_\bk\int_0^{t}\frac{\hK^{\bbeta,\bj}_1(s)\hK^{\bbeta,\bk}_1(s)}{|Z^\bj_s||Z^\bk_s|K_0^{\bbeta,\bs w}(s)^2}\d \la B^{\bj},B^{\bk}\ra_s\notag\\
&\quad -\sum_{\bj\in \mc E_N\setminus\{ \bi\}}w_\bi w_\bj\int_0^t \frac{\hK^{\bbeta,\bj}_1(s)\hK_1^{\bbeta,\bi}(s)\sum_{\bk\in \mc E_N\setminus\{\bi\}}w_\bk K^{\bbeta,\bk}_0(s)}{|Z^\bj_s|^2 K_0^{\bbeta,\bs w}(s)^2\cdot w_\bi K_0^{\bbeta,\bi}(s)}\Re\left(\frac{Z^\bj_s}{Z^\bi_s}\right)\sigma(\bj)\cdot \sigma(\bi)\d s,\notag
\end{align}
where the last equality changes the third term on its left-hand side by 
using \eqref{def:Bj-2}. The last equality proves \eqref{match:1}.
\end{proof}

To use \eqref{match:1}, we also note the following \emph{formal} decomposition suggested by \eqref{def:Nwiring}: 
\begin{align}
&\quad\;\la \mathring{N}^{\bbeta,\bs w,\bi}_0,\mathring{N}^{\bbeta,\bs w,\bi}_0\ra_t\notag\\
&=w_\bi^2\int_0^{t}\frac{\hK^{\bbeta,\bi}_1(s)^2}{|Z^\bi_s|^2}\biggl(\frac{1}{w_\bi K_0^{\bbeta,\bi}(s)}-\frac{1}{K_0^{\bbeta,\bs w}(s)}\biggr)^2\d s\label{MQV:special}\\
&=w_\bi^2\int_0^{t}\frac{\hK^{\bbeta,\bi}_1(s)^2}{|Z^\bi_s|^2}\biggl(\frac{1}{w_\bi^2 K_0^{\bbeta,\bi}(s)^2}-\frac{2}{w_\bi K_0^{\bbeta,\bi}(s)K^{\bbeta,\bs w}_0(s)}+\frac{1}{K_0^{\bbeta,\bs w}(s)^2}\biggr)\d s\notag\\
&=\I_{\ref{match:2}}+\II_{\ref{match:2}}+\III_{\ref{match:2}},\label{match:2}
\end{align}
where
\begin{align*}
\I_{\ref{match:2}}&\;\defeq\,  \int_0^{t}\frac{\hK^{\bbeta,\bi}_1(s)^2}{|Z^\bi_s|^2}\biggl(\frac{1}{K_0^{\bbeta,\bi}(s)^2}\biggr)\d s,\quad \;
\II_{\ref{match:2}}\;\defeq\,  -w_\bi\int_0^{t}\frac{\hK^{\bbeta,\bi}_1(s)^2}{|Z^\bi_s|^2}\biggl(\frac{2}{K_0^{\bbeta,\bi}(s)K^{\bbeta,\bs w}_0(s)}\biggr)\d s,\\
\III_{\ref{match:2}}&\;\defeq\, w_\bi^2\int_0^{t}\frac{\hK^{\bbeta,\bi}_1(s)^2}{|Z^\bi_s|^2}\biggl(\frac{1}{K_0^{\bbeta,\bs w}(s)^2}\biggr)\d s.
\end{align*}
Here, we regard \eqref{match:2} formal since the three integrals $\I_{\ref{match:2}}$, $\II_{\ref{match:2}}$ and $\III_{\ref{match:2}} $ are not within the scope of \cite[Proposition~4.5 (1$\cc$)]{C:SDBG1-2}. 
We are unable to guarantee their absolute convergence.

The following proof of Proposition~\ref{prop:lim3} is presented in the form of justifying the following
\emph{non-rigorous} approximations of two different types as $\bs \vep\searrow \bs 0$:
\begin{align}
&
\begin{cases}
I_{\ref{eq:ItoKw}}(4)\displaystyle \approx-\frac{\I_{\ref{match:1}}}{2}, 
\quad I_{\ref{eq:ItoKw}}(6)\approx-\frac{\II_{\ref{match:1}}}{2},\\
\vspace{-.2cm}\\
 I_{\ref{eq:ItoKw}}(7)+I_{\ref{eq:ItoKw}}(8)\displaystyle \approx-\frac{\III_{\ref{match:1}}}{2},\label{app:1}
\end{cases}\\
&\begin{cases}
-I_{\ref{eq:ItoKi}}(4)-I_{\ref{eq:ItoKi}}(5)\displaystyle \approx -\frac{\I_{\ref{match:2}}}{2},\quad 
I_{\ref{eq:ItoKw}}(9)\approx-\frac{\II_{\ref{match:2}}}{2},\\
\vspace{-.2cm}\\
 I_{\ref{eq:ItoKw}}(5)\displaystyle \approx-\frac{\III_{\ref{match:2}}}{2} .\label{app:2}
\end{cases}
\end{align}
These approximations  are motivated by Lemma~\ref{lem:unif2} and
 a comparison of the left-hand side of \eqref{Ilim:4wi}  on the one hand and 
 Lemma~\ref{lem:MNQV} and the formal equation \eqref{match:2} on the other hand. 
In more detail, the approximation of $ I_{\ref{eq:ItoKw}}(7)+I_{\ref{eq:ItoKw}}(8)$ in \eqref{app:1} considers the following algebra: by using \eqref{def:Bj-2}, the sum of the formal limits of 
$I_{\ref{eq:ItoKw}}(7)$ and $I_{\ref{eq:ItoKw}}(8)$ is
\begin{align*}
& \quad -\sum_{\bj\in \mc E_N\setminus\{ \bi\}}w_\bj w_\bi(2\beta_\bj )(2\beta_\bi)\int_0^t \frac{G_{1}(R)}{2R}\frac{G_{1}(R')}{2R'}\Bigg|_{\stackrel{\scriptstyle R=2\beta_\bj R^\bj_s}{R'=2\beta_\bi R^\bi_s}}\\
&\quad \times \frac{4|Z^\bj_s| |Z^\bi_s| }{H_0^{\bs w}(\bs R_s)^2}\frac{\sigma(\bj)\cdot\sigma(\bi)}{2}\frac{|Z^\bi_s|}{|Z^\bj_s|}\Re\left(\frac{Z^\bj_s}{Z^\bi_s}\right)\d s\\
&\quad+ \sum_{\bj\in \mc E_N\setminus\{\bi\}}w_\bj(2\beta_\bj)\int_0^t \frac{G_{1}(R)}{2R}\Bigg|_{R=2\beta_\bj R^\bj_s}\times \frac{\sigma(\bj)\cdot \sigma(\bi)}{H_0^{\bs w}(\bs R_s)}\left(
 \Re\left(\frac{Z^\bj_s}{Z^\bi_s}\right)\frac{ \K^{\bbeta,\bi}_{1}(s)}{ K^{\bbeta,\bi}_0(s)}\right)\d s \\
 &=-\frac{1}{2}\sum_{\bj\in \mc E_N\setminus\{\bi\}}w_\bj w_\bi \int_0^t \frac{\widehat{K}^{\bbeta,\bj}_1(s)\widehat{K}^{\bbeta,\bi}_1(s)}{|Z^\bj_s|^2K_0^{\bbeta,\bs w}(s)^2}\Re\left(\frac{Z^\bj_s}{Z^\bi_s}\right)\sigma(\bj)\cdot \sigma(\bi)\d s\\
 &\quad +\frac{1}{2}\sum_{\bj\in \mc E_N\setminus\{\bi\}}w_\bj w_\bi \int_0^t \frac{\widehat{K}^{\bbeta,\bj}_1(s)\widehat{K}^{\bbeta,\bi}_1(s)}{|Z^\bj_s|^2K^{\bbeta,\bs w}_0(s)\cdot w_\bi K_0^{\bbeta,\bi}(s)}\Re\left(\frac{Z^\bj_s}{Z^\bi_s}\right)\sigma(\bj)\cdot \sigma(\bi)\d s =-\frac{\III_{\ref{match:1}}}{2}.
\end{align*} 

\begin{proof}[Proof of Proposition~\ref{prop:lim3}]
We begin by proving the simpler limits in \eqref{app:1}. Write
\begin{align*}
I_{\ref{eq:ItoKw}}(4)&=-\frac{1}{2}\sum_{\bj\in \mc E_N\setminus\{ \bi\}}w_\bj^2\int_0^t \frac{ G_1(2\beta_\bj(\vep_\bj+R_\bj))^2R^\bj_s}{(\vep_\bj+R^\bj_s)^2H^{\bs w}_0(\bs \vep+\bs R_s)^2}\d s,\\
I_{\ref{eq:ItoKw}}(6)&=-\frac{1}{2}\sum_{\stackrel{\scriptstyle \bj,\bk\in \mathcal E_N\setminus\{\bi\}}{\bj\neq \bk}}w_\bj w_\bk \int_0^t \frac{ G_1(2\beta_\bj(\vep_\bj+R^\bj_s))G_1(2\beta_\bk(\vep_\bk+R^\bk_s)|Z^\bj_s||Z^\bk_s|}{(\vep_\bj+R^\bj_s)(\vep_\bk+R^\bk_s) H_0^{\bs w} (\bs \vep+\bs R_s)^2}\d \la B^\bj,B^\bk\ra_s,\\
I_{\ref{eq:ItoKw}}(7)+I_{\ref{eq:ItoKw}}(8)&=-\sum_{\bj\in \mc E_N\setminus\{ \bi\}}w_\bj w_\bi\int_0^t \frac{ G_1(2\beta_\bj(\vep_\bj+R^\bj_s))G_1(2\beta_\bi(\vep_\bi+R^\bi_s))|Z^\bj_s| |Z^\bi_s|}{(\vep_\bj+R^\bj_s)(\vep_\bi+R^\bi_s)H_0^{\bs w}(\bs \vep+\bs R_s)^2}\\
&\quad \quad \times\Re\left(\frac{Z^\bj_s}{Z^\bi_s}\right)\frac{|Z^\bi_s|}{|Z^\bj_s|}\left(\frac{\sigma(\bj)\cdot\sigma(\bi)}{2}\right)\d s\\
&\quad \;+\sum_{\bj\in \mc E_N\setminus\{ \bi\}}w_\bj\int_0^t\frac{G_1(2\beta_\bj(\vep_\bj+R^\bj_s))\hK^{\bbeta,\bi}_1(s)}{(\vep_\bj+R^\bj_s)H^{\bs w}_0(\bs\vep+\bs R_s)K^{\bbeta,\bi}_0(s)} \Re\left(\frac{Z^\bj_s}{Z^\bi_s}\right) \left(\frac{\sigma(\bj)\cdot\sigma(\bi)}{2}\right)\d s\\
&=\sum_{\bj\in \mc E_N\setminus\{ \bi\}}w_\bj w_\bi\int_0^t\frac{G_1(2\beta_\bj(\vep_\bj+R^\bj_s))}{(\vep_\bj+R^\bj_s)H^{\bs w}_0(\bs\vep+\bs R_s)^2\cdot w_\bi K^{\bbeta,\bi}_0(s)}\\
&\quad \times\left[-G_1(2\beta_\bi(\vep_\bi+R^\bi_s))\left(\frac{R^\bi_s}{\vep_\bi+R^\bi_s}\right)w_\bi K_0^{\bbeta,\bi}(s)+
\widehat{K}^{\bbeta,\bi}_1(s)H_0^{\bs w}(\bs \vep+\bs R_s)
\right]\\
&\quad\times \Re\left(\frac{Z^\bj_s}{Z^\bi_s}\right) \left(\frac{\sigma(\bj)\cdot\sigma(\bi)}{2}\right)\d s.
\end{align*} 
Here, to get the next to the last equality, we write out $\d \la B^\bj,B^\bi\ra_s$ in $I_{\ref{eq:ItoKw}}(7)$ by using \eqref{def:Bj-2}. To pass the limits of  $I_{\ref{eq:ItoKw}}(4)$, $I_{\ref{eq:ItoKw}}(6)$ and $I_{\ref{eq:ItoKw}}(7)+I_{\ref{eq:ItoKw}}(8)$, note that the dominated convergence theorem applies by the above integral representations, Lemma~\ref{lem:bullet} and \cite[Proposition~4.5 (1$\cc$)]{C:SDBG1-2}. Hence, we obtain the following rigorous justifications of the three approximations in \eqref{app:1}:
\begin{align}
\begin{split}
\lim_{\bs \vep\searrow \bs 0}I_{\ref{eq:ItoKw}}(4)&=-\frac{\I_{\ref{match:1}}}{2},\quad \lim_{\bs \vep\searrow \bs 0}
I_{\ref{eq:ItoKw}}(6)=-\frac{\II_{\ref{match:1}}}{2},\\
\lim_{\bs \vep\searrow \bs 0}I_{\ref{eq:ItoKw}}(7)+I_{\ref{eq:ItoKw}}(8)&=-\frac{\III_{\ref{match:1}}}{2}.\label{app:1-1}
\end{split}
\end{align}

Next, we realize \eqref{app:2} by considering \eqref{match:2} and the sums of both sides of the approximate equalities. Specifically, the goal is to prove that 
\begin{align}\label{app:2-1}
\lim_{\bs \vep\searrow \bs 0}-I_{\ref{eq:ItoKi}}(4)-I_{\ref{eq:ItoKi}}(5)+I_{\ref{eq:ItoKw}}(5)+I_{\ref{eq:ItoKw}}(9)=-\frac{1}{2}\la \mathring{N}^{\bbeta,\bs w,\bi}_0,\mathring{N}^{\bbeta,\bs w,\bi}_0\ra_t.
\end{align}
We begin by simplifying $-I_{\ref{eq:ItoKi}}(4)-I_{\ref{eq:ItoKi}}(5)$, $I_{\ref{eq:ItoKw}}(5)$ and $I_{\ref{eq:ItoKw}}(9)$:
\begin{align}
\begin{split}
-I_{\ref{eq:ItoKi}}(4)-I_{\ref{eq:ItoKi}}(5)&=-\frac{1}{2}\int_0^t \frac{G_1(2\beta_\bi(\vep_\bi+R^\bi_s))}{G_0(2\beta_\bi(\vep_\bi+R^\bi_s))(\vep_\bi+R^\bi_s)}\\
&\quad \times\left(\frac{-R^\bi_s G_1(2\beta_\bi(\vep_\bi+R^\bi_s))}{(\vep_\bi+R^\bi_s)G_0(2\beta_\bi(\vep_\bi+R^\bi_s))}+\frac{2\hK^{\bbeta,\bi}_1(s)}{K^{\bbeta,\bi}_0(s)}\right)\d s,\label{6''}
\end{split}\\
I_{\ref{eq:ItoKw}}(5)&=-\frac{1}{2}\int_0^t \frac{w_\bi^2 R^\bi_s G_1(2\beta_\bi(\vep_\bi+R^\bi_s))^2}{(\vep_\bi+R^\bi_s)^2H_0^{\bs w}(\bs \vep+\bs R_s)^2}\d s,\label{6''-1}\\
I_{\ref{eq:ItoKw}}(9)&=\frac{1}{2} \int_0^t \frac{w_\bi\cdot 2G_1(2\beta_\bi(\vep_\bi+R^\bi_s))}{(\vep_\bi+R^\bi_s)H_0^{\bs w}(\bs \vep+\bs R_s)}\left(\frac{\hK^{\bbeta,\bi}_1(s)}{K_0^{\bbeta,\bi}(s)}\right)\d s.\label{6''-2}
\end{align}
To bound the sum of the right-hand sides, we do some further simplifications by taking its limit. By Lemma~\ref{lem:lim3-1} proven below, we can replace $R^\bi_s/(R^\bi_s+\vep)$ in the integrands of the right-hand sides of \eqref{6''} and \eqref{6''-1} by $1$ in the limit. Specifically, 
\begin{align*}
&\quad\;\lim_{\bs \vep\searrow \bs 0}- I_{\ref{eq:ItoKi}}(4)-I_{\ref{eq:ItoKi}}(5)+I_{\ref{eq:ItoKw}}(5)+I_{\ref{eq:ItoKw}}(9)\\
&=\lim_{\bs \vep \searrow \bs 0}-\frac{1}{2}\int_0^t \frac{G_1(2\beta_\bi(\vep_\bi+R^\bi_s))}{G_0(2\beta_\bi(\vep_\bi+R^\bi_s))(\vep_\bi+R^\bi_s)}\left(\frac{- G_1(2\beta_\bi(\vep_\bi+R^\bi_s))}{G_0(2\beta_\bi(\vep_\bi+R^\bi_s))}+\frac{2\hK^{\bbeta,\bi}_1(s)}{K^{\bbeta,\bi}_0(s)}\right)\d s\\
&\quad -\frac{1}{2}\int_0^t \frac{w_\bi^2 G_1(2\beta_\bi(\vep_\bi+R^\bi_s))^2}{(\vep_\bi+R^\bi_s)H_0^{\bs w}(\bs \vep+\bs R_s)^2}\d s+I_{\ref{eq:ItoKw}}(9). 
\end{align*}
Moreover, the last equality can be further improved as follows:
\begin{align*}
&\quad\;\lim_{\bs \vep\searrow \bs 0} -I_{\ref{eq:ItoKi}}(4)-I_{\ref{eq:ItoKi}}(5)+I_{\ref{eq:ItoKw}}(5)+I_{\ref{eq:ItoKw}}(9)\\
&=\lim_{\bs \vep\searrow \bs 0}-\frac{1}{2}\int_0^t \frac{G_1(2\beta_\bi(\vep_\bi+R^\bi_s))}{G_0(2\beta_\bi(\vep_\bi+R^\bi_s))(\vep_\bi+R^\bi_s)}\left(\frac{\hK^{\bbeta,\bi}_1(s)}{K^{\bbeta,\bi}_0(s)}\right)\d s -\frac{1}{2}\int_0^t \frac{w_\bi^2 G_1(2\beta_\bi(\vep_\bi+R^\bi_s))^2}{(\vep_\bi+R^\bi_s)H_0^{\bs w}(\bs \vep+\bs R_s)^2}\d s\\
&\quad +I_{\ref{eq:ItoKw}}(9)\\
&=\lim_{\bs \vep\searrow \bs 0}-\frac{1}{2}\int_0^t \frac{1}{(\vep_\bi+R^\bi_s)}\left(\frac{\hK^{\bbeta,\bi}_1(s)}{K^{\bbeta,\bi}_0(s)}\right)^2\d s -\frac{1}{2}\int_0^t \frac{w_\bi^2 G_1(2\beta_\bi(\vep_\bi+R^\bi_s))^2}{(\vep_\bi+R^\bi_s)H_0^{\bs w}(\bs \vep+\bs R_s)^2}\d s+I_{\ref{eq:ItoKw}}(9),
\end{align*}
where, by the occupation times formula \eqref{oct}, the next to the last equality and the last equality follow from \eqref{oct:approx} and \eqref{oct:approx1} proven below, respectively. By the last equality and \eqref{6''-2}, we get
\begin{align*}
&\quad\;\lim_{\bs \vep\searrow \bs 0}-I_{\ref{eq:ItoKi}}(4)-I_{\ref{eq:ItoKi}}(5)+I_{\ref{eq:ItoKw}}(5)+I_{\ref{eq:ItoKw}}(9)\\
&=\lim_{\bs \vep\searrow \bs 0}-\frac{1}{2}\int_0^t \frac{1}{(\vep_\bi+R^\bi_s)}\left(\frac{\hK_1^{\bbeta,\bi}(s)}{K_0^{\bbeta,\bi}(s)}-\frac{w_\bi G_1(2\beta_\bi(\vep_\bi+R^\bi_s))}{H_0^{\bs w}(\bs \vep+\bs R_s)}\right)^2\d s.
\end{align*}
The dominated convergence theorem applies to evaluate the last limit by \cite[Proposition~4.5 (1$\cc$)]{C:SDBG1-2} and Lemma~\ref{lem:bullet}, and this limit equals $-\frac{1}{2}\la \mathring{N}^{\bbeta,\bs w,\bi}_0,\mathring{N}^{\bbeta,\bs w,\bi}_0\ra_t$ by  \eqref{MQV:special}. We have proved the required identity in \eqref{app:2-1}. 

The proof is complete upon combining \eqref{app:1-1} and~\eqref{app:2-1}.  
\end{proof}

\begin{lem}\label{lem:lim3-1}
It holds that 
\begin{align*}
&\lim_{\bs \vep\searrow \bs 0}\int_0^t G_1(2\beta_\bi(\vep_\bi+R^\bi_s))^2\cdot \frac{\vep_\bi}{(\vep_\bi+R^\bi_s)^2G_0(2\beta_\bi(\vep_\bi+R^\bi_s))^2}\d s=2 L^\bi_t,\\
&\lim_{\bs \vep\searrow \bs 0}\int_0^t \frac{w_\bi^2G_1(2\beta_\bi(\vep_\bi+R^\bi_s))^2 G_0(2\beta_\bi(\vep_\bi+R^\bi_s))^2}{H_0^{\bs w}(\bs \vep+\bs R_s)^2}\cdot \frac{\vep_\bi}{(\vep+R^\bi_s)^2G_0(2\beta_\bi(\vep_\bi+R^\bi_s))^2}\d s= 2L^\bi_t.
\end{align*}
\end{lem}
\begin{proof}
To obtain the first limit, note that the uniform convergence of $G_1(2\beta_\bi(\vep_\bi+R^\bi_s))^2$ to $1$, $0\leq s\leq t$, holds by 
Lemma~\ref{lem:unif1} (1$\cc$) and Lemma~\ref{lem:unif2}. This convergence, \eqref{eq:LTnorm1-2}, Lemma~\ref{lem:unif1} (2$\cc$) and the fact that $G_1(0)=\widehat{K}_1(0)=1$ give the required limit. 
The second limit follows similarly, now using the following uniform convergence instead:
\begin{align*}
\frac{w_\bi^2G_1(2\beta_\bi(\vep_\bi+R^\bi_s))^2 G_0(2\beta_\bi(\vep_\bi+R^\bi_s))^2}{H_0^{\bs w}(\bs \vep+\bs R_s)^2}
=\Biggl(\frac{w_\bi G_1(2\beta_\bi(\vep_\bi+R^\bi_s))}{w_\bi+\sum_{\bj\in \mc E_N\setminus\{\bi\}}w_\bj \frac{G_0(2\beta_\bj(\vep_\bj+R^\bj_s))}{G_0(2\beta_\bi(\vep_\bi+R^\bi_s)) }}\Biggr)^2\to 1,\;0\leq s\leq t,
\end{align*}
which holds by Lemma~\ref{lem:unif1} (1$\cc$) and  Lemma~\ref{lem:unif2}.
\end{proof}

\begin{lem}\label{lem:lim3-2}
It holds that $(\hK_1/K_0)(\cdot)$ is an increasing function, and for any $M>0$, 
\begin{align}
0&\leq \int_0^{M}\frac{1}{K_0(\sqrt{\vep +y^2})(\vep+y^2)}\left(\frac{\hK_1}{K_0}(\sqrt{\vep+y^2})-\frac{\hK_1}{K_0}(y)\right)yK_0(y)^2\d y\xrightarrow[\vep\searrow 0]{}0,\label{oct:approx}\\
0&\leq \int_0^{M}\frac{1}{K_0(y)(\vep+y^2)}\left(\frac{\hK_1}{K_0}(\sqrt{\vep+y^2})-\frac{\hK_1}{K_0}(y)\right)yK_0(y)^2\d y\xrightarrow[\vep\searrow 0]{}0.
\label{oct:approx1}
\end{align}
\end{lem}
\begin{proof}
First, the property that $(\hK_1/K_0)(\cdot)$ is increasing has been obtained in \cite[Proposition~4.12 (4$\cc$)]{C:BES-2} by taking $\alpha=0$ there. Second, since $K_0$ is decreasing,  it suffices to show the convergence to zero in \eqref{oct:approx}. In this case, we can just consier the case that $M=1/2$ since $K_0(\cdot)$ is bounded away from zero on compacts in $(0,\infty)$.

To prove the convergence to zero in \eqref{oct:approx} for $M=1/2$, write
\begin{align}
&\quad\;\int_0^{1/2}\frac{1}{K_0(\sqrt{\vep +y^2})(\vep+y^2)}\left(\frac{\hK_1}{K_0}(\sqrt{\vep+y^2})-\frac{\hK_1}{K_0}(y)\right)yK_0(y)^2\d y\notag\\
&=\int_0^{1/2} \frac{1}{\vep+y^2}\frac{[\hK_1(\sqrt{\vep+y^2})K_0(y)-\hK_1(y)K_0(\sqrt{\vep+y^2})]}{K_0(\sqrt{\vep+y^2})^{2}K_0(y)}yK_0(y)^2\d y\notag\\
\begin{split}\label{oct:cal1}
&=\int_0^{1/2} \frac{1}{\vep+y^2}\frac{[\hK_1(\sqrt{\vep+y^2})-\hK_1(y)]K_0(y)}{K_0(\sqrt{\vep+y^2})^{2}K_0(y)}yK_0(y)^2\d y\\
&\quad+\int_0^{1/2} \frac{1}{\vep+y^2}\frac{\hK_1(y)[K_0(y)-K_0(\sqrt{\vep+y^2})]}{K_0(\sqrt{\vep+y^2})^2K_0(y)}yK_0(y)^2\d y.
\end{split}
\end{align}
In the remaining of this proof, we consider the two integrals on the right-hand side for $0<\vep\leq 1/2$. The goal is to show that both of them tend to zero as $\vep\to 0$.

To show that the first integral on the right-hand side of \eqref{oct:cal1} tends to zero, we note that 
\begin{align*}
\hK_1(\sqrt{\vep+y^2})-\hK_1(y)=\int_y^{\sqrt{\vep+y^2}}\frac{\d}{\d x}\hK_1(x)\d x=\int_y^{\sqrt{\vep+y^2}}-xK_0(x)\d x
\end{align*}
since $(\d/\d x)\hK_1(x)=-xK_0(x)$ \cite[(5.7.9)]{Lebedev-2}.
Hence, by the asymptotic representation \eqref{K00} of $K_0(x)$ as $x\searrow 0$, we have, for all $0\leq y\leq 1/2$ and $0<\vep\leq 1/2$, 
\begin{align*}
|\hK_1(\sqrt{\vep+y^2})-\hK_1(y)|
\less \int_y^{\sqrt{\vep+y^2}}-x\log x\d x=\frac{-x^2(\log x^2-1)}{4}\Big|_{x=y}^{\sqrt{\vep+y^2}},
\end{align*}
since $\int -x\log x\d x=\frac{-x^2(\log x^2-1)}{4}+C$ for $x>0$. Applying the last equality and the asymptotic representation of $K_0(x)$ as $x\to 0$ to the first integral on the right-hand side of \eqref{oct:cal1}, we get
\begin{align}
&\quad\;\left|\int_0^{1/2} \frac{1}{\vep+y^2}\frac{[\hK_1(\sqrt{\vep+y^2})-\hK_1(y)]K_0(y)}{K_0(\sqrt{\vep+y^2})^2K_0(y)}yK_0(y)^2\d y\right|\notag\\
&\less \int_0^{1/2} \frac{y}{\vep+y^2}\frac{-(\vep+y^2)[\log(\vep+y^2)-1
]+y^2(\log y^2-1)}{\log^2(\vep+y^2)}(\log^{2}y)\d y\notag\\
&= \int_0^{1/2} \frac{y}{\vep+y^2}\frac{\vep+y^2}{\log^2(\vep+y^2)}(\log^{2}y)\d y\notag\\
&\quad+\int_0^{1/2} \frac{y}{\vep+y^2}\frac{-(\vep+y^2)\log(\vep+y^2)
+y^2(\log y^2-1)}{\log^2(\vep+y^2)}(\log^{2}y)\d y\notag\\
&\xrightarrow[\vep\searrow 0]{}\int_0^{1/2}y\d y+\int_0^{1/2}-y\d y=0.\label{oct:cal2}
\end{align}
In more detail, the limit in \eqref{oct:cal2} holds by using the following domination bounds where the right-hand sides are all $\in L^1([0,1/2],\d y)$, for all $0<\vep\leq 1/2$:
\begin{align*}
&  \left|\frac{y}{\vep+y^2}\frac{\vep+y^2}{\log^2(\vep+y^2)}(\log^{2}y)\right|\leq
 \frac{y(\log^2y)}{\log^2(1/2+y^2)},\\
&\left| \frac{y}{\vep+y^2}\frac{-(\vep+y^2)\log(\vep+y^2)
+y^2(\log y^2-1)}{\log^2(\vep+y^2)}(\log^{2}y)\right|\less y\log^2y+\frac{y|\log^2y-1|}{\log^2(1/2+y^2)}(\log^2y).
\end{align*}
Note that these bounds have used the fact that $x\mapsto \log^2x$ is decreasing on $(0,1]$. 

Next, we show that the second integral on the right-hand side of \eqref{oct:cal1} tends to zero. We use the identity 
\[
K_0(x)=\frac{1}{2}\int_0^\infty \e^{-s}\int_{x^2/(4s)}^\infty \e^{-u}u^{-1}\d u\d s.
\]
\cite[(4.77)]{C:BES-2}. The foregoing display implies that 
\[
0\leq K_0(y)-K_0(\sqrt{\vep+y^2})\leq \int_0^\infty \e^{-s}\int^{(\vep+y^2)/(4s)}_{y^2/(4s)} u^{-1}\d u\d s\less \log \left(\frac{\vep+y^2}{y^2}\right),\quad \forall\;y>0. 
\]
By the foregoing display  and the fact that $\hK_1(y)\less 1$ for all $0<y\leq 1/2$ using the asymptotic representation \eqref{K10} of $K_1(x)$ as $x\to 0$,
the second integral on the right-hand side of \eqref{oct:cal1} thus satisfies the next two inequalities:
\begin{align}
0&\leq \int_0^{1/2} \frac{1}{\vep+y^2}\frac{\hK_1(y)[K_0(y)-K_0(\sqrt{\vep+y^2})]}{K_0(\sqrt{\vep+y^2})^2K_0(y)}yK_0(y)^2\d y\notag\\
&\less \int_0^{1/2} \frac{y}{\vep+y^2}\frac{-\log y}{\log^2(\vep+y^2)}\log \left(1+\frac{\vep}{y^2}\right)\d y\notag\\
\begin{split}
&= \int_0^{0.5/\sqrt{\vep}} \frac{ y}{1+ y^2}\frac{-\log \sqrt{\vep}-\log y}{[\log \vep +\log (1+ y^2)]^2}\log \left(1+\frac{1}{ y^2}\right)\d y
.\label{oct:cal2-0}
\end{split}
\end{align}
We consider the last integral via
$\int_0^{0.5/\sqrt{\vep}}=\int_0^{0.5}+\int_{0.5}^{0.5/\sqrt{\vep}}$. In the first case,
\begin{align}
 &\quad \int_0^{0.5} \frac{ y}{1+ y^2}\frac{-\log \sqrt{\vep}-\log y}{[\log \vep +\log (1+ y^2)]^2}\log \left(1+\frac{1}{ y^2}\right)\d y\notag\\
 &\less \frac{1}{|\log \vep|} \int_0^{0.5} y\log \left(1+\frac{1}{ y^2}\right)\d y+\frac{1}{\log^2 \vep }\int_0^{0.5} y|\log y|\log \left(1+\frac{1}{ y^2}\right)\d y\xrightarrow[\vep\searrow 0]{}0,\label{oct:cal2-1}
\end{align}
where the inequality uses $y/(1+y^2)\leq y$, and the convergence holds since we have the convergence
$\int_{0}^{0.5}y|\log y|\log (1+\frac{1}{y^2})\d y<\infty$.
Also, since $y/(1+y^2)\less 1/y$ for all $y\geq 1$ and $\log (1+x)\less x$ for all $x>0$, the other part of the integral in \eqref{oct:cal2-0} satisfies 
\begin{align}
&\quad\;\int_{0.5}^{0.5/\sqrt{\vep}} \frac{ y}{1+ y^2}\frac{-\log \sqrt{\vep}-\log y}{[\log \vep +\log (1+ y^2)]^2}\log \left(1+\frac{1}{ y^2}\right)\d y\notag\\
&\less \int_{0.5}^{0.5/\sqrt{\vep}} \frac{1}{y}\frac{-\log \sqrt{\vep}-\log y}{[\log \vep +\log (1+y^2)]^2}\log \left(1+\frac{1}{ y^2}\right)\d y\notag\\
&\less \int_{0.5}^{0.5/\sqrt{\vep}} \frac{-\log \sqrt{\vep}-\log y}{y(\log^2\vep)}\cdot \frac{1}{y^2}\d y\xrightarrow[\vep\to 0]{}0,\label{oct:cal2-2}
\end{align}
where the second inequality uses $-\log \vep-\log (1+ y^2)\geq -\log (\vep+0.5^2)$ for all $0.5<y<0.5/\sqrt{\vep}$.
By \eqref{oct:cal2-0}, \eqref{oct:cal2-1} and \eqref{oct:cal2-2}, 
we have proved that 
\begin{align}
\int_0^{1/2} \frac{1}{\vep+y^2}\frac{\hK_1(y)[K_0(y)-K_0(\sqrt{\vep+y^2})]}{K_0(\sqrt{\vep+y^2})^2K_0(y)}yK_0(y)^2\d y\xrightarrow[\vep\searrow 0]{}0.\label{oct:cal3}
\end{align}

Finally, we combine \eqref{oct:cal2} and \eqref{oct:cal3}. This shows that the right-hand side of \eqref{oct:cal1} tends to zero as $\vep\searrow 0$. We have proved \eqref{oct:approx} for $M=1/2$. The proof is complete.
\end{proof}

\subsection{End of the proof of Proposition~\ref{prop:logsum}}
We show \eqref{RN} first. 
Recall that Proposition~\ref{prop:logsum} assumes $z_0\in \CNwni$. By putting together 
\eqref{Ito:scheme}, Lemmas~\ref{lem:ItoKi}--\ref{lem:ItoKw},
and Propositions~\ref{prop:lim1}, \ref{prop:lim2} (2$\cc$) and \ref{prop:lim3}, for any fixed $T>0$, 
we can find a sequence $\bs \vep_n=\{\vep_{n,\bj}\}_{\bj\in \mc E_N}\searrow 0$ such that with $\P^\bi_{z_0}$-probability one, the following equalities hold 
for all $0\leq t\leq T$:
\begin{align}
\log \frac{K^{\bbeta,\bs w}_0 (t)}{w_\bi K^{\bbeta,\bi}_0 (t)}
&=\lim_{n\to\infty}F^{\bs w}_{\bs \vep_n}(\bs R_0)-F^\bi_{\vep_{n,\bi}}(R^\bi_0)+\sum_{i=1}^9I_{\ref{eq:ItoKw}}(i)_t-\sum_{i=1}^5I_{\ref{eq:ItoKi}}(i)_t\notag\\
&=\log \frac{K^{\bbeta,\bs w}_0 (0)}{ w_\bi K^{\bbeta,\bi}_0 (0)}+A^{\bbeta,\bs w,\bi}_0(t) +N_0^{\bbeta,\bs w,\bi}(t)-\frac{1}{2}\la N^{\bbeta,\bs w,\bi}_0,N^{\bbeta,\bs w,\bi}_0\ra_t,\notag
\end{align}
where $I_{\ref{eq:ItoKw}}(i)_t$ and $I_{\ref{eq:ItoKi}}(i)_t$ on the right-hand side of the second equality are understood to use $\bs \vep=\bs \vep_n$. Since $0<T<\infty$ is arbitrary, we obtain the property that with $\P^\bi_{z_0}$-probability one, \eqref{RN} holds for all $t\geq 0$. 

The remaining properties in Proposition~\ref{prop:logsum} can be obtained as follows. First, $\{N^{\bbeta,\bs w,\bi}_0(t)\}$ is a well-defined continuous local martingale by Proposition~\ref{prop:lim2} (1$\cc$). To verify the properties in Proposition~\ref{prop:logsum} (4$\cc$), note that those of $\{\mathring{N}^{\bbeta,\bs w,\bi}_0(t\wedge T^{\bs w\setminus\{w_\bi\}}_{\eta})\}$ have been obtained in Proposition~\ref{prop:lim2} (1$\cc$). The $L^2$-martingale property of 
$\{N^{\bbeta,\bs w,\bi}_0(t\wedge T^{\bs w\setminus\{w_\bi\}}_{\eta})\}$ then follows from the definitions of $\Twi_\eta$ and $\{N^{\bbeta,\bs w,\bi}_0(t)\}$ in \eqref{def:Twi} and \eqref{def:Nwiring} since for any $\bj\in \mc E_N$ with $w_\bj>0$,
\[
\frac{w_\bj\hK^{\bbeta,\bj}_1(s)}{|Z^\bj_s|K^{\bbeta,\bs w}_0(s)}\leq \frac{w_\bj\hK^{\bbeta,\bj}_1(s)}{|Z^\bj_s| w_\bj K_0^{\bbeta,\bj}(s)}\leq C(\beta_\bj,\eta),\quad \forall\;s\leq \Twi_\eta,
\]
where the second inequality uses the asymptotic representations \eqref{K0infty} and \eqref{K1infty} of $K_0$ and $K_1$ as $x\to\infty$. The proof of Proposition~\ref{prop:logsum} is complete.

\end{document}